\documentclass[11pt]{amsart} 
\usepackage{amscd,amssymb,amsxtra}
\usepackage{mathrsfs}  
\DeclareSymbolFontAlphabet{\mathrsfs}{rsfs}
\usepackage[mathscr]{eucal} 
\usepackage{mathabx}\DeclareMathSymbol{\emptyset}{\mathord}{symbols}{59}
                    \DeclareMathSymbol{\boxtimes}{\mathbin}{AMSa}{"02}
\usepackage{comment}
\usepackage{color}
\usepackage{enumitem} 
\usepackage{marginnote}

\makeatletter
\let\@secnumfont\bfseries
\def\section{\@startsection{section}{1}%
  \z@{4\linespacing\@plus\linespacing}{\linespacing}%
  {\bfseries\centering}}
\def\introsection{\@startsection{section}{1}%
  \z@{3\linespacing\@plus\linespacing}{\linespacing}%
  {\bfseries\centering}}
\def\subsection{\@startsection{subsection}{2}%
   \z@{1.25\linespacing\@plus.7\linespacing}{.5\linespacing}%
   {\normalfont\bfseries}}
\def\subsectionsinline{\def\subsection{\@startsection{subsection}{2}%
  \z@{1\linespacing\@plus.7\linespacing}{-.5em}%
  {\normalfont\bfseries}}}

\makeatother

\theoremstyle{definition}
\newtheorem{definition}[equation]{Definition}

\newtheorem*{definition*}{Definition}
\newtheorem*{example*}{Example}
\newtheorem*{problem*}{\color{blue}Problem}
\newtheorem*{exercise*}{Exercise}
\newtheorem*{question*}{\color{blue}Question}
\newtheorem*{project*}{\color{blue}Project}
\newtheorem*{construction*}{Construction}

\theoremstyle{remark}

\newtheorem{remark}[equation]{Remark}

\newtheorem*{note*}{Note}
\newtheorem*{notation*}{Notation}
\newtheorem*{remark*}{Remark}
\newtheorem*{data*}{Data}

\theoremstyle{plain}
\newtheorem{theorem}[equation]{Theorem}
\newtheorem{corollary}[equation]{Corollary}
\newtheorem{lemma}[equation]{Lemma}
\newtheorem{proposition}[equation]{Proposition}

\newtheorem*{theorem*}{Theorem}
\newtheorem*{corollary*}{Corollary}
\newtheorem*{lemma*}{Lemma}
\newtheorem*{proposition*}{Proposition}
\newtheorem*{conjecture*}{Conjecture}
\newtheorem*{claim*}{Claim}
\newtheorem*{proposal*}{Proposal}
\newtheorem*{conclusion*}{Conclusion}
\newtheorem*{hypothesis*}{Hypothesis}
\newtheorem*{assumption*}{Assumption}

\newenvironment{proof*}[1][\proofname]{
  \begin{proof}[#1]}{  \renewcommand\qedsymbol{\relax}
\end{proof}}

\numberwithin{equation}{section}

\definecolor{refkey}{rgb}{0,.6,.4}

\renewcommand{\:}{\colon}
\renewcommand{\AA}{{\mathbb A}}

\newcommand{\CC}{{\mathbb C}}

\DeclareMathOperator{\End}{End}

\DeclareMathOperator{\Hom}{Hom}
\DeclareMathOperator{\id}{id}

\DeclareMathOperator{\pt}{pt}
\newcommand{\QQ}{{\mathbb Q}}

\newcommand{\RR}{{\mathbb R}}

\newcommand{\ZZ}{{\mathbb Z}}

\newcommand{\chiup}{\raise.5ex\hbox{$\chi$}}
\newcommand{\cir}{S^1}

\newcommand{\inv}{^{-1}}
\DeclareRobustCommand{\mstrut}{^{\vphantom{1*\prime y\vee M}}}

\newcommand{\temsquare}{\raise3.5pt\hbox{\boxed{ }}}

\newcommand{\zmod}[1]{\ZZ/#1\ZZ}

\newcommand{\zt}{\zmod2}

\DeclareMathOperator{\SO}{SO}

\let\O\relax
\DeclareMathOperator{\O}{O}

\DeclareMathOperator{\GL}{GL}

\newcounter{alphthm}
\newcounter{alphthmp}
\theoremstyle{plain}
\newtheorem{mainthm}[alphthm]{Theorem}
\newtheorem{mainthmp}[alphthmp]{Theorem}

\usepackage[all,2cell]{xy}\renewcommand{\cir}{\ensuremath{S^1}}
\UseAllTwocells
\usepackage{graphicx}\usepackage{epsf}
\usepackage[colorlinks,backref=page,citecolor=refkey]{hyperref}
\let\O\relax\DeclareMathOperator{\O}{O}

\definecolor{refkey}{rgb}{0,.8,.2}\definecolor{labelkey}{rgb}{1,0,0} 

\DeclareMathOperator{\Alg}{Alg}
\DeclareMathOperator{\Bord}{Bord}
\DeclareMathOperator{\Cat}{Cat}
\DeclareMathOperator{\FC}{Fus}
\DeclareMathOperator{\Gr}{Gr}
\DeclareMathOperator{\SC}{FSCat}
\DeclareMathOperator{\Sti}{St}
\DeclareMathOperator{\Vect}{Vect}

\DeclareMathOperator{\sVect}{sVect}
\newcommand{\Alez}{\AA^k_{\le0}}
\newcommand{\BC}{\Alg_2(\cc)}
\newcommand{\BGLo}{B\!\GL_{n-1}\!\RR}
\newcommand{\BGL}{B\!\GL_n\!\RR}
\newcommand{\Bd}[1]{B^\delta _{-#1}}
\newcommand{\Bmo}{B\mstrut _{-1}}
\newcommand{\Bnp}{\Bord_{n,\partial}}
\newcommand{\Bo}[1]{B^1_{-#1}}
\newcommand{\Bz}[1]{B^0_{-#1}}
\newcommand{\DF}{|F|^2}
\newcommand{\Dd}{\Delta ^{\!\vee}}
\newcommand{\FL}{{}_{\Phi}\hspace{-.2pt}\Phi }
\newcommand{\FR}{\Phi _\Phi }
\newcommand{\FSb}{F(S^1_b)}
\newcommand{\FS}{F(S^0)}
\newcommand{\Fd}{F^d}
\newcommand{\Fs}{\DF}
\newcommand{\Ft}{F_{\langle 1,2,3  \rangle}}
\newcommand{\Lx}[1]{L_{x,-(#1)}}
\newcommand{\OCfd}{\Omega \sCfd}
\newcommand{\OC}{\Omega \sC}
\newcommand{\PL}{{}_{\Psi}\!\Psi }
\newcommand{\PR}{\Psi _\Psi }
\newcommand{\Pd}{\Phi^{\!\vee}}
\newcommand{\QZ}{\QQ/\ZZ}
\newcommand{\Sb}{S^1_b}

\newcommand{\Spl}{S^{+}}
\newcommand{\St}{S^{\otimes}}
\newcommand{\TC}{\Alg_1(\cc)}
\newcommand{\TPd}{\TP^d}

\newcommand{\TP}{T_\Phi }
\newcommand{\TSb}{T(S^1_b)}
\newcommand{\TSn}{T(S^1_n)}

\newcommand{\Tt}{T_{\langle 1,2,3  \rangle}}
\newcommand{\Vc}{\Vect\mstrut _{\CC}}
\newcommand{\Xd}[1]{X^\delta _{-#1}}
\newcommand{\Xo}[1]{X^1_{-#1}}
\newcommand{\Xz}[1]{X^0_{-#1}}
\newcommand{\ac}{\alpha _c}
\newcommand{\ao}{\alpha _1}
\newcommand{\bC}{\mathbb{C}}
\newcommand{\bF}{\overline{F}}

\newcommand{\bR}{\beta ^R}

\newcommand{\bX}{\partial X}

\newcommand{\bao}{\ao^{\SO}}
\newcommand{\bd}{\beta ^d}
\newcommand{\bfot}{\Bord^{\textnormal{fr}}_{\langle 1,2,3  \rangle}}
\newcommand{\bftb}{\Bord^{\textnormal{fr}}_{3,\partial }}
\newcommand{\bft}{\Bord^{\textnormal{fr}}_3}
\newcommand{\bnon}{\bno\to\bn}
\newcommand{\bno}{\Bord_{n-1}}
\newcommand{\bn}{\Bord_n}

\newcommand{\bp}{\beta (+)}
\newcommand{\btbf}{\Bord_{2,\partial }^{\textnormal{fr}}}
\newcommand{\btf}{\Bord_2^{\textnormal{fr}}}
\newcommand{\bt}{\Bord_2}

\newcommand{\cM}{\mathcal{M}}

\newcommand{\cc}{\Cat\mstrut _{\CC}}
\newcommand{\ck}{\Cat\mstrut _{\Bbbk}}

\newcommand{\ciF}{\;\circ\mstrut_{\Phi }\;}
\newcommand{\cet}{\check{\eta }}
\newcommand{\cx}{\check{x}}
\newcommand{\dual}{^\vee}
\newcommand{\eFd}{\eF^d}
\newcommand{\eF}{\widetilde{F}}
\newcommand{\eqcat}{\simeq_{\textnormal{cat}}}
\newcommand{\fpm}{b_{\pm}}

\newcommand{\gdj}{\gamma ^\delta _{-j}}
\newcommand{\ghT}{(g\circ h)^\textnormal{T}}
\newcommand{\mBo}{\mB\mstrut _{-1}}
\newcommand{\mB}{\mathring{B}}
\newcommand{\mM}{\mathring{M}}
\newcommand{\mX}{\mathring{X}}
\newcommand{\mYb}{P}
\newcommand{\mY}{\mathring{Y}}
\newcommand{\mb}{\mathring{b}}

\newcommand{\mg}{\mathring{\gamma }}
\newcommand{\mo}{^{\textnormal{mo}}}
\newcommand{\mps}{\mathring{\psi }^\delta _{-j}}
\newcommand{\mx}{\mathring{x}}
\newcommand{\m}{(-)}
\newcommand{\op}{^{\textnormal{op}}}
\newcommand{\p}{(+)}
\newcommand{\rev}{^{\textnormal{rev}}}
\newcommand{\sCfd}{\sC^{\textnormal{fd}}}
\newcommand{\sC}{\mathscr{C}}
\newcommand{\sH}{\mathscr{H}}
\newcommand{\sVc}{\sVect\mstrut _{\CC}}

\newcommand{\sX}{\mathscr{X}}
\newcommand{\tF}{\tau\mstrut _{\le2}F}
\newcommand{\tGr}{\widetilde{\Gr}_{n-1}}
\newcommand{\tT}{\widetilde{T}}

\newcommand{\triv}[1]{\underline{\RR^{#1}}}
\newcommand{\tstar}{topological${}^*$}
\newcommand{\uEnd}{\underline{\End}}
\newcommand{\uE}[2]{\uEnd^{#1}(#2)}
\newcommand{\uHom}{\underline{\Hom}}
\newcommand{\uH}[3]{\uHom^{#1}(#2,#3)}
\newcommand{\uR}{\underline{\RR}}
\newcommand{\vep}{\varepsilon}

\newcommand{\squig}{{\scriptstyle\sim\mkern-3.9mu}}

\newcommand{\rsquigend}{{\scriptstyle\rule{.1ex}{0ex}\rhd}}
\newcounter{sqindex}

\begin{document}

\abovedisplayskip18pt plus4.5pt minus9pt
\belowdisplayskip \abovedisplayskip
\abovedisplayshortskip0pt plus4.5pt
\belowdisplayshortskip10.5pt plus4.5pt minus6pt
\baselineskip=15 truept
\marginparwidth=55pt

\makeatletter
\renewcommand{\tocsection}[3]{%
  \indentlabel{\@ifempty{#2}{\hskip1.5em}{\ignorespaces#1 #2.\;\;}}#3}
\renewcommand{\tocsubsection}[3]{%
  \indentlabel{\@ifempty{#2}{\hskip 2.5em}{\hskip 2.5em\ignorespaces#1%
    #2.\;\;}}#3} 
\renewcommand{\tocsubsubsection}[3]{%
  \indentlabel{\@ifempty{#2}{\hskip 5.5em}{\hskip 5.5em\ignorespaces#1%
    #2.\;\;}}#3} 
\renewcommand\@captionfont{\normalfont\small}
\makeatother

\setcounter{tocdepth}{2}




 \title[Gapped boundary theories in 3d]{Gapped Boundary Theories in Three Dimensions} 
 \author[D. S. Freed]{Daniel S.~Freed}
 \thanks{This material is based upon work supported by the National Science
Foundation under Grant Number DMS-1611957.  Parts of this work were performed
at the Aspen Center for Physics, which is supported by National Science
Foundation grant PHY-1607611.  We also thank the IAS/Park City Mathematics
Institute and the Mathematical Sciences Research Institute, which is
supported by National Science Foundation Grant 1440140.}
 \address{Department of Mathematics \\ University of Texas \\ Austin, TX
78712} 
 \email{dafr@math.utexas.edu}

 \author[C. Teleman]{Constantin Teleman} 
 \address{Department of Mathematics \\ University of California \\ 970 Evans
Hall \#3840 \\ Berkeley, CA 94720-3840}  
 \email{teleman@math.berkeley.edu}

 \date{August 9, 2021}
 \begin{abstract} 
 We prove a theorem in 3-dimensional topological field theory: a
Reshetikhin-Turaev theory admits a nonzero boundary theory iff it is a
Turaev-Viro theory.  The proof immediately implies a characterization of
fusion categories in terms of dualizability.  Our results rely on a (special
case of) the cobordism hypothesis with singularities.  The main theorem
applies to physics, where it implies an obstruction to a gapped 3-dimensional
quantum system admitting a gapped boundary theory.  Appendices on bordism
multicategories, on 2-dualizable categories, and on internal duals may be of
independent interest.
 \end{abstract}
\maketitle


{\small\tableofcontents}


Quantum mechanical theories bifurcate into gapped and gapless theories.  The
classical notion of a local boundary condition for a partial differential
equation has a quantum analog---a boundary theory.  There is a basic
question: Does a gapped quantum system~$S$ admit a gapped boundary theory?
We formulate and prove a mathematical theorem (Theorem~\ref{thm:5}
in~\S\ref{subsec:1.2}) which addresses this question for a large class of
(2+1)-dimensional systems.  The route from a gapped quantum system to the
theorem goes via a low energy effective extended topological field theory,
whose existence we merely assume.  The absence of a gapped boundary theory
implies the presence of gapless edge modes---conduction on the boundary---an
important feature of quantum Hall systems, for example.
 
Let $F\:\bft\to\sC$ be a 3-dimensional topological field theory: a
homomorphism from a bordism multicategory of framed manifolds to a symmetric
monoidal 3-category~$\sC$.  We impose hypotheses on~$\sC,F$ to model
Reshetikhin-Turaev theories~\cite{RT1,RT2,T}, whose key invariant is a
modular fusion category~$C$, the value of~$F$ on the bounding framed circle.
Theorem~\ref{thm:5} asserts that if $F$~admits a nonzero boundary theory,
then $C$~is the Drinfeld center of a fusion\footnote{We do not assume that a
fusion category has a simple unit: our `fusion' is \cite{EGNO}'s
`multifusion'.} category~$\Phi $.  With extra assumptions on the
codomain~$\sC$, we conclude (Theorem~\ref{thm:57} in~\S\ref{subsec:1.2}) that
the entire theory $F$~is isomorphic to the Turaev-Viro theory~$\TP$ based
on~$\Phi $.  Conversely, given a fusion category~$\Phi $, the Turaev-Viro
theory~$\TP$ has a nonzero boundary theory built from the (regular) left
$\Phi $-module~$\Phi $.  The class of~$C$ in the Witt group~\cite{DMNO},
which is nonzero when $C$~is not the Drinfeld center of a fusion category, is
{almost} a complete obstruction to the existence of a nonzero boundary
theory; see~Remark~\ref{thm:60}(4).

A corollary of our proof is a characterization of fusion categories
(Theorem~\ref{thm:46} in~\S\ref{subsec:1.3}).  Let $\cc$ be the symmetric
monoidal 2-category of finitely cocomplete $\CC$-linear categories and right
exact $\CC$-linear functors, and let $\TC$ be the Morita 3-category of tensor
categories.  (In this paper `tensor category' means `algebra object
in~$\cc$'.)  Then $\Psi \in \TC$ is a fusion category iff $\Psi $~is
3-dualizable and the regular left $\Psi $-module is 2-dualizable.  The
forward direction (`only if') is proved in~\cite{DSS}.
 
A key feature of our approach is the use of fully extended field theories.
Naive analogs of Theorem~\ref{thm:5} fail in the traditional context of
$(1,2,3)$-theories; see Remark~\ref{thm:58}.

Here is a brief outline of the paper.  Section~\ref{sec:1} contains
background and the statements of the main theorems.  In section~\ref{sec:2}
we discuss preliminaries about bordism multicategories and algebras in
multicategories.  The proofs of the main theorems are deferred
to~\S\ref{sec:7}.  The application to gapped quantum systems is the subject
of~\S\ref{sec:3}.  We provide several appendices with background material of
interest independent of our main theorems.  Appendix~\ref{sec:4} contains a
detailed definition of objects and morphisms in bordism multicategories.  In
Appendix~\ref{sec:8} we prove a characterization of finite semisimple abelian
categories in the 2-category~$\cc$.  Appendix~\ref{sec:5} proves a criterion
for internal duals in tensor categories, developed in a more general context.
Appendix~\ref{sec:6} proves the complete reducibility of fusion categories,
which is implicit in~\cite{EGNO}: a fusion category is Morita equivalent to a
direct sum of fusion categories with simple unit.
 
Papers related to the problems considered here include~\cite{KK,KS,FSV,Le,KZ}.

We warmly thank Pavel Etingof, Theo Johnson-Freyd, Victor Ostrik, Emily
Riehl, Claudia Scheimbauer, Chris Schommer-Pries, Sam Raskin, Noah Snyder,
Will Stewart, and Kevin Walker for discussions related to this work.  We also
thank the referees for their careful readings and detailed suggestions; they
catalyzed many improvements.

   \section{Mathematical background and statement of main theorem}\label{sec:1}

  \subsection{RT theories and TV theories}\label{subsec:1.1}

In the late 1980s Witten~\cite{W1} and Reshetikhin-Turaev~\cite{RT1,RT2}
introduced new invariants of closed 3-manifolds and generalizations of the
Jones invariants of knots.  Witten's starting point is the classical
Chern-Simons invariant, which he feeds into the physicists' path integral,
whereas Reshetikhin-Turaev begin with an intricate algebraic structure: a
quantum group.  Later, quantum groups were replaced by modular fusion
categories~\cite{T}, which were originally introduced\footnote{The version
in~\cite{MS} uses a central charge in~$\QQ/24\ZZ$, whereas the version
standard in mathematics, which we use, only has a central charge
in~$\QQ/8\ZZ$.} in the context of 2-dimensional conformal field
theory~\cite{MS}.  These disparate approaches are reconciled in
\emph{extended} topological field theory~\cite{F1}.  In modern terms~\cite{L1}
this extended field theory is a symmetric monoidal functor
  \begin{equation}\label{eq:1}
     F\mstrut _{\langle 1,2,3 \rangle}\:\bfot\longrightarrow \cc 
  \end{equation}
with domain the 2-category\footnote{In this paper we use discrete categories:
for example $\cc$~is a $(2,2)$-category (as opposed to a more general
$(\infty ,2)$-category).  Some of our exposition in this section applies to
$(\infty,n) $-categories though we just write `$n$-categories'.} of
3-framed\footnote{A \emph{3-framing} of a 3-manifold is a global parallelism,
a trivialization of its tangent bundle.  For a manifold~$M$ of
dimension~$k<3$ it is a trivialization of the inflated tangent bundle
$\triv{3-k}\oplus TM\to M$.} 1-, 2-, and 3-dimensional bordisms; the codomain
is a certain 2-category of complex linear categories; see
Definition~\ref{thm:30}.  The value of~$\Ft$ on the bounding 3-framed
circle~$S^1_b$ is the modular fusion category that defines the theory.  We
call~\eqref{eq:1} a \emph{Reshetikhin-Turaev} (RT) theory.

  \begin{remark}[]\label{thm:1}
 The RT theories in the original references factor through the bordism
2-category of manifolds equipped with a $(w_1,p_1)$-structure~\cite{BHMV},
that is, an orientation and a trivialization of the first Pontrjagin
class~$p_1$.  There is a unique isomorphism class of $(w_1,p_1)$-structures
on a circle, so no distinction between bounding and nonbounding circles.
``Spin Chern-Simons theories'' require a trivialization of the second
Stiefel-Whitney class~$w_2$ as well.  For those theories the codomain should
include $\zt$-gradings; see Remark~\ref{thm:45}.
  \end{remark}

A \emph{fully extended} topological field theory has domain
$\bft=\Bord^{\textnormal{fr}}_{\langle 0,1,2,3 \rangle}$, the 3-category of
3-framed bordisms of dimension~$\le3$.  There is no canonical codomain for
these theories, so for now we posit an arbitrary symmetric monoidal
3-category~$\sC$.  The \emph{cobordism hypothesis}---conjectured by
Baez-Dolan~\cite{BD}, proved by Hopkins-Lurie in 2~dimensions and by
Lurie~\cite{L1} in all dimensions; see also~\cite{AF}---asserts that a fully
extended theory
  \begin{equation}\label{eq:2}
     F\:\bft\longrightarrow \sC 
  \end{equation}
is determined by its value~$F(+)$ on a positively oriented 3-framed point.
Furthermore, a \emph{3-dualizable} object of~$\sC$ determines a unique
theory~\eqref{eq:2}, up to a contractible space of choices.  A symmetric
monoidal 3-category~$\sC$ has a \emph{fully dualizable part}~$\sCfd\subset
\sC$ whose objects are 3-dualizable and whose morphisms have all adjoints.
We say \emph{$\sC$~has duals} if $\sC=\sCfd$.  A functor~\eqref{eq:2} factors
through~$\sCfd$.  Given a general RT theory~\eqref{eq:1} it is still an open
problem to construct~$\sC$ and an extension~\eqref{eq:2}, or even better to
construct a single~$\sC$ which works for all RT theories.  (However,
see~\cite{He} for a special case in the framework of bicommutant categories.)

There is a subclass of RT theories, the \emph{Turaev-Viro (TV)
theories}~\cite{TV}, which are fully extended.  Let $\FC$ be the symmetric
monoidal 3-category whose objects are \emph{fusion} categories; see
Definition~\ref{thm:49}.   We remark that throughout this paper we take~
$\CC$ as the ground field.

  \begin{theorem}[Douglas-Schommer-Pries-Snyder~\cite{DSS}]\label{thm:51}
 $\FC$~has duals, i.e., $\FC=\FC^{\textnormal{fd}}$.
  \end{theorem}

\noindent
 In particular, a fusion category~$\Phi $ is 3-dualizable in~$\FC$.  The
cobordism hypothesis implies that there is a fully extended topological field
theory
  \begin{equation}\label{eq:3}
     T_\Phi \:\bft\longrightarrow \FC,
  \end{equation}
unique up to equivalence, whose value on a chosen framed point~$+$ is~$T_\Phi
\p=\Phi $.  A TV theory is a fully extended theory with codomain~ $\FC$; its
truncation to $\bfot$ is an RT theory.  Examples include 3-dimensional gauge
theory for a finite group~$G$, in which case~$\Phi $ is the fusion category
of finite rank complex vector bundles over~$G$ with convolution product;
there is also a version twisted by a cocycle for a class in~$H^3(G;\QZ)$, as
in~\cite{DW}.  Special toral Chern-Simons theories are also TV theories.  

  \begin{remark}[]\label{thm:66}
 The original state sum construction~\cite{TV} is quite different from the
construction with the cobordism hypothesis, but nevertheless we use
`Turaev-Viro theory' to identify this class of topological field theories.
Also, the construction in~\cite{TV} is for unoriented manifolds and a
$(2,3)$-theory, whereas we use framed manifolds and a fully extended
$(0,1,2,3)$-theory.   
  \end{remark}

  \subsection{Definitions and terminology}\label{subsec:1.5}

The definitions and terminology for abelian categories are standard;
see~\cite[\S1]{EGNO} for example.  For tensor categories there is tremendous
variation in the literature, so we spell out our usage here.  The term
`modular fusion category' is standard; see \cite[\S8.13]{EGNO}, for example.

  \begin{definition}[]\label{thm:30}
 \ 

 \begin{enumerate}[label=\textnormal{(\roman*)}]

 \item $\cc$~is the symmetric monoidal 2-category defined as
follows.\footnote{The symbol `Rex' is sometimes used in place of `$\cc$'; it
emphasizes the right exactness of 1-morphisms.}  Its objects are finitely
cocomplete $\CC$-linear categories.  1-morphisms in~$\cc$ are right exact
$\CC$-linear functors---functors that preserve finite colimits---and
2-morphisms are natural transformations.  The symmetric monoidal structure is
the Deligne-Kelly tensor product~$\boxtimes$; see~\cite{K,D,Fr}.

 \item A \emph{tensor category} is an algebra object in~$\cc$.

 \item $\TC$~is the symmetric monoidal 3-category defined as follows.  Its
objects are tensor categories.  A 1-morphism $M\:A\to B$ is an object~$M\in
\cc$ equipped with the structure of a $(B,A)$-bimodule category.  A
2-morphism $M'\to M$ is a 1-morphism in~$\cc$ which respects the bimodule
structure.  A 3-morphism is a natural transformation of functors.
 
 \item $\BC$ is the symmetric monoidal 4-category defined as follows.  Its
objects are braided tensor categories, a 1-morphism $M\:A\to B$ is an
object~$M\in \TC$ equipped with the compatible structure of a
$(B,A)$-bimodule category, etc.  (See~\cite[Definition-Proposition~1.2]{BJS}.)

 \end{enumerate}
  \end{definition}

  \begin{remark}[]\label{thm:48}
 \ 
 
 \begin{enumerate}

 \item We do not assume that a tensor category has internal duals (rigidity).
See Appendix~\ref{sec:5} for a discussion of internal duals in tensor
categories.  Also, we do not assume that the tensor unit of a tensor category
is a simple object.

 \item The algebraic theory of tensor categories is exposited in the text
Etingof-Gelaki-Nikshych-Ostrik~\cite{EGNO} (where rigidity and simple unit
are included in the definition of `tensor category').  The theory of Morita
higher categories, such as~$\TC$ and~$\BC$, is developed by
Haugseng~\cite{H}, Johnson-Freyd-Scheimbauer~\cite{JS},
Gwilliam-Scheimbauer~\cite{GS}, among others.
Douglas-Schommer-Pries-Snyder~\cite{DSS} define a 3-category of tensor
categories which is a subcategory of~$\TC$; in particular, they assume
rigidity.  Tensor categories and braided tensor categories in an infinite
setting are explored in~\cite{BJS}, and in an $\infty $-setting in~\cite{L2}.
 
 \item The symmetric monoidal structure on $\TC$ is Deligne-Kelly tensor
product, as in~$\cc$.  Composition of 1-morphisms in~$\TC$ is the
\emph{relative} Deligne-Kelly tensor product: tensor product of module
categories over a tensor category.  Its existence is discussed in
\cite[Remark~3.2.1]{BZBJ} and \cite[Example~8.10]{JS}.

 \item For a 3-category~$\sC$ set $\OC=\End_{\sC}(1)$, the endomorphism
2-category of the tensor unit object.  There is a canonical identification
  \begin{equation}\label{eq:41}
     \Omega \TC= \cc. 
  \end{equation}

 \item A modular fusion category is invertible as an object of~$\BC$;
see~\cite{BJSS}.

 \end{enumerate}
  \end{remark}

The 3-category~$\FC$ of fusion categories is introduced in~\cite{DSS}.

  \begin{definition}[]\label{thm:49}
 \  

 \begin{enumerate}[label=\textnormal{(\roman*)}]

 \item $\SC$~is the full subcategory of~$\cc$ whose objects are finite
semisimple abelian categories.

 \item A \emph{fusion category} is a finite semisimple rigid tensor
category. 

 \item $\FC$~is the symmetric monoidal 3-category subcategory of~$\TC$ whose
objects are fusion categories and whose 1-morphisms are finite semisimple
abelian bimodule categories.

 \end{enumerate}
  \end{definition}

  \begin{remark}[]\label{thm:50}
 \ 
 \begin{enumerate}

 \item $\SC\subset \cc$ is closed under Deligne-Kelly tensor
product~\cite[\S5]{Fr}.

 \item We do not assume that a fusion category has simple unit.  (\cite{EGNO}
use `multifusion' for Definition~\ref{thm:49}(ii) and reserve `fusion' for
the case of a simple unit.)

 \item The loop category of~$\FC$ is
  \begin{equation}\label{eq:73}
     \Omega \FC= \SC .
  \end{equation}

 \end{enumerate}
  \end{remark}

A companion result to Theorem~\ref{thm:51} asserts that the symmetric
monoidal 2-category $\SC$ has duals.  We also need the following result,
which we prove in Appendix~\ref{sec:8}. 

  \begin{theorem}[]\label{thm:53}
 If~$C\in \cc$ is 2-dualizable, then $C$~is finite semisimple abelian.
  \end{theorem}

\noindent
 There are many variations of this theorem, such as~\cite{Se,Ti}; see
\cite[Appendix]{BDSV} for a survey.

  \subsection{Relations between RT and TV theories}\label{subsec:1.4}

Let 
  \begin{equation}\label{eq:71}
     T\:\bft\longrightarrow \FC
  \end{equation}
be a TV~theory.  Then the associated modular fusion category~$T(\Sb)$, the
value of~$T$ on the bounding 3-framed circle, is the \emph{Drinfeld center}
of~$T\p$.  It is the modular fusion category of the associated RT~theory,
which is the truncation
  \begin{equation}\label{eq:72}
     \Tt\:\bfot\longrightarrow \cc. 
  \end{equation}
  Also, the value~ $\TSn$ on the nonbounding 3-framed circle is the
\emph{Drinfeld cocenter}\footnote{Objects of the Drinfeld cocenter of a
fusion category~$\Phi$ are pairs~$(x,\gamma )$ in which $x$~is an object
of~$\Phi$ and the functor $\gamma \:x\otimes -\to -\otimes x^{**}$ is a
twisted half-braiding.}  of~$T\p$, a module category over~$\TSb$.  (Notice
that $\TSn$~is equivalent to~$\TSb$ if $T\p$~is spherical.)  See
\cite[\S3.2.2]{DSS} for an exposition.  For a general RT~theory~\eqref{eq:1}
there does not exist a fusion category whose Drinfeld center is the modular
fusion category~$\Ft(\Sb)$.

  \begin{remark}[]\label{thm:4}

 The double $\DF=FF\dual$ of an RT theory~$F$ is the truncation of a TV
theory, as we now explain.  ($F\dual$~is the dual theory to~$F$ in the
symmetric monoidal category of theories.) Let $B$ be a modular fusion
category.  Suppose $F\:\bft\to\sC$ is an extension of an RT
theory~\eqref{eq:1} with $\FSb=B$, where we assume the hypotheses on~$\sC$ in
Theorem~\ref{thm:5} below.  Use the cobordism hypothesis to define theories
  \begin{equation}\label{eq:4}
     F\dual,\Fs\:\bft\longrightarrow \sC 
  \end{equation}
which are characterized by
  \begin{equation}\label{eq:5}
  \begin{aligned}
  F\dual(+) &= F(+)\dual \\
  \Fs\p &= F\p\otimes F\p\dual = \FS.
  \end{aligned}
  \end{equation}
Here $F\p\dual$ is the dual object to $F\p\in \sC$.  The theory~$F\dual$ may
be defined as the composition of~$F$ with the involution $\bft\to\bft$ which
reverses the first ``arrow of time''; see~\S\ref{subsubsec:2.1.6}.  Using
this description identify the braided fusion category~$F\dual(\Sb)$
with~$B\rev$, which is the same underlying fusion category~$B$ equipped with
the inverse braiding.  It follows that $\Fs(S^1_b)\cong B\boxtimes B\rev$,
which by \cite[Proposition~8.20.12]{EGNO} is braided tensor equivalent to the
Drinfeld center of~$B$.
  \end{remark}

  \subsection{Existence of boundary theories}\label{subsec:1.2}

Let 
  \begin{equation}\label{eq:69}
     F\:\bft\longrightarrow  \sC 
  \end{equation}
be a 3-dimensional 3-framed topological field theory, as in~\eqref{eq:2}.
Lurie~\cite[Example~4.3.22]{L1} defines an extended bordism
3-category~$\bftb$ and an inclusion $\bft\to\bftb$.
(See~\S\ref{subsec:4.3} for a definition of objects in~$\bftb$.)  A
\emph{boundary theory for~$F$} is an extension
  \begin{equation}\label{eq:55}
      \eF\:\bftb\to\sC
  \end{equation}
of~$F$ to a symmetric monoidal functor.  The cobordism hypothesis with
singularities implies that $\eF$~ is determined by a 3-dualizable object
$F(+)\in \sC$ together with a 2-dualizable 1-morphism $1\to F(+)$.  To
isolate the data of the boundary theory, let
  \begin{equation}\label{eq:6}
     \tF\:\Bord^{\textnormal{fr}}_2\longrightarrow \sC 
  \end{equation}
be the truncation of~$F$ to \emph{2-framed} 0-, 1-, and 2-dimensional
bordisms.  It is a \emph{once categorified topological field theory}.  Then
the data of a {boundary theory} for~$F$ is a natural
transformation\footnote{There are also boundary theories $\tF\to 1$.  They
are extension of~$F$ to a variation of~$\bftb$.}
  \begin{equation}\label{eq:7}
     \beta \:1\longrightarrow \tF 
  \end{equation}
of symmetric monoidal functors on~$\Bord^{\textnormal{fr}}_2$.  More
precisely, $\beta $~is an \emph{oplax} natural transformation in the sense of
Johnson-Freyd and Scheimbauer~\cite{JS}.  They apply the cobordism hypothesis
with singularities~\cite[\S4.3]{L1} to prove that the data of an extended
theory~\eqref{eq:55} is equivalent to the data of the pair consisting
of~\eqref{eq:69} and~\eqref{eq:7}; see~\cite[Theorem~7.15]{JS}.  Furthermore,
that data is determined by the values~$F(+)$ and $\beta (+)\:1\to F(+)$ on a
point, which satisfy maximal dualizability constraints.  We discuss natural
transformations in pictorial terms in~\S\ref{subsubsec:2.1.7}.

Our main result is the following.  

  \begin{mainthm}[]\label{thm:5}
 Let $\sC$~be a symmetric monoidal 3-category whose fully dualizable
part~$\sC^{\textnormal{fd}}$ contains the 3-category~$\FC$ of fusion
categories as a full subcategory.  Let $ F\:\bft\to\sC$ be a 3-framed
topological field theory such that

 \begin{enumerate}[label=\textnormal{(\alph*)}]

 \item $F(S^0)$~is isomorphic in~$\sC$ to a fusion category, and

 \item $ F(S^1_b)$ is invertible as an object in the 4-category $
\Alg_2(\OC)$ of braided tensor categories.

 \end{enumerate} 
\noindent 
 Assume $F$~extends to $\eF\:\bftb\to\sC$ such that the associated boundary
theory $\beta \:1\to \tF$ is nonzero.  Then $F(S^1_b)$~is braided tensor
equivalent to the Drinfeld center of a fusion category~$ \Phi $, which may be
taken to be $\End_{\sCfd}\bigl(\beta (+)\bigr)$.
  \end{mainthm}

\noindent
 The hypothesis that~$\FC$ is a full subcategory of~$\sC^{\textnormal{fd}}$
ensures that Theorem~\ref{thm:5} applies to TV theories.  That hypothesis
implies that 
  \begin{equation}\label{eq:74}
     \OCfd= \SC; 
  \end{equation}
see~\eqref{eq:41}.  Since $\SC\subset \cc$, the $\langle 1,2,3 \rangle$~
truncation~$F_{\langle 1,2,3 \rangle}$ of~$F$ has the form~\eqref{eq:1}.
Hypotheses~(a) and~(b) express that $\Ft$~is an RT theory.  A modular fusion
category is invertible as an object in the 4-category of braided tensor
categories~\cite{S-P,BJSS}; hypothesis~(b) captures this central feature of
RT theories.  It remains to explain what it means that $\beta $~is
\emph{nonzero}.  Observe that the value of~$\beta $ on a closed 1-manifold is
an object in a finite semisimple complex linear abelian category, and the
value of~$\beta $ on a closed 2-manifold is a vector in a finite dimensional
complex vector space.  We require that $\beta $~take a nonzero value on some
nonempty closed 1-~or 2-manifold.

  \begin{remark}[]\label{thm:60}
 \ 
 \begin{enumerate}
 \item The unit in our fusion category~$\Phi $ may not be simple.  We prove
(Corollary~\ref{thm:34}) that $\Phi $~is Morita equivalent to a fusion
category~$\Phi _0$ with simple unit.  The category~$\Phi $ is canonical in
terms of~$\eF=(F,\beta )$, whereas $\Phi _0$~is only determined up to Morita
equivalence. 

 \item The relationship between $F(S^0)$ and~$\Phi $ is explained in
Lemma~\ref{thm:28} (in which $\Xi $~is a fusion category isomorphic
to~$F(S^0)$). 

 \item
 Theorem~\ref{thm:5} gives an obstruction to a nonzero topological boundary
theory: the modular fusion category~$\FSb$ must be the Drinfeld center of a
fusion category.  There is a simpler obstruction, even if we only assume
$F$~is a (2,3)-theory: the central charge~$c$---which is defined
modulo~24---must vanish.  To see this, let $Y$~be a closed 3-framed surface
of genus~$g\ge 3$.  Then $\pi _0$~ of the automorphism group of~$Y$ is a
central extension of the mapping class group by~$\ZZ$, and the central~$\ZZ$
acts\footnote{See~\cite{MR} for the case of $(w_1,p_1)$-theories (as in
Remark~\ref{thm:1}), where the generator of the appropriate centrally
extended mapping class group acts as~$e^{2\pi ic/24}$.  A framing
trivializes~$(w_1,w_2,p_1/2)$; the extra factor of~2 explains the discrepancy
between~$e^{2\pi ic/24}$ here and $e^{2\pi ic/12}$~ in the text.} on the
state space~$F(Y)$ by the character~$e^{2\pi ic/12}$.  If $\beta $~is a
nonzero boundary theory, then $\beta (Y)\in F(Y)$ is an invariant vector,
which we may suppose is nonzero for some~$g$, and so the center must act
trivially.

 \item  The modular fusion category~$F(\Sb)$ does not determine the theory~$F$
completely~\cite{BK}.  For example, the $E_8$~Chern-Simons theory at level~1
is invertible and the modular fusion category is trivial, whereas the theory
is nontrivial---its central charge is 8 mod 24---and it is not a
Turaev-Viro theory.  Therefore, by Theorem~\ref{thm:57} below, it does not
admit a nonzero boundary theory.

\item  At first glance it might seem that the hypotheses of Theorem~\ref{thm:5},
which lead to~\eqref{eq:74}, are too restrictive: we might rather have
nonsemisimple categories be possible values of~$\Ft$.  However, this is ruled
out by Theorem~\ref{thm:53} and a dimensional reduction argument, such as in
the proof of Lemma~\ref{thm:19}. 

 \item A converse to Theorem~\ref{thm:5} follows from the cobordism hypothesis.
Namely, if $\Phi $~is a fusion category, then the Turaev-Viro theory~$\TP$
defined in~\eqref{eq:3} has a canonical boundary theory built from the
regular left $\Phi $-module~$\Phi $, as in Theorem~\ref{thm:57} below. 

 \item In forthcoming work with Claudia Scheimbauer~\cite{FST}, we construct
an extension of any RT~theory which satisfies the hypotheses of
Theorem~\ref{thm:5}. 

 \end{enumerate}
 \end{remark}

 \noindent We prove Theorem~\ref{thm:5} in~\S\ref{subsec:7.1}.

Following a suggestion of Theo Johnson-Freyd, we can improve
Theorem~\ref{thm:5} by making a slightly stronger assumption on~$\sC$.  We
spell out that assumption in the following definition, which expresses the
co-completeness of $\sC$ under very special finite colimits coming from
tensor products.

  \begin{definition}[]\label{thm:69}
 Let $\sC$~be a symmetric monoidal 3-category whose fully dualizable
part~$\sC^{\textnormal{fd}}$ contains the 3-category~$\FC$ of fusion
categories as a full subcategory.  Then $\sC$ is \emph{fusion tensor
cocomplete} if the following holds for all objects~$x,y,z\in \sC$:

 \begin{enumerate}[label=\textnormal{(\alph*)}]

 \item for every triple~$(\Phi ,M,N)$ consisting of a fusion category $\Phi
$, viewed as an algebra object in $\End_\sC(1)$; a left $\Phi$-module $M$ in
$\sC(x,y)$; and a right $\Phi$-module $N$ in $\End_\sC(1)$, the relative
tensor product $N\boxtimes_\Phi M$ exists as a colimit of the bar resolution;

 \item for all~$H\in \sC(y,z)$ the natural map $(H\circ
M)\boxtimes_\Phi N \to H\circ (M\boxtimes_\Phi N)$ is an equivalence;

 \item for all~$K\in \sC(z,x)$ the natural map $M\boxtimes_\Phi (N\circ K)
\to (M\boxtimes_\Phi N)\circ K$ is an equivalence.

 \end{enumerate}
  \end{definition}

  \begin{mainthmp}[]\label{thm:57}
 In the context of Theorem~\ref{thm:5} assume in addition that $\sC$~is
fusion tensor cocomplete.  Then $F$~is isomorphic to the Turaev-Viro
theory~$\TP$, and the boundary theory is determined by the regular left
module category~$\Phi $.
  \end{mainthmp}

\noindent 
 We prove Theorem~\ref{thm:57} in \S\ref{subsec:7.2}.

  \begin{remark}[]\label{thm:68}
 A more obvious hypothesis on~$\sC$---namely that \textnormal{(i)}~for
all~$x,y\in \sC$ the 2-category $\sC(x,y)$ of 1-morphisms is finitely
cocomplete, and \textnormal{(ii)} composition of 1-morphisms in~$\sC$ is
finitely cocontinuous---is too strong for many applications.  (We thank
Sam Raskin for this observation.)
  \end{remark}

  \begin{remark}[]\label{thm:58}
 If $F$~is isomorphic to the tensor unit theory, then for any Turaev-Viro
representative~$\TP$ the fusion category~$\Phi $ is endomorphisms of a finite
semisimple abelian category~$M$.  Under the isomorphism
$\TP\xrightarrow{\;\sim \;}1$, executed via the Morita trivialization
of~$\End(M)$, the regular left $\Phi $-module goes over to the finite
semisimple abelian category~$M$.  Note that the $(1,2)$~part of the
$(0,1,2)$-theory based on~$M$ assigns a semisimple commutative algebra
to~$\Sb$.  

If from the beginning we work with $(1,2,3)$-theories~\eqref{eq:1}, then the
conclusion of Theorem~\ref{thm:57} can fail.  For example, consider a
$(1,2)$-theory with $\beta (S^1_b)=\CC[x]/(x^2)$, a nonsemisimple commutative
algebra.  This theory is not extendable to a (0,1,2)-theory with values in
$\Omega \FC=\SC$, so does not arise as in previous paragraph.
  \end{remark}

  \begin{remark}[]\label{thm:45}
 There is a generalization of RT theories~\eqref{eq:1} with codomain a
2-category of ``super'' complex linear categories.  Some developments in the
theory of these categories---which are either enriched over the
category~$\sVc$ of super vector spaces or are a module category
over~$\sVc$---especially for fusion supercategories, may be found
in~\cite{GK,BE,U}.  Our main theorems generalize to allow supercategories in
the codomain~\cite{FT}.
  \end{remark}

  \subsection{A characterization of fusion categories}\label{subsec:1.3}

En route to proving Theorem~\ref{thm:5}, we prove the following
characterization of fusion categories.

  \begin{mainthm}[]\label{thm:46}
 Let $\Psi \in \TC$ be a tensor category.  Then $\Psi $~is a fusion
category iff

 \begin{enumerate}[label=\textnormal{(\roman*)}]

 \item $\Psi $~is 3-dualizable in~$\TC$, and
 
 \item $\Psi $~as a left $\Psi $-module is 2-dualizable as a 1-morphism
in~$\TC$.

 \end{enumerate} 
  \end{mainthm}

\noindent
 The forward direction, proved in~\cite{DSS}, is stated as
Theorem~\ref{thm:51}.  We prove the converse in~\S\ref{subsec:2.2}.

  \begin{remark}[\cite{BJS}]\label{thm:47}
 Let $A=\CC[x]/(x^2)$ be the non-semisimple algebra of dual numbers.  The
tensor category~$\Psi $ of finite dimensional $A$-$A$ bimodules is Morita
equivalent to $\Vc$, so is 3-dualizable, but it is not a fusion category: for
example, it does not have internal duals.  This example illustrates that
`fusion' is not a Morita invariant notion.  The dualizability in
Theorem~\ref{thm:46}(i) is Morita invariant, whereas the regular module in
Theorem~\ref{thm:46}(ii) is not.  For example, under the Morita equivalence
which sends~$\Psi $ to~$\Vc$, the regular left module~$\PL$ is sent to the
linear category of $A$-modules, which is not the regular left module
over~$\Vc$.  We regard Theorem~\ref{thm:5} as a Morita invariant variant of
Theorem~\ref{thm:46}.
  \end{remark}

   \section{Preliminaries}\label{sec:2}

  \subsection{Bordism $n$-categories}\label{subsec:2.1}

As a preliminary, we recall features of bordism multicategories and explain
how they are encoded in the pictures we draw.  In Appendix~\ref{sec:4} we
give a formal and precise account valid in all dimensions; the heuristic
exposition here is focused on the low dimensional cases of interest.
See~\cite{BM,CS,AF} for complete constructions of the bordism multicategory.

  \begin{figure}[ht]
  \centering
  \includegraphics[scale=.65]{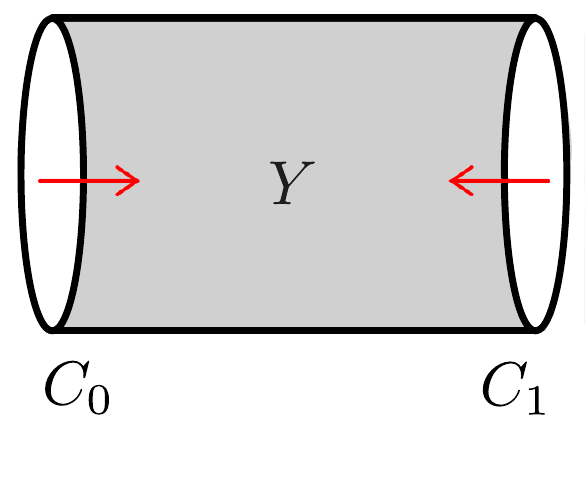}
  \vskip -.5pc
  \caption{A 2-morphism $C_0\amalg C_1\xrightarrow{\;Y\;}\emptyset ^1$ in
$\bt$}\label{fig:1}
  \end{figure}

   \subsubsection{Arrows of time}\label{subsubsec:2.1.1}
 Begin with $\bt$, the 2-category of unoriented bordisms of dimension~$\le
2$.  An endomorphism $\emptyset ^0\xrightarrow{\;C\;}\emptyset ^0$ of the
empty 0-manifold is a closed 1-manifold.  In the bordism 2-category its
tangent bundle is stabilized to the rank~2 vector bundle $\uR\oplus TC\to C$,
where $\uR\to C$ is the trivial bundle of rank~1 with its canonical
orientation.  This orientation is called an ``arrow of time''.  In a
1-morphism $C_0\amalg C_1\xrightarrow{\;Y\;}\emptyset ^1$, such as the one
depicted in Figure~\ref{fig:1}, the arrows of time distinguish incoming and
outgoing boundary components.  An object~$P$ in $\bt$ is a finite set of
points with inflated tangent bundle $\RR\oplus \RR\to P$.  Figure~\ref{fig:2}
is a 1-morphism $P_0\amalg P_1\xrightarrow{\;C\;}\emptyset ^0$.  Note that
the manifolds~$P_0,P_1$ have two ordered arrows of time; the ordering is
depicted in our figures by the number of arrowheads.  (In our conventions the
indexing is by codimension, so the single-headed arrow has index~$-1$ and the
double-headed arrow has index~ $-2$.)  The single-headed arrow of time is
constrained to be compatible with the single arrow of time of~$C$; that is,
corresponding trivial summands augmenting the tangent bundle are identified
at~$\partial C$.  The single-headed arrow of time in Figure~\ref{fig:2}
carries no information---it evokes the standard orientation of the trivial
real line bundle over~$C$---whereas the double-headed arrows of time in
Figure~\ref{fig:2} distinguish incoming and outgoing boundary components.  By
contrast, the arrows of time in Figure~\ref{fig:5} do carry information; they
depict the standard basis of~$\RR\oplus \RR$, so give meaning to the depicted
framings.

  \begin{figure}[ht]
  \centering
  \includegraphics[scale=.65]{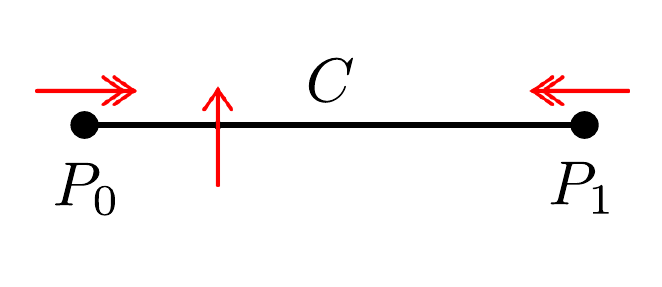}
  \vskip -.5pc
  \caption{A 1-morphism $P_0\amalg P_1\xrightarrow{\;C\;}\emptyset ^0$ in
  $\bt$}\label{fig:2}
  \end{figure}

   \subsubsection{Constancy data}\label{subsubsec:2.1.5}
 A 2-morphism in $\bt$ is a compact 2-manifold~$Y$ with corners and arrows of
time together with \emph{constancy data}~\cite[Definition~5.1]{CS}.  Namely,
if $Y=\mY_0\amalg \mY_{-1}\amalg \mY_{-2}$ is the partition into boundaries
and corners, so $\dim \mY_{-k}=2-k$, then there is a free involution
specified on~$\mY_{-2}$ with quotient~$\mYb$ and an embedding
  \begin{equation}\label{eq:42}
      [0,1]\times \mYb\hookrightarrow \overline{Y_{-1}} 
  \end{equation}
such that the arrows of time are constant along the image of $[0,1]\times
\{p\}$ for all $p\in \mYb$.  Thus Figure~\ref{fig:3}(a) is legal: the dashed
vertical edges comprise the image of the constancy embedding~\eqref{eq:42}.
The constancy gives rise to an interpretation of Figure~\ref{fig:3}(a) as a
2-morphism
  \begin{equation}\label{eq:10}
     \xymatrix@C+22pt{P_0\amalg P_1\rtwocell<5>^{C_0
     }_{C_1\amalg C_2}{^{\qquad \;\;Y}}& \emptyset ^0} ,
  \end{equation}
essentially by collapsing the dashed edges.  On the other hand,
Figure~\ref{fig:3}(b) is not allowed because there is no embedding
$[0,1]\times \{P_i\}\hookrightarrow Y_{-1}$ for which the specified arrows of
time are constant.

  \begin{figure}[ht]
  \centering
  \includegraphics[scale=.65]{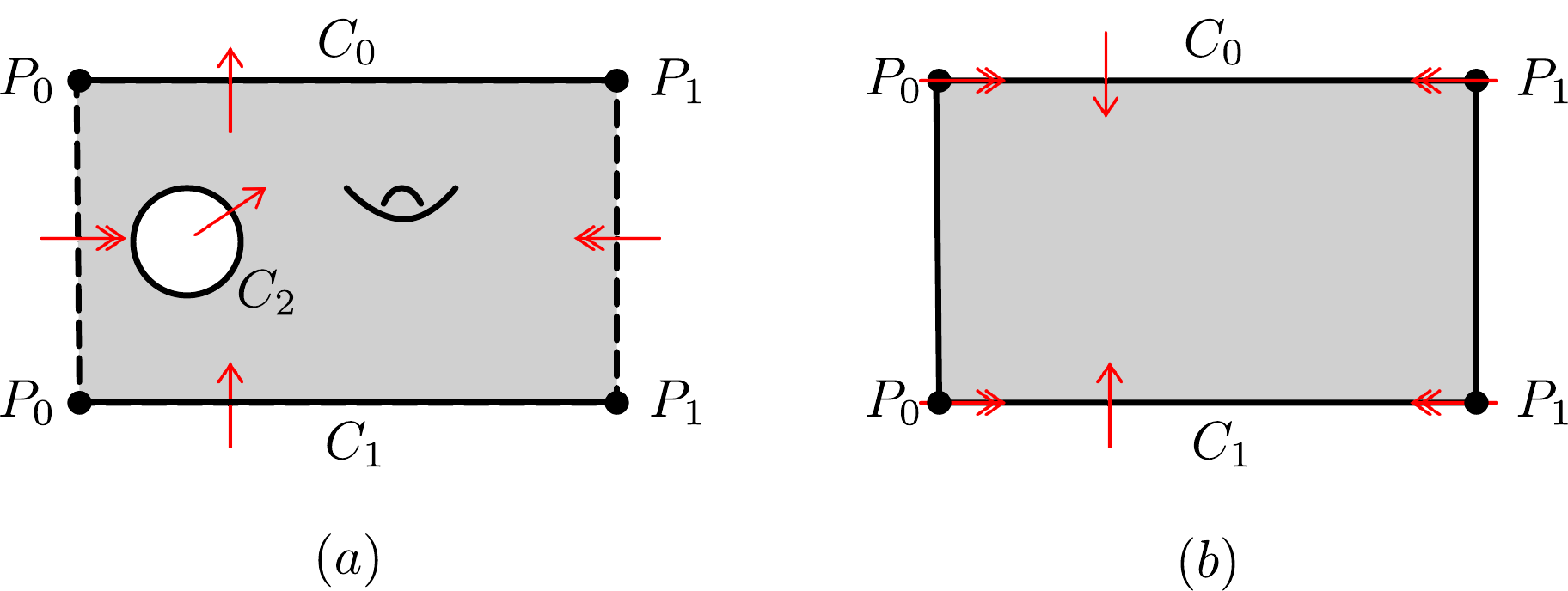}
  \vskip -.5pc
  \caption{(a) a legal 2-morphism and (b) not a 2-morphism}\label{fig:3}
  \end{figure}

A formal specification of constancy data is~\eqref{eq:48}, a consequence of
Definition~\ref{thm:37}. 

   \subsubsection{Tangential structures}\label{subsubsec:2.1.2}
 The main point to emphasize is that in a bordism $n$-category of manifolds
of dimension~$\le n$, the tangential structure is on the \emph{stabilized}
rank~$n$ tangent bundle.  In particular, Figure~\ref{fig:4} is a valid
1-morphism in~$\btf$, the bordism 2-category of \emph{2-framed} manifolds.
In the figure the arrows of time are notated as earlier.  The
2-frame~$f_1,f_2$ is depicted as a long line segment~($f_1$) followed by a
short line segment~$(f_2)$.  There is no relationship imposed between the
arrows of time and the tangential structure (2-framing).  

  \begin{figure}[ht]
  \centering
  \includegraphics[scale=.65]{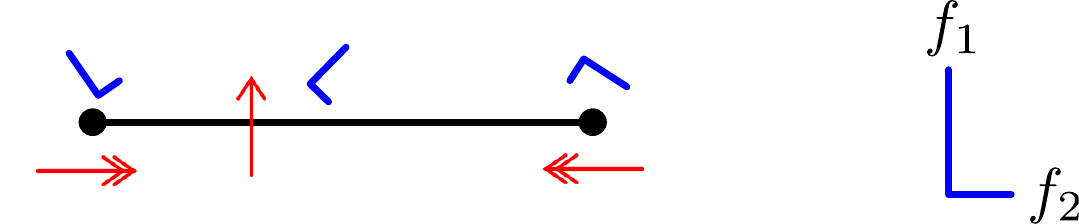}
  \vskip -.5pc
  \caption{A 2-framed 1-morphism}\label{fig:4}
  \end{figure}

  \begin{remark}[]\label{thm:36}
 The relative positions of the framing and the arrows of time does have
significance, for example in Figure~\ref{fig:5} below.
  \end{remark}

  \begin{figure}[ht]
  \centering
  \includegraphics[scale=.65]{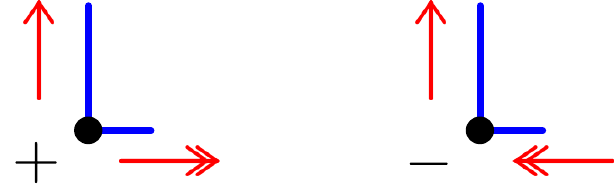}
  \vskip -.5pc
  \caption{The standard points}\label{fig:5}
  \end{figure}

   \subsubsection{Two conventions}\label{subsubsec:2.1.3}
 We depict objects and morphisms in~$\btf$ as manifolds with corners embedded
in the Euclidean plane---the plane of the paper/screen---together with arrows
of time, constancy data, and orthonormal 2-framings.  The first convention is
that we identify two such which are related by either a translation or a
translation composed with a reflection about either a horizontal or a
vertical line.  (If any such identification exists it is unique, since it
preserves the 2-framings.)  Take the standard points~$+,-$ to be those
depicted in Figure~\ref{fig:5}.  Although our pictures lie in $\btf$, our
arguments in the proof are for $\bft$.  (Only in the proof of
Lemma~\ref{thm:18} do we use truly 3-dimensional pictures.)  The second
convention, then, is that we embed $\btf$ in~$\bft$ by adding a trivial line
bundle\footnote{For remarks on conventions, see Remark~\ref{thm:77}.  We
index by codimension, so for $\bft$ the indices are~$-1,-2.-3$.  We might
have used `$f_{-1},f_{-2},f_{-3}$' for the framings, but that would have been
a step too far.} to the (inflated) rank~2 tangent bundle of each $k$-morphism
in~ $\btf$.  Furthermore, we inflate orthonormal 2-framings to orthonormal
3-framings by prepending the standard basis~$f_0$ of the trivial line bundle
to the 2-framing~$f_1,f_2$.  In the pictures we regard~$f_0$ and the aligned
arrow of time as pointing \emph{into} the paper/screen.  (One should then
slide the indices $f_0,f_1,f_2 \mathbin{\squig\squig\squig\rsquigend}
f_1,f_2,f_3$ to normal positions, but we will not do so in the text or
figures.)

   \subsubsection{Duals and adjoints}\label{subsubsec:2.1.6}
 A topological bordism $n$-category or $(\infty ,n)$-category has all duals
and adjoints.  Duals are formed by reversing an arrow of time.  This can be
done at any depth, which reflects the $O(1)^{\times n}$-action discussed in
\cite[Remark~4.4.10]{L1} and \cite[\S4]{BS-P}.  For example, the two standard
points in~$\btf$, depicted in Figure~\ref{fig:5}, are dual by reversing the
double-headed arrow of time.  Right and left adjoints of morphisms are
constructed by a more complicated prescription which we specify
in~\S\ref{subsubsec:4.2.5}.

   \subsubsection{Coloring with a boundary theory}\label{subsubsec:2.1.4}
 We now briefly describe a bordism 2-category $\btbf$ into which $\btf$
embeds.  (Use the second convention of~\S\ref{subsubsec:2.1.3} to extend
objects, 1-morphisms, and 2-morphisms to the 3-category $\bftb$.)  A sketch of
this construction appears in~\cite[Example 4.3.22]{L1}; details for
$\Bord^{\textnormal{fr}}_{n,\partial }$, $n\in \ZZ^{\ge 1}$, are provided
in~\S\ref{subsec:4.3}.

  \begin{figure}[ht]
  \centering
  \includegraphics[scale=.65]{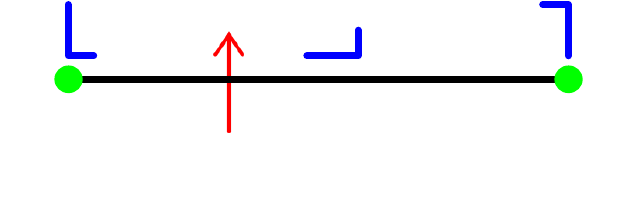}
  \vskip -.5pc
  \caption{A 1-morphism in $\btbf$}\label{fig:6}
  \end{figure}

The 0-morphisms in $\btbf$ are the same as 0-morphisms in $\btf$: a finite
set of 2-framed points equipped with arrows of time.  There are new
1-morphisms, such as the one depicted in Figure~\ref{fig:6}.  The two
boundary points are colored and the extra arrow of time at colored boundary
points is replaced by the conditions that the frame vector~$f_1$ be tangent
to the colored boundary and that $f_2$~ be an inward normal.\footnote{This is
appropriate for a boundary theory $1\to \tF$; for a boundary theory $\tF\to
1$ we require that $f_2$~be an outward normal.}  Effectively, we have a
1-framing at those points.  There are, of course, new 2-morphisms, such as
the one depicted in Figure~\ref{fig:7}.  In terms of the partition
$Y=\mY_0\amalg \mY_{-1}\amalg \mY_{-2}$, there is a submanifold with boundary
$\mB \mstrut _{-1} \amalg \mB \mstrut _{-2}\subset \mY\mstrut _{-1}\amalg
\mY\mstrut _{-2}$ that is colored with the boundary condition; in
Figure~\ref{fig:7} we have $\mB \mstrut _{-2}=\mY\mstrut _{-2}$, and
$\mY\mstrut _{-1}\setminus \mB \mstrut _{-1}$ consists of three open line
segments.  There are no arrows of time on~$\mB \mstrut _{-1}$.  The constancy
data~\eqref{eq:42}, vacuous for Figure~\ref{fig:7}, is as
in~\S\ref{subsubsec:2.1.5} with $\mY\mstrut _{-2}\setminus \mB \mstrut _{-2}$
replacing~$\mY\mstrut _{-2}$.  (Figure~\ref{fig:14} below illustrates the
constancy data.)

  \begin{figure}[ht]
  \centering
  \includegraphics[scale=.65]{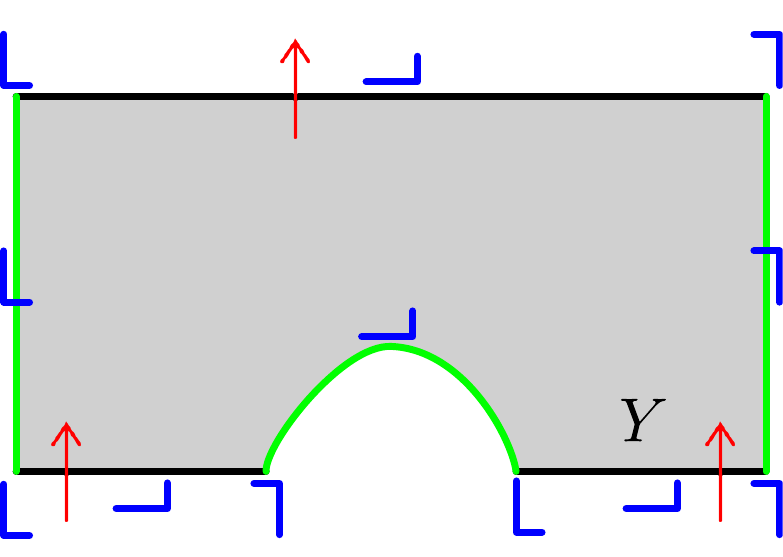}
  \vskip -.5pc
  \caption{A 2-morphism in $\btbf$}\label{fig:7}
  \end{figure}

  \begin{remark}[]\label{thm:9}
  A boundary theory~\eqref{eq:7} for an $n$-dimensional
theory is essentially an $(n-1)$-dimensional theory.  This explains why at a
boundary component ``colored'' with a boundary theory there is one fewer
arrow of time and a constraint on the tangential structure.
  \end{remark}

  \begin{remark}[]\label{thm:10}
 There does not exist a 2-morphism in $\btf$ from which Figure~\ref{fig:7} is
obtained by coloring a subset of the boundary (with inward arrow of time). 
  \end{remark}

  \begin{remark}[]\label{thm:33}
 Intuitively, the colored boundary components include a time direction---they
are timelike---which motivates our convention about which frame vector is
constrained; see Remark~\ref{thm:77}. 
  \end{remark}

   \subsubsection{Natural transformations}\label{subsubsec:2.1.7}
 
A boundary theory~$\beta $ of a topological field theory~$F$ is a natural
transformation of functors out of a bordism category; see~\eqref{eq:7}.  The
pair~$\eF=(F,\beta )$ is best encoded as a functor out of~$\Bnp$, as
in~\S\ref{subsubsec:2.1.4}.  In this section we introduce new ``pictures''
which encode the idea of a natural transformation, and then too an algorithm
for converting them to morphisms in~$\Bnp$.  We do not use these new pictures
in the sequel, which justifies the sparse provisional account here of a
single example. 

  \begin{remark}[]\label{thm:79}
 We leave open the question of whether there is a single bordism category
which includes~$\Bnp$ as well as the new pictures. 
  \end{remark}

  \begin{figure}[ht]
  \centering
  \includegraphics[scale=.6]{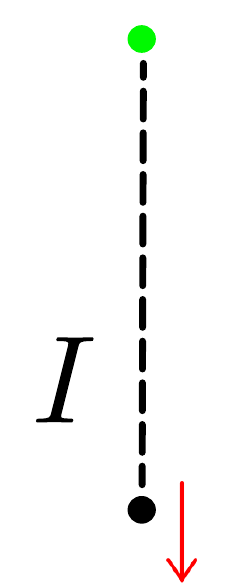}
  \vskip -.5pc
  \caption{The new picture~$I$}\label{fig:35}
  \end{figure}

Introduce the manifold~$I=[0,1]$, embellished as in Figure~\ref{fig:35}.  We
view it as 1-framed, with framing vector aligned with the arrow of time.  Let
$X$~be a $k$-morphism in $\btf$, $k\in \{0,1,2\}$, and consider the Cartesian
product~$I\times X$.  For example, in Figure~\ref{fig:36}(a) we have the
1-morphism $X=\id_{+}$, the identity map on the $+$~point.  Using our
numbering convention $0, 1, 2$ in $\bft$---which replaces the more uniform
convention $-1, -2, -3$ used in Appendix~\ref{sec:4}---we choose to number
the directions in $\btf$ as~$0, 2$ and the direction in~$I$ as~$1$.  Thus in
Figure~\ref{fig:36}(a) the frame vector~$f_0$ and 0-arrow of time point into
the paper/screen.  The frame vector~$f_2$ and 2-arrow of time point to the
right.  The Cartesian product~$I\times X$ is Figure~\ref{fig:35}(b), with
arrows of time and constant 3-framing as indicated.

  \begin{figure}[ht]
  \centering
  \includegraphics[scale=.7]{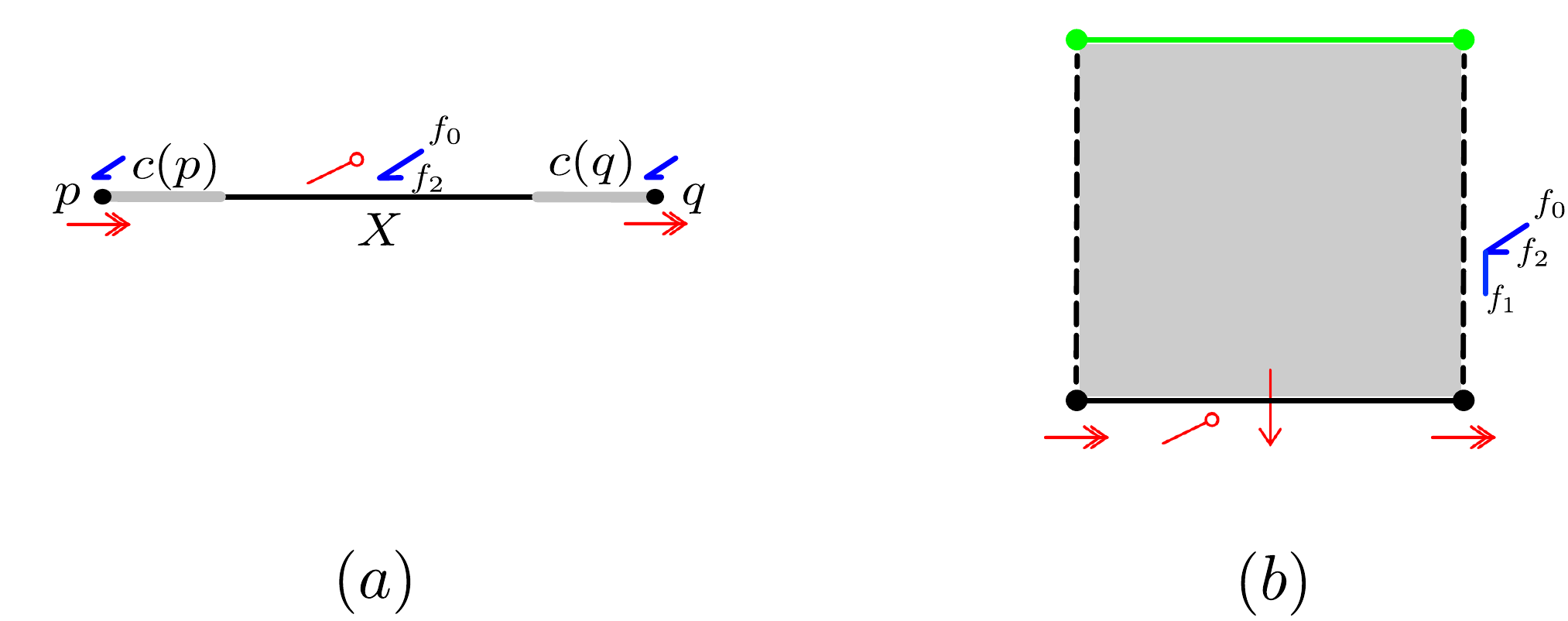}
  \vskip -.5pc
  \caption{(a)~$\id_+$\quad (b)~$I\times X$}\label{fig:36}
  \end{figure}

We now proceed in three steps to convert Figure~\ref{fig:35}(b) to a
2-morphism in~$\bftb$.

 \begin{enumerate}

 \item Rotate the frame through angle~$\pi /2$ in the $f_1$-$f_2$ plane: send
$f_2\to f_1$ and $f_1\to -f_2$.  Use our standard pictorial representation of
the frame vectors~$f_1,f_2$ and omit the pictorial representation of~$f_0$; it
points into the paper/screen everywhere in subsequent pictures.

  \begin{figure}[ht]
  \centering
  \includegraphics[scale=.7]{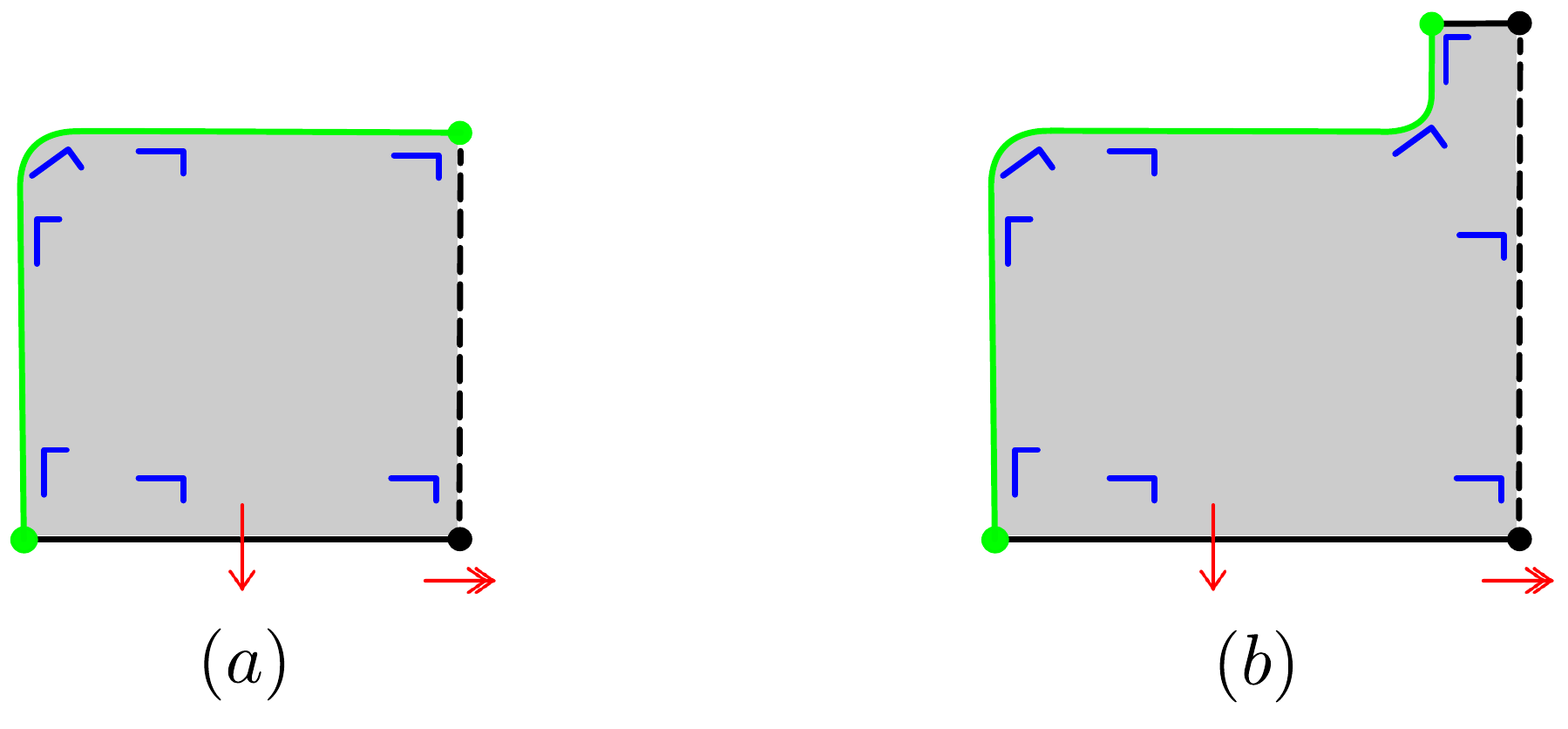}
  \vskip -.5pc
  \caption{Two maneuvers:  Figure~\ref{fig:36}(b) $\mathbin{\squig\squig\squig\rsquigend}$ a 2-morphism
  in~$\bftb$}\label{fig:37} 
  \end{figure}

 \item Now convert to a 2-morphism in $\bftb$.  Let $c(\bX)=c(p)\amalg c(q)$
be a collar neighborhood of the boundary, as in Figure~\ref{fig:36}(a).  For
the incoming boundary point~$p$, ``pull'' the colored boundary in
Figure~\ref{fig:36}(b) ``down'' through $c(p)\times I$ and rotate the framing
accordingly, as depicted in Figure~\ref{fig:37}(a).  Let $I'=[1,2]$, viewed
as glued ``above'' $I=[0,1]$.  For the outgoing boundary point~$q$, ``pull''
the colored boundary in Figure~\ref{fig:36}(b) ``up'' through $c(q)\times I'$
and rotate the framing accordingly, as depicted in Figure~\ref{fig:37}(b).
There is a new edge which connects to the dotted edge.  The arrows of time on
the new edge and on the new vertex are fixed by the constancy condition along
the dotted edge.  The result can be redrawn in a more standard form; see
Figure~\ref{fig:38}.

  \begin{figure}[ht]
  \centering
  \includegraphics[scale=.75]{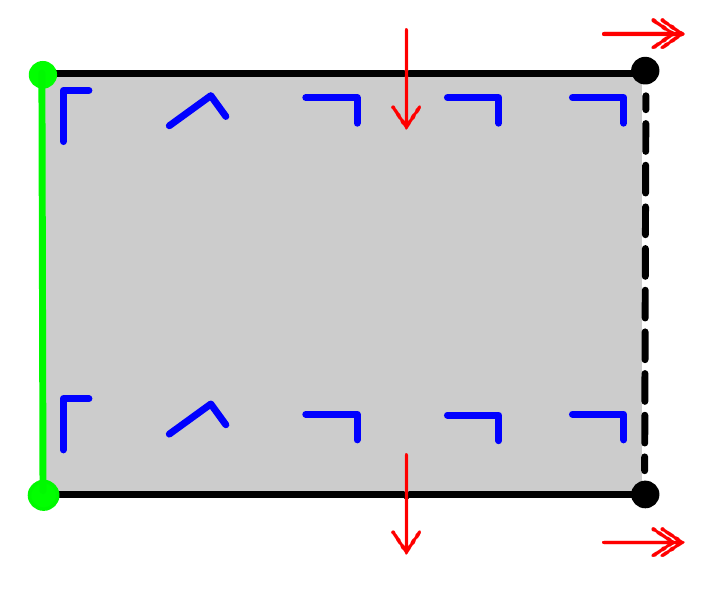}
  \vskip -.5pc
  \caption{The 2-morphism in Figure~\ref{fig:37}(b) redrawn}\label{fig:38}
  \end{figure}

 \item Finally, in a tubular neighborhood of the outgoing dotted boundary,
rotate the frame through angle~$\pi /2$ in the $f_1$-$f_2$~plane by the
inverse of the rotation in step~(1): send $f_1\to f_2$ and $f_2\to -f_1$.
The result is depicted in Figure~\ref{fig:39}.
 \end{enumerate}

  \begin{figure}[ht]
  \centering
  \includegraphics[scale=.8]{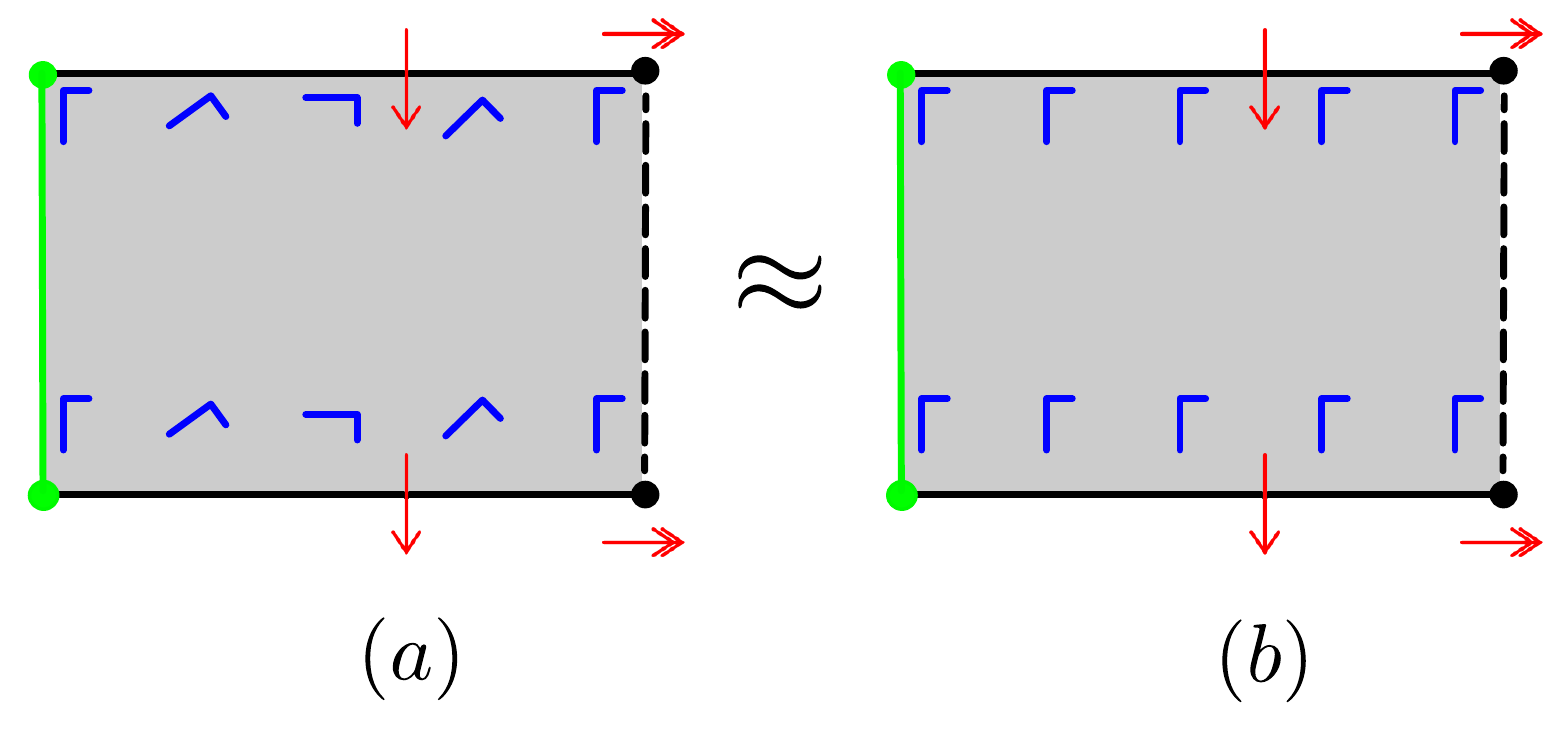}
  \vskip -.5pc
  \caption{(a) The third maneuver (b) The 2-morphism~$\id_{b_+}$ in $\bftb$}\label{fig:39}
  \end{figure}

The 2-morphism represented by Figure~\ref{fig:36}(b) is derived from the
diagram 
  \begin{equation}\label{eq:98}
     \xymatrix@R+1em@C+1em{
     \emptyset ^0\ar[r]^-{\id_{\emptyset
     ^0}}\ar[d]_-{b_+}&\emptyset ^0\ar[d]^-{b_+}\ar@2{->}[dl]\\ 
     +\ar[r]_-{\id_+}&+ } 
  \end{equation}
which is a map $b_+\circ \id_{\emptyset ^0}\Longrightarrow \id_+\circ b_+$,
namely the identity map $\id_{b_+}\:b_+\Longrightarrow b_+$.  This is
precisely the 2-morphism depicted in Figure~\ref{fig:38}. 

  \begin{remark}[]\label{thm:80}
 It is instructive to carry this out for the coevaluation~$c$ and the
evaluation~$e$.  
  \end{remark}

  \subsection{A higher categorical preliminary}\label{subsec:2.3}

   \subsubsection{Internal homs}\label{subsubsec:2.3.4}
 Let $\cM$ be a (weak) 2-category.  A 1-morphism $x\xrightarrow{\;f\;}y$
in~$\cM$ defines a functor $\cM(z,x)\xrightarrow{\;f\circ -\;}\cM(z,y)$ for
any $z\in \cM$.  We call its right adjoint, if it exists, the \emph{right
internal hom} functor: 
  \begin{equation}\label{eq:24}
     \begin{gathered} \xymatrix@C+5ex{\cM(z,x)\ar@<.5ex>[r]^{f\circ -}&\cM(z,y)
     \ar@<.5ex>[l]^{\uH Rf-}}.  \end{gathered} 
  \end{equation}
The \emph{left internal hom} is defined as a right adjoint for right
composition for any~$w\in \cM$: 
  \begin{equation}\label{eq:25}
     \begin{gathered} \xymatrix@C+5ex{\cM(y,w)\ar@<.5ex>[r]^{-\circ f
     }&\cM(x,w) \ar@<.5ex>[l]^{\uH Lf-}}.  \end{gathered} 
  \end{equation}
See \cite{W} and~\cite[\S16]{MaSi} for discussions.\footnote{We thank Emily
Riehl for correspondence and for pointing us to~\cite{W}.  The nonstandard
terminology and notation are our responsibility.  The name `right internal
hom' is apt if~$z=x$, and `left internal hom' is apt if~$w=y$.} 
 
If $\cM$~is a symmetric monoidal 2-category, and the 1-morphism~$f$ has right
and left adjoints~$f^R,f^L$ in~$\cM$, then there are natural isomorphisms
  \begin{equation}\label{eq:26}
     \begin{aligned} \uH Rfg &\cong  f^R\circ g\;\in \cM(z,x),\qquad &&g\:z\to y,
      \\ \uH Lfh &\cong  h\circ f^L\;\in \cM(y,w),\qquad &&h\:x\to w. \\
      \end{aligned} 
  \end{equation}
If $h=g=f$, hence $z=x$ and $w=y$, then we use the notations 
  \begin{equation}\label{eq:27}
     \begin{aligned} \uE Rf :=  \uH Rff\cong f^R\circ f\;\in \cM(x,x), \\ \uE
     Lf := \uH Lff\cong 
     f\circ f^L\;\in \cM(y,y). \\ \end{aligned} 
  \end{equation}

The 1-morphism $\uE Rf$ is an algebra object in the 1-category~$\cM(x,x)$,
and $\uE Lf$~is an algebra object in~$\cM(y,y)$.  The unit of~$\uE Rf$ is the
unit $1\to f^R\circ f$ of the adjunction, and the multiplication uses the
counit $f\circ f^R\to 1$ of the adjunction: 
  \begin{equation}\label{eq:28}
     \uE Rf\circ \uE Rf = f^R\circ (f\circ f^R)\circ f\longrightarrow
     f^R\circ f = \uE Rf. 
  \end{equation}
The formulas for~$\uE Lf$ are similar. 

  \begin{remark}[]\label{thm:22}
 This discussion applies to $n$-categories, $n\ge3$, by taking 2-categorical
slices.  
  \end{remark}

  \begin{remark}[]\label{thm:26}
 A symmetric monoidal functor preserves internal homs, internal
endomorphisms, and the composition laws. 
  \end{remark}

   \subsubsection{Internal homs in~$\FC$ and~$\TC$}\label{subsubsec:2.3.1}
 Recall from Definition~\ref{thm:30} that $\FC$ is a symmetric monoidal
3-category whose objects are fusion categories.  Our convention, different
than some references, is that a $(B,A)$-module category~$M$ is a 1-morphism
$M\:A\to B$ in~$\FC$.  This convention renders tensor products and
compositions in the same order.  The following results are generalized in
\cite[\S5.1]{BJS}. 

  \begin{proposition}[\protect{\cite[\S3.2.1]{DSS}}]\label{thm:24}
 Let $A,B\in \FC$ and $M\:A\to B$ a finite semisimple $(B,A)$-bimodule
category.  Then $M$~has right and left adjoints, and we can take them to be
  \begin{equation}\label{eq:29}
     \begin{aligned} M^R &= \Hom_B(M,B), \\ M^L &= \Hom_A(M,A). \\
      \end{aligned} 
  \end{equation}
  \end{proposition}

  \begin{corollary}[]\label{thm:25}
 If also $C,D\in \FC$ and $N\:C\to B$, $P\:A\to D$ are finite semisimple
bimodule categories, then
  \begin{equation}\label{eq:30}
     \begin{aligned} \uH RMN &= \Hom_B(M,N), \\ \uH LMP &= \Hom_A(M,P). \\
      \end{aligned} 
  \end{equation}
In particular, 
  \begin{equation}\label{eq:31}
     \begin{aligned} \uE RM &= \End_B(M), \\ \uE LM &= \End_A(M). \\
      \end{aligned} 
  \end{equation}
Furthermore, the algebra structure on~$\uE RM$ \textnormal{(}and $\uE
LM$\textnormal{)} is composition of module functor endomorphisms.
  \end{corollary}

  \begin{proof}
 All but the final assertion follow from Proposition~\ref{thm:24} and
\cite[Proposition 2.4.10]{DSS}.  For the algebra structure on~$\uE RM$, the
unit $1\to M^R\circ M$ is the $(A,A)$-bimodule map
  \begin{equation}\label{eq:32}
     A\longrightarrow \Hom_B(M,B)\boxtimes _BM\cong \End_B(M) 
  \end{equation}
which maps $a\in A$ to right multiplication by~$a$; in particular $1\in
A$~maps to the identity endomorphism of~$M$.  Since the counit $M\circ M^R\to
1$ is the evaluation map, the multiplication~\eqref{eq:28} on~$\uE RM$ 
  \begin{equation}\label{eq:33}
  \begin{split}
     \End_B(M)\boxtimes_A \End_B(M) \cong \Hom_B(M,B)\boxtimes _B\bigl(M\boxtimes
     _A\Hom_B(M,B) \bigr)\boxtimes _BM\qquad \qquad \\\longrightarrow
     \Hom_B(M,B)\boxtimes _BB\boxtimes _BM\cong \End_B(M) 
  \end{split}
  \end{equation}
is the usual composition of endomorphisms of~$M$. 
  \end{proof}

We prove a partial converse to Proposition~\ref{thm:24} in~$\TC$, which for
convenience we state for \emph{right} adjoints only. 

  \begin{proposition}[]\label{thm:56}
 Suppose $A,B\in \TC$ and $M\:A\to B$ has a right adjoint~$N$.  Then $N\cong
\Hom_B(M,B)$ as $(A,B)$-bimodule categories. 
  \end{proposition}

  \begin{proof}
 The adjunction can be expressed as isomorphisms 
  \begin{equation}\label{eq:78}
     \Hom_A(X,N\boxtimes _BY)\xrightarrow{\;\;\sim \;\;}\Hom_B(M\boxtimes
     _AX,Y) 
  \end{equation}
which are functorial in left $A$-module categories~$X$ and right $A$-module
categories~$Y$.  Choose~$X=A$ and~$Y=B$ to obtain the desired isomorphism
$N\xrightarrow{\;\sim \;}\Hom_B(M,B)$.  The isomorphism intertwines the
right $A$-action on~$X$ and the right $B$-action on~$Y$. 
  \end{proof}

   \section{Proofs}\label{sec:7}

  \subsection{Proof of Theorem~\ref{thm:46}}\label{subsec:2.2} 

  \begin{figure}[ht]
  \centering
  \includegraphics[scale=.6]{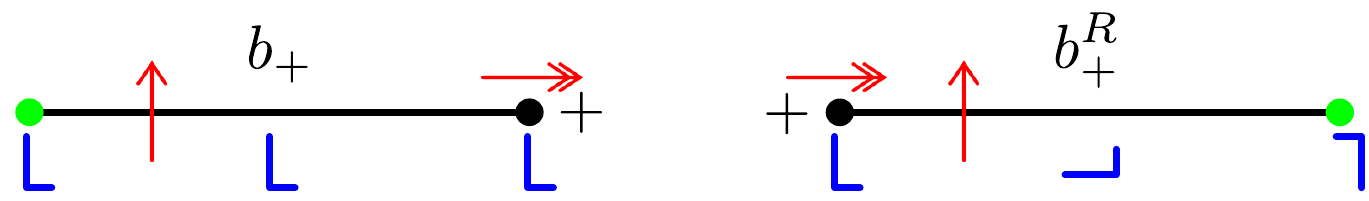}
  \vskip -.5pc
  \caption{The 1-morphisms whose $\eF$-images are~$\beta (+)$
and~$\bR(+)$}\label{fig:13}
  \end{figure}

To begin, assume given a symmetric monoidal 3-category~$\sC$ such that
$\OCfd\subset \cc$ and a symmetric monoidal functor $\eF\:\bftb\to\sC$.  The
restriction of~$\eF$ to $\bft$ is denoted~$F$, and we define~$\beta $ as
in~\eqref{eq:7}.

 Diagrams for the extended theory $\eF\:\bftb\to\sC$, are drawn according to
the rules of~\S\ref{subsubsec:2.1.4}; see \S\ref{subsec:4.3} for more detail.
Figure~\ref{fig:13} depicts a 1-morphism $b_+\:\emptyset ^0\to +$ in
$\btbf\subset \bftb$ and its right adjoint $b_+^R\:+\to\emptyset ^0$, the
latter constructed according to~\S\ref{subsubsec:4.2.5}.  The $\eF$-image
of~$b_+$ is $\beta (+)\:1\to F(+)$.  The cobordism hypothesis with
singularities \cite[\S4.3]{L1} implies that the right adjoint boundary theory
$\bR\:\tF\to 1$~is determined by~$F\p$ and~$\bR(+):=\eF (b_+^R)$, hence the
verification that $\bR$~is right adjoint to~$\beta $ proceeds by producing a
unit and counit (Figure~\ref{fig:14}) for an adjunction between~$b_+$
and~$b_+^R$ in~$\btbf$, and then using their $\eF$-images to exhibit $\bR(+)$
as the right adjoint of~ $\beta (+)$ in~$\sC$.

  \begin{figure}[ht]
  \centering
  \includegraphics[scale=.55]{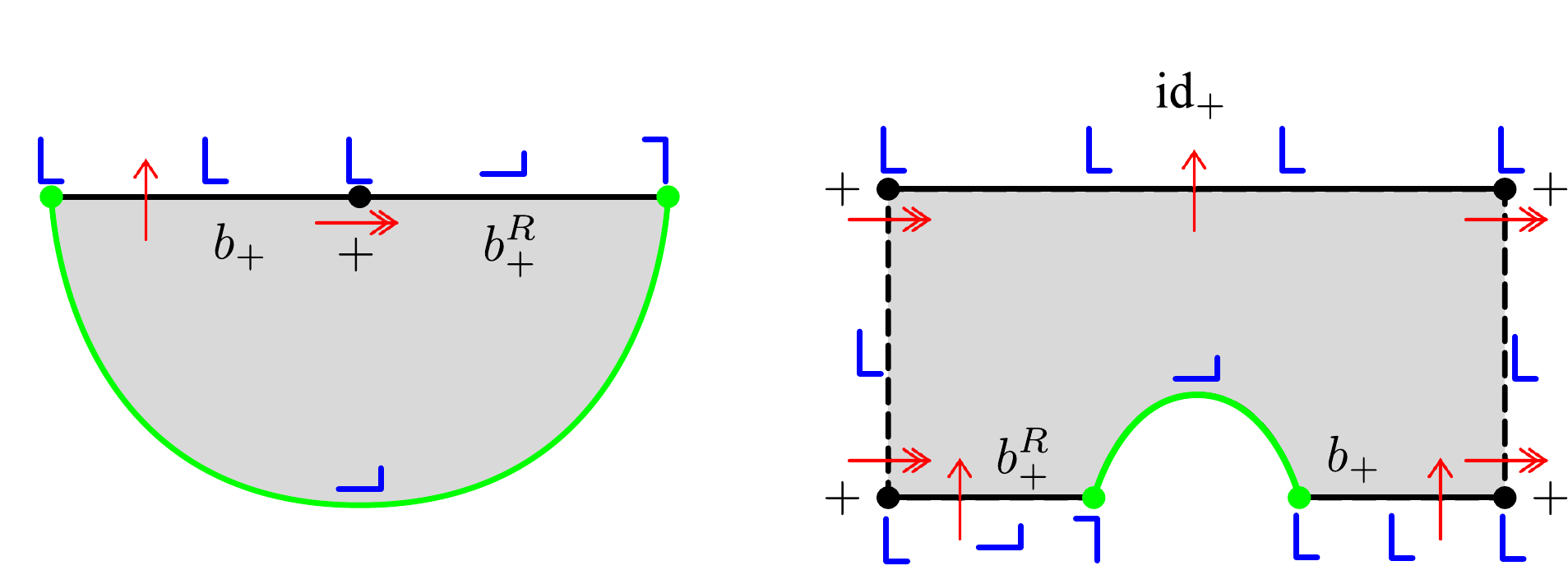}
  \vskip -.5pc
  \caption{The unit $1\to b_+^R\circ b_+$ and counit $b_+\circ b_+^R\to
  1$}\label{fig:14} 
  \end{figure}

  \begin{definition}[]\label{thm:14}
 Set 
  \begin{equation}\label{eq:35}
     \Phi = \eF\bigl(\uE R{b_+} \bigr)\;\in \OCfd\subset \cc. 
  \end{equation}
  \end{definition}

  \begin{figure}[ht]
  \centering
  \includegraphics[scale=.6]{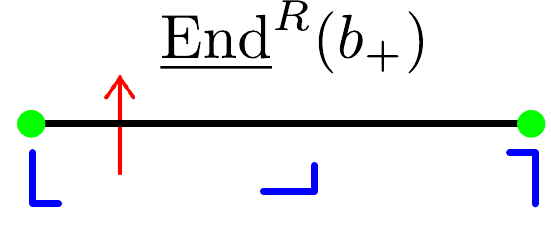}
  \vskip -.5pc
  \caption{The category~$\Phi $ is the $\eF$-image of $\uE
  R{b_+}=b_+^R\circ b_+$}\label{fig:15}
  \end{figure}

\noindent
 See Figure~\ref{fig:15} for a depiction of~$\uE R{b_+}$.  Also, note we can
write $\Phi=\uEnd^R\bigl(\beta (+) \bigr)$.  By virtue of being~$\uEnd^R$ of
a 1-morphism, $\Phi $~ is an algebra object in~$\cc$, i.e., $\Phi $~is a
tensor category.  The composition law~\eqref{eq:28} is the $\eF$-image of the
2-morphism depicted in Figure~\ref{fig:16}.

  \begin{figure}[ht]
  \centering
  \includegraphics[scale=.7]{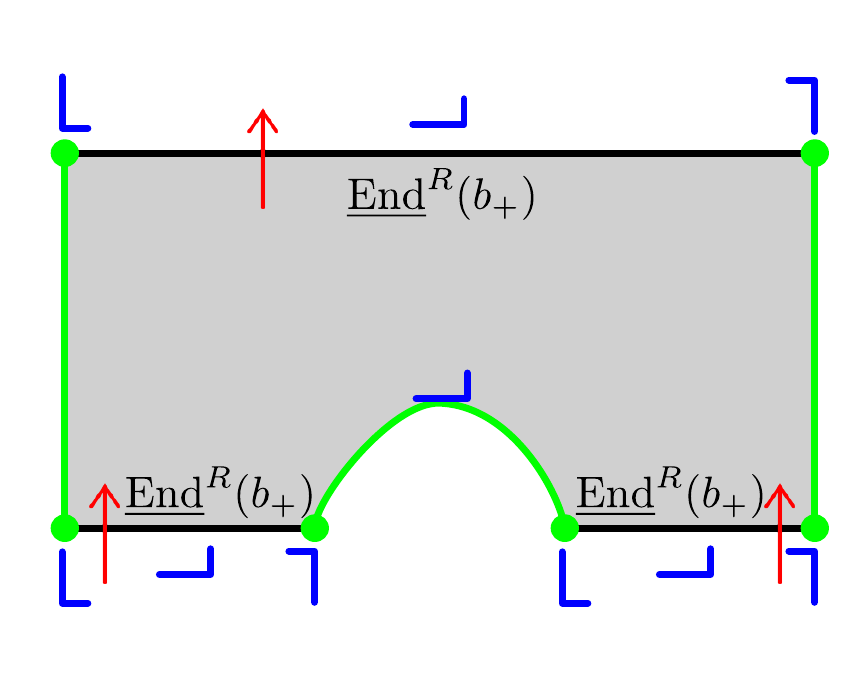}
  \vskip -.5pc
  \caption{The monoidal structure of~$\Phi $ is the $\eF$-image of this
  2-morphism}\label{fig:16} 
  \end{figure}

  \begin{proposition}[]\label{thm:20}
 $\Phi $~is a fusion category. 
  \end{proposition}

\noindent 
 The proof of Proposition~\ref{thm:20} is broken up into Lemma~\ref{thm:19}
and Lemma~\ref{thm:18}.

  \begin{lemma}[]\label{thm:19}
 $\Phi $~is a finite semisimple abelian category. 
  \end{lemma}

\noindent 
 The stronger hypothesis of Theorem~\ref{thm:5}, that $\FC\subset \sCfd$ is a
\emph{full} subcategory, immediately implies Lemma~\ref{thm:19} in view
of~\eqref{eq:74}.  Under the hypotheses of Theorem~\ref{thm:46} we use the
following argument.

  \begin{proof}
 The 1-morphism in $\btbf$ depicted in Figure~\ref{fig:18} has left boundary
colored with~$\beta $ and right boundary colored with~$\bR$.  It evaluates
under~$(F,\beta ,\bR)$ to~$\Phi $.  Define the \emph{dimensional reduction}
  \begin{equation}\label{eq:19}
     \bF\:\btf\longrightarrow \SC 
  \end{equation}
of~$F$ as the theory whose value on any object or morphism in $\btf$ is the
value of~$F$ on its Cartesian product with the 1-morphism $\emptyset
^0\to\emptyset ^0$ in Figure~\ref{fig:18}.  Since the 2-framing of the latter
is induced from a 1-framing, the Cartesian product is naturally equipped with
a 3-framing.  The lemma now follows since $\bF(+)$~is a 2-dualizable
category, hence is finite semisimple abelian by Theorem~\ref{thm:53}.
  \end{proof}

  \begin{figure}[ht]
  \centering
  \includegraphics[scale=.6]{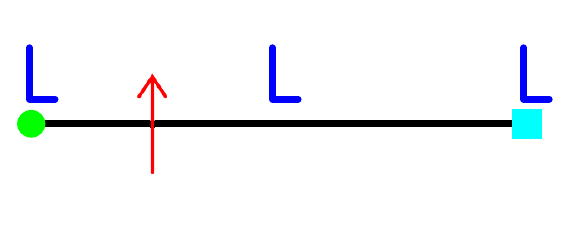}
  \vskip -.5pc
  \caption{A 1-framed bordism with boundary theories $\beta $~on the left and
  $\bR$ on the right}\label{fig:18}
  \end{figure}

  \begin{lemma}[]\label{thm:18}
 $\Phi $~is a rigid monoidal category. 
  \end{lemma}

\noindent 
 That is, $\Phi $~has internal left and right duals~\cite[\S2.10]{EGNO}.
See Appendix~\ref{sec:5} for a discussion and generalization of rigidity.

  \begin{proof}
 As a corollary of Lemma~\ref{thm:19}, the dual~$\Phi \dual$ to ~$\Phi $ in~$\SC$
is its opposite category.  We prove rigidity by verifying the hypotheses of
Theorem~\ref{thm:B1}.

  \begin{figure}[ht]
  \centering
  \includegraphics[scale=.55]{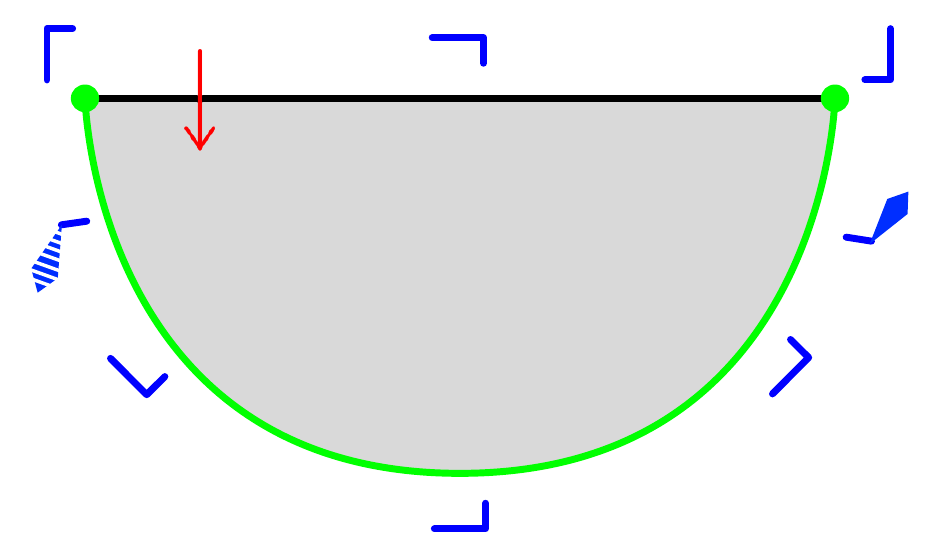}
  \vskip -.5pc
  \caption{The right adjoint to the unit in Figure~\ref{fig:14}}\label{fig:28}
  \end{figure}

The unit of~$\Phi $ is the $\eF$-image of the first 2-morphism in
Figure~\ref{fig:14}.  Its right adjoint has $\eF$-image the
counit~$\varepsilon $ of~\eqref{eq:64}.  We implement the prescription
of~\S\ref{subsubsec:4.2.5} to compute it.  The main concern is the 3-framing
which results on the boundary of the hemidisk; it necessarily extends to the
interior, and since $\pi _2SO_3=0$ that extension is unique up to isotopy.
The result is illustrated in Figure~\ref{fig:28}.  Recall
(\S\ref{subsubsec:2.1.4}) that the frame vectors are labeled\footnote{The
numbering $0,1,2$ corresponds to the numbering $-1,-2,-3$ by codimension
utilized in Appendix~\ref{sec:4}.}  $f_0,f_1,f_2$; that $f_1$~is depicted as
long, $f_2$~as short; and that in all previous pictures, such as
Figure~\ref{fig:14}, the vectors~$f_1,f_2$ lie in the plane of the
paper/screen and $f_0$~is perpendicular to that plane and points into the
paper/screen.  Now, in Figure~\ref{fig:28}, the vectors~$f_0,f_1$ rotate in
the plane perpendicular to~$f_2$ as we descend from the incoming boundary.
The dashed line in
 \raisebox{-5pt}{\includegraphics[scale=.75]{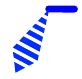}}
 indicates that $f_1$~points into the paper/screen; the solid wedge in
 \raisebox{-6pt}{\includegraphics[scale=.7]{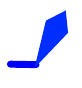}}
 indicates that $f_1$~points out of the paper/screen.  Compose this right
adjoint with the multiplication depicted in Figure~\ref{fig:16} to compute
the 2-morphism in~$\bft$ whose $\eF$-image is the pairing~$B$ of
Appendix~\ref{sec:5}.  It and the 2-morphisms obtained from it by
duality~\eqref{eq:65} are depicted in Figure~\ref{fig:31}.  The Frobenius
condition of Definition~\ref{frob} is satisfied since the latter two
2-morphisms are invertible in~$\bft$.

  \begin{figure}[ht]
  \centering
  \includegraphics[scale=.5]{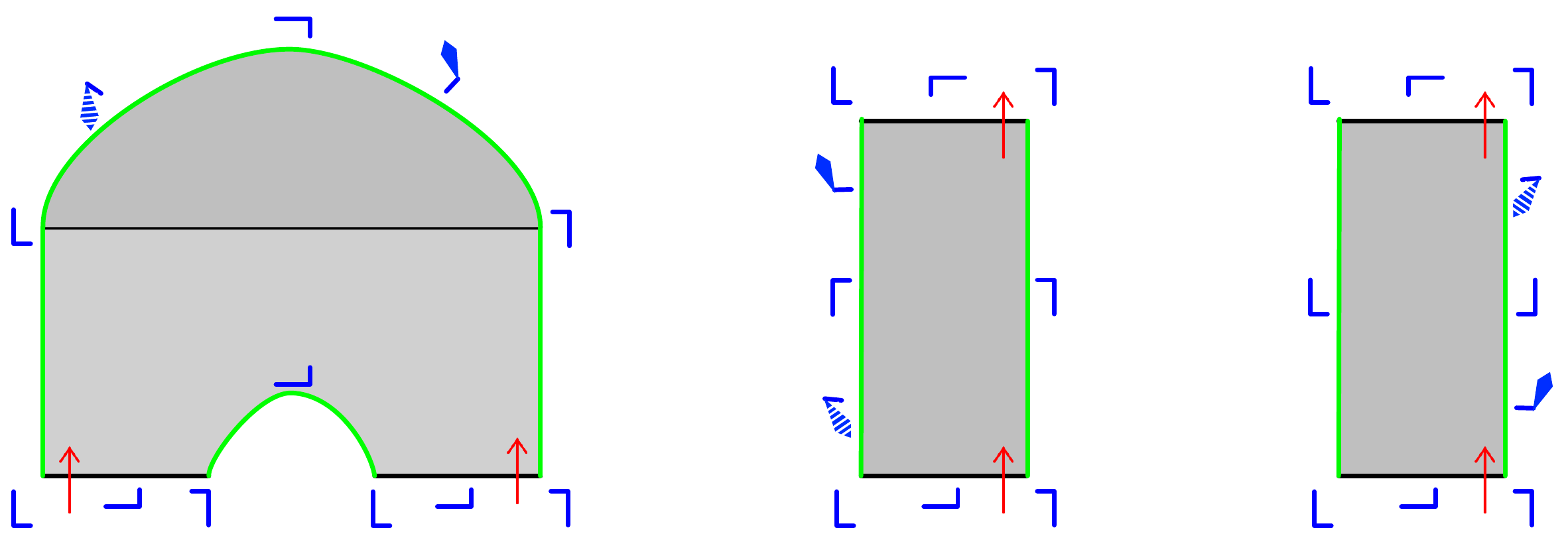}
  \vskip -.5pc
  \caption{Multiplication followed by the counit and its duals}\label{fig:31}
  \end{figure}

The right adjoint bordism to Figure~\ref{fig:16}, computed
following~\S\ref{subsubsec:4.2.5}, is depicted in Figure~\ref{fig:32}; its
$\eF$-image is the comultiplication~$\Delta $ on~$\Phi $.  Again it suffices to
compute the framing on the boundary.  Notice that the half-turns in the
framing cancel on the vertical colored edges, whereas they cohere into a full
turn on the colored half-circle.  The second condition in
Theorem~\ref{thm:B1}, that the comultiplication is a bimodule map, follows
immediately from Figure~\ref{fig:33}.
  \end{proof}

  \begin{figure}[ht]
  \centering
  \includegraphics[scale=.55]{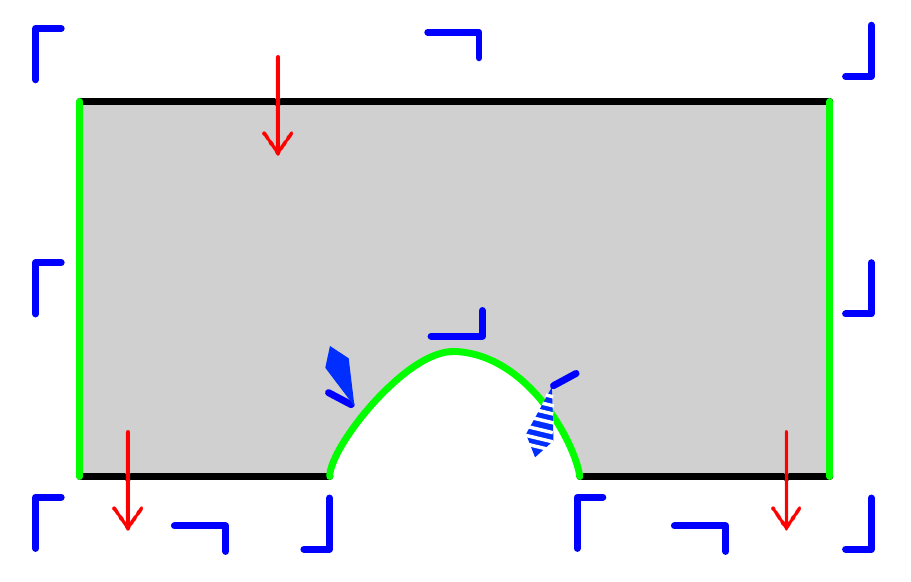}
  \vskip -.5pc
  \caption{The right adjoint to Figure~\ref{fig:16}}\label{fig:32}
  \end{figure}

  \begin{figure}[ht]
  \centering
  \includegraphics[scale=.4]{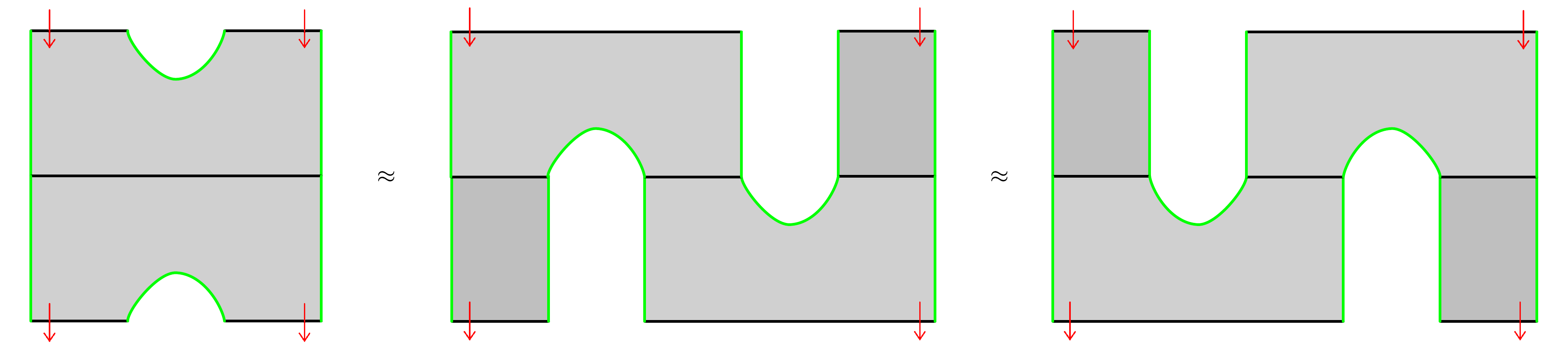}
  \vskip -.5pc
  \caption{Comultiplication is a bimodule map}\label{fig:33}
  \end{figure}

  \begin{remark}[]\label{thm:62}
 In the context of Theorem~\ref{thm:5}, identify the category~$\Phi $ with
its collection of objects $\Hom_{\FC}(1,\Phi )=\Hom_{\sCfd}(1,\Phi )$.  By
duality,
  \begin{equation}\label{eq:92}
     \Hom_{\sCfd}(1,\Phi ) \cong \Hom_{\sCfd}\bigl(1,\bR\p\circ \beta \p
     \bigr)\cong \Hom_{\sCfd}\bigl(\bp,\bp\bigr), 
  \end{equation}
which explains the last statement in Theorem~\ref{thm:5}.  
  \end{remark}

  \begin{proof}[Proof of Theorem~\ref{thm:46}]
 As pointed out earlier, the forward direction follows from
Theorem~\ref{thm:51}.  For the converse, first apply the cobordism hypothesis
to construct $F\:\bft\to\TC$ with $F(+)=\Psi $, and then apply the cobordism
hypothesis with singularities to construct an extension $\eF\:\bftb\to\TC$
whose associated boundary theory $\beta \:1\to\tF$ has $\beta (+)=\PL $, the
regular left $\Psi $-module.  The category defined in Definition~\ref{thm:14}
is $\Phi =\uE R{\PL}$.  By Proposition~\ref{thm:56} the right adjoint
to~$\PL$ is $\Hom_{\Psi }(\Psi ,\Psi )$, which may be identified with $\PR$,
the regular right $\Psi $-module.  Hence $\Phi =\uE R{\PL}\cong \Psi
\boxtimes_\Psi \Psi \cong \Psi$.  Conclude using Proposition~\ref{thm:20}.
  \end{proof}

  \begin{figure}[ht]
  \centering
  \includegraphics[scale=.65]{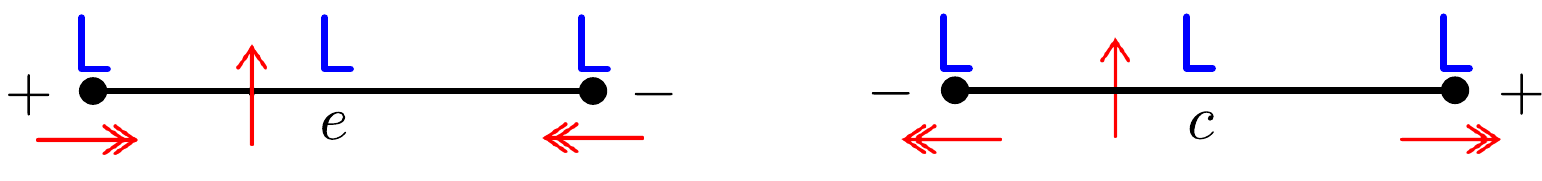}
  \vskip -.5pc
  \caption{Duality of~$+$ and~$-$: evaluation and coevaluation}\label{fig:8}
  \end{figure}

  \subsection{Proof of Theorem~\ref{thm:5}}\label{subsec:7.1}

 The $+$~point and $-$~point (Figure~\ref{fig:5}) are duals in~$\btf$.
Choose duality data as the evaluation $e\:+\amalg -\to\emptyset ^0$ and
coevaluation $c\:\emptyset ^0\to -\amalg +$ 1-morphisms depicted in
Figure~\ref{fig:8}.  One of the ``S-diagrams''
  \begin{equation}\label{eq:11}
     \left( +\xrightarrow{\;\id\otimes c\;}+\amalg -\amalg
     +\xrightarrow{\;e\otimes \id\;} +\right) \xrightarrow{\;\;\sim
     \;\;}\left( +\xrightarrow{\;\id\;} +\right) 
  \end{equation}
that proves that $(-,c,e)$ are duality data for~$+$ is the 2-morphism of
Figure~\ref{fig:9}.

  \begin{figure}[ht]
  \centering
  \includegraphics[scale=.6]{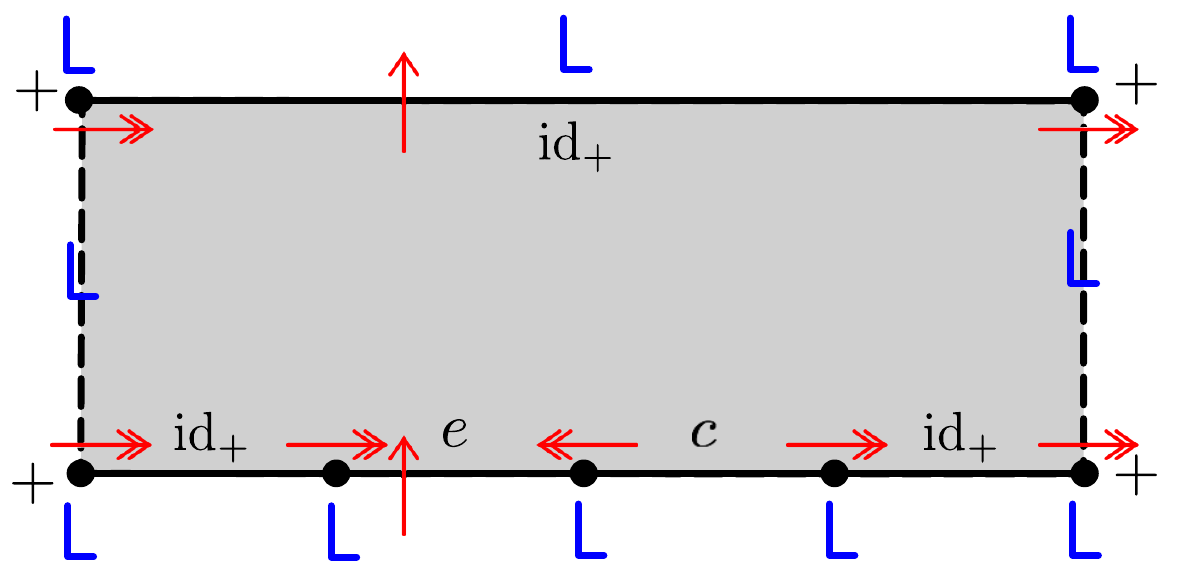}
  \vskip -.5pc
  \caption{The S-diagram~\eqref{eq:11}}\label{fig:9}
  \end{figure}

The 1-morphism~$e$ has right and left adjoints $e^R,e^L\:\emptyset ^0\to
-\amalg +$ depicted in Figure~\ref{fig:10}.  Their construction follows the
general prescription in~\S\ref{subsubsec:4.2.5}; see especially
Figure~\ref{fig:26}.  In $\btf$ they are distinct and distinct from
coevaluation: $e^R\neq c\neq e^L$.  In $\bft$ we have $e^L\cong e^R$ since $\pi
_1\SO_3\cong \zt$ with generator a full rotation of the frame $f_0,f_1,f_2$
in the $f_1$-$f_2$ plane.  In $\bft$ we have an isomorphism
  \begin{equation}\label{eq:12}
     S^1_b\cong e\circ e^L = \uE Le 
  \end{equation}
illustrated in Figure~\ref{fig:11}.  

  \begin{figure}[ht]
  \centering
  \includegraphics[scale=.65]{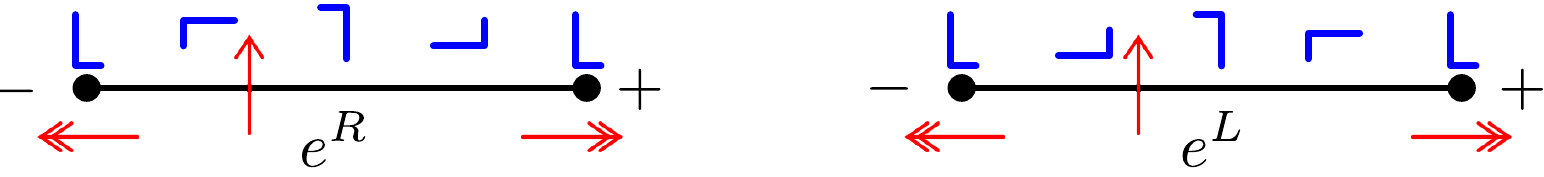}
  \vskip -.5pc
  \caption{The right and left adjoints to~$e$}\label{fig:10}
  \end{figure}

  \begin{remark}[]\label{thm:11}
 In $\bft$ the nonbounding 3-framed circle~$S^1_n$ satisfies the isomorphism 
  \begin{equation}\label{eq:13}
     S^1_n\cong e\circ c. 
  \end{equation}
  \end{remark}

  \begin{figure}[ht]
  \centering
  \includegraphics[scale=.6]{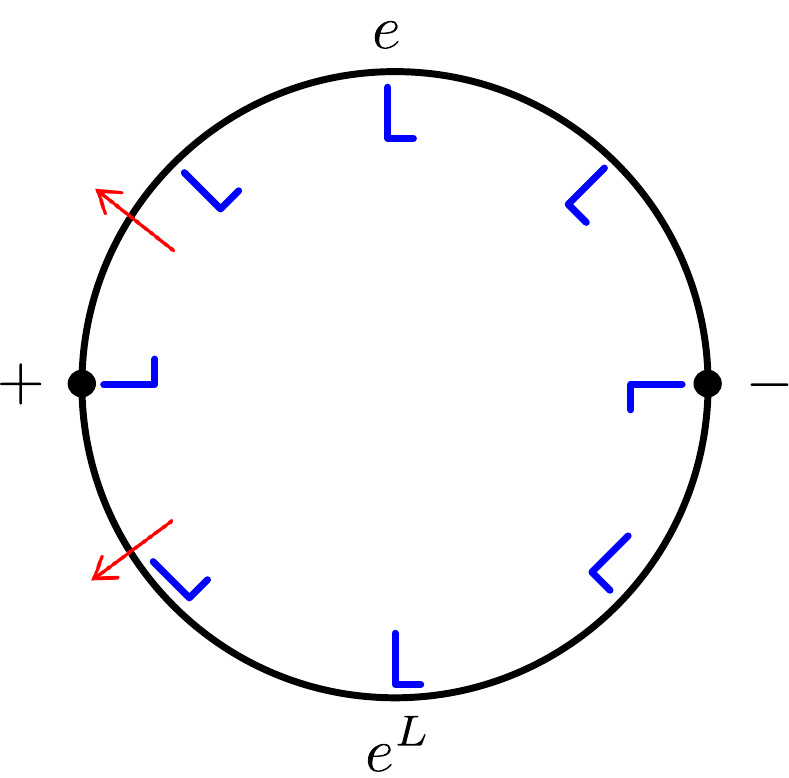}
  \vskip -.5pc
  \caption{The isomorphism $\Sb\cong e\circ e^L$ in $\btf$}\label{fig:11}
  \end{figure}

Define the 2-framed 0-sphere as $S^0=+\amalg -$.  Note $e\:S^0\to\emptyset
^0$.

  \begin{proposition}[]\label{thm:12}
 Let $F$~satisfy the hypotheses of Theorem~\ref{thm:5}.  Let $\Xi $~be a
fusion category which is isomorphic to~$F(S^0)$.  Then $\Xi $~ is Morita
equivalent to~$F(S^1_b)$ as fusion categories.
  \end{proposition}

Recall that a fusion category is \emph{indecomposable} if it is not a
nontrivial direct sum.  As a preliminary we prove the following.

  \begin{lemma}[]\label{thm:13}
 $\Xi $ is an indecomposable fusion category.
  \end{lemma}

  \begin{proof}
 As in Remark~\ref{thm:4} introduce the double theory
  \begin{equation}\label{eq:14}
     \DF\:\bft\longrightarrow \FC 
  \end{equation}
characterized by $\DF(+)=\Xi  \cong F(S^0)= F(+)\otimes F(-)\cong
F(+)\otimes F(+)\dual$.  Then by Hypothesis~(b) of Theorem~\ref{thm:5},
deduce that $\DF(S^1_b)\cong F(S^1_b)\boxtimes F(S^1_b)\rev$ is invertible as a
braided tensor category.  (Recall that the reverse of a braided tensor
category is the same underlying tensor category equipped with the inverse
braiding.)  On the other hand, $\DF(S^1_b)$ is the Drinfeld center
of~$\DF(+)\cong \Xi $.  Since the Drinfeld center of the direct sum of tensor
categories is the direct sum of the Drinfeld centers, and a nontrivial direct
sum is not invertible, it follows that $\Xi $~is indecomposable.
  \end{proof}

  \begin{proof}[Proof of Proposition~\ref{thm:12}]
 Define
  \begin{equation}\label{eq:37}
     M:=\Xi \xrightarrow{\;\;\sim \;\;}F(S^0)\xrightarrow{\;\;F(e)\;\;}  1, 
  \end{equation}
where $1=F(\emptyset )=\Vc\in \FC\subset \sC$ is the tensor unit.  Since
$\FC\subset \sC$ is a \emph{full} subcategory, $M$~is a 1-morphism in~$\FC$.
By~\eqref{eq:12} and~\eqref{eq:31} we have the categorical equivalences
  \begin{equation}\label{eq:15}
     \begin{aligned} F(\Sb)&\eqcat F\bigl(\uE Le \bigr)\\ &\eqcat
     \uEnd^L\bigl(F(e) \bigr) \\ &\eqcat\End\mstrut _\Xi (M).\end{aligned} 
  \end{equation}
The last assertion of Corollary~\ref{thm:25} implies that \eqref{eq:15}~is an
equivalence of {tensor categories}.

To conclude the proof of Proposition~\ref{thm:12} we must show that $M$~is a
\emph{faithful} right $\Xi $-module.  According to the remark after
\cite[Definition 7.12.9]{EGNO}, this follows from Lemma~\ref{thm:13} since
$F(e)$~ is nonzero by virtue of being part of the duality data between
$F(+)$ and~$F(-)$.
  \end{proof}

 Since $\Phi $ is a fusion category, as in~\eqref{eq:3} the cobordism
hypothesis produces
  \begin{equation}\label{eq:20}
     \TP\:\bft\longrightarrow \FC ,
  \end{equation}
a theory of Turaev-Viro type with $\TP\p=\Phi $.  Then $\TP(\Sb)$~is the
Drinfeld center
  \begin{equation}\label{eq:21}
     Z\bigl(\Phi \bigr) = \End_{\TP(S^0)}\bigl(\Phi \bigr), 
  \end{equation}
where $\TP(S^0)= \TP(+)\boxtimes \TP(-)\simeq \Phi\boxtimes
\Phi^{\textnormal{mo}}$.  (Here $\Phi^{\textnormal{mo}}$ is the monoidal
opposite to the monoidal category~$\Phi$, which is its dual in~$\FC$.)  We
identify
  \begin{equation}\label{eq:36}
     \TP\cong \uE R\beta 
  \end{equation}
as algebra objects in the 2-category of symmetric monoidal functors
$\btf\to\Omega \sC=\SC$.

  \begin{figure}[ht]
  \centering
  \includegraphics[scale=.6]{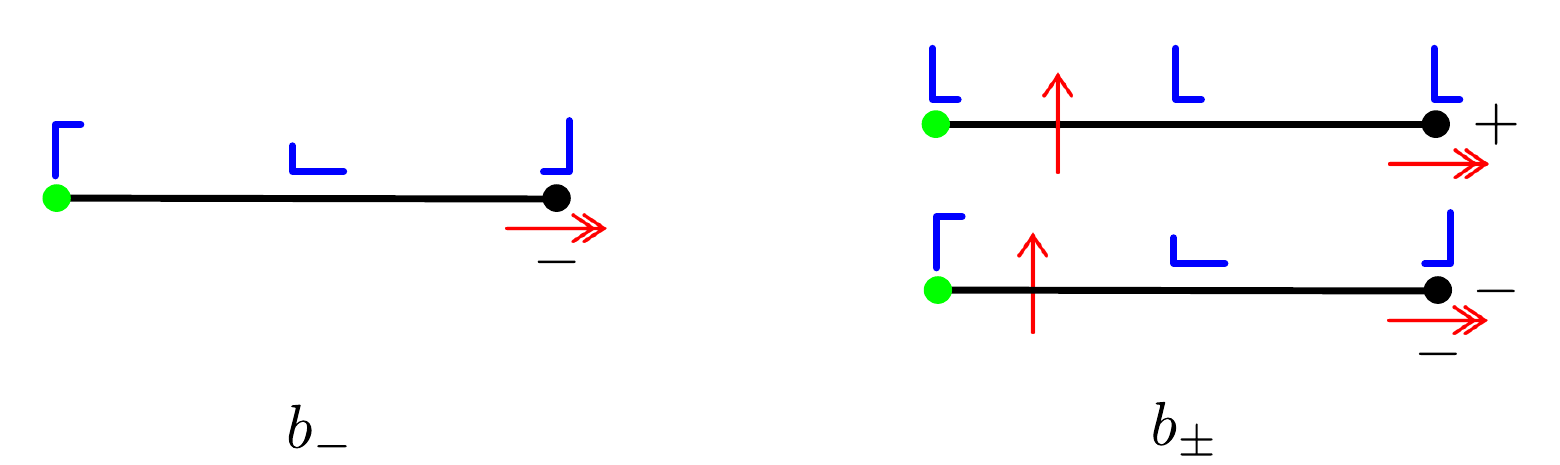}
  \vskip -.5pc
  \caption{The bordisms~$b_-$ and~$\fpm$}\label{fig:34}
  \end{figure}

Let $b_-\:\emptyset \to -$ be the morphism dual to~$b_+^R$ in $\bftb$,
obtained by reversing the double-headed arrow; see~\S\ref{subsubsec:2.1.6}
and Figure~\ref{fig:34}.  Let $\fpm = b_+\amalg b_-$.  Then the proof of
\cite[Proposition~7.10]{JS} implies that $\beta (-)\cong \eF(b_-)$.  Define
the composition
  \begin{equation}\label{eq:38}
     N\:1\xrightarrow{\;\;\eF(\fpm)\;\;} F(S^0)\xrightarrow{\;\;\sim
     \;\;}\Xi . 
  \end{equation}
Notice $\eF(\fpm)=\beta (S^0)$.  As in the proof of Proposition~\ref{thm:12},
$N$~is a 1-morphism in~$\FC$.  Apply~\eqref{eq:36} and~ \eqref{eq:31} to
deduce an equivalence of tensor categories
  \begin{equation}\label{eq:39}
     \TP(S^0)\cong \uE RN\cong \End_\Xi (N). 
  \end{equation}

  \begin{lemma}[]\label{thm:28}
 The left $\Xi $-module category~$N$ provides a Morita equivalence
$\Xi \xrightarrow{\;\sim\;}\TP(S^0)$.
  \end{lemma}

  \begin{proof} 
 As in the proof of Proposition~\ref{thm:12}, it suffices to show that $N$~is
nonzero.  But $\beta (S^0)=\beta (+)\otimes \beta (-)$, so if $N=0$ then so
too $\beta (+)=0$, and then the cobordism hypothesis would imply~$\beta =0$,
which contradicts the hypothesis in Theorem~\ref{thm:5} that $\beta $~is
nonzero.
  \end{proof}

  \begin{figure}[ht]
  \centering
  \includegraphics[scale=.6]{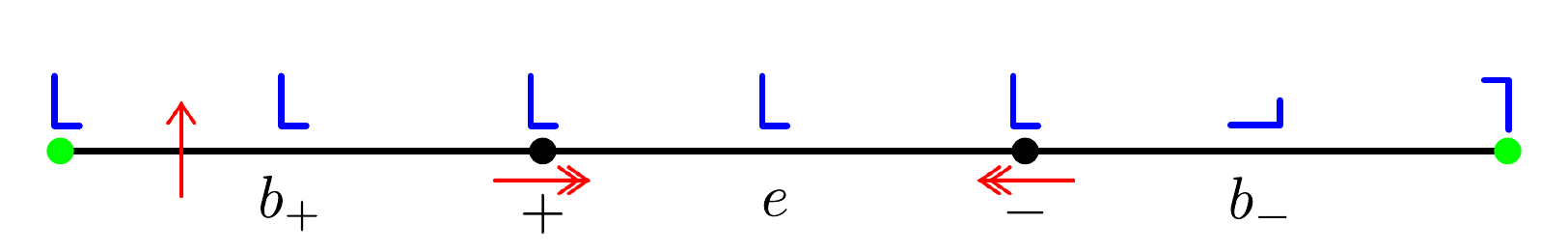}
  \vskip -.5pc
  \caption{The equivalence~\eqref{eq:17}}\label{fig:17} 
  \end{figure}

  \begin{lemma}[]\label{thm:15}
 There is an equivalence of categories 
  \begin{equation}\label{eq:17}
     \Phi \eqcat M\boxtimes \mstrut _{\,\Xi }N. 
  \end{equation}
  \end{lemma}

  \begin{proof} 
 The 1-morphisms $b_+^R$ and $e\circ b_-$ are equivalent in $\bftb$, as
follows from \S\ref{subsubsec:4.2.5} or directly by inspection of
Figure~\ref{fig:13} and Figure~\ref{fig:17}.  Hence $b_+^R\circ b\mstrut _+$
is equivalent to $e\circ (b_-\amalg b_+)$; compare Figure~\ref{fig:15} and
Figure~\ref{fig:17}.  Apply~$\eF$ to deduce~\eqref{eq:17}.
  \end{proof}

Since the $\bigl(\Xi ,\TP(S^0) \bigr)$-bimodule~$N$ is invertible---by
Lemma~\ref{thm:28} it induces a Morita equivalence---from Lemma~\ref{thm:15}
we deduce that tensoring with~$\id_N$ induces a tensor equivalence
  \begin{equation}\label{eq:40}
     \alpha \:\End_\Xi (M)\xrightarrow{\;\;\boxtimes \id_N\;\;}
     \End_{\TP(S^0)}\bigl(\Phi \bigr).  
  \end{equation}
By~\eqref{eq:15} and~\eqref{eq:21} this is a tensor equivalence
  \begin{equation}\label{eq:23}
     \alpha \:\FSb\longrightarrow \TP(\Sb),
  \end{equation}
since $\TP(\Sb)$~is the Drinfeld center of~$\Phi$.  To complete the proof of
Theorem~\ref{thm:5} we need the following.

  \begin{lemma}[]\label{thm:29}
 $\alpha $~is a braided equivalence. 
  \end{lemma}

  \begin{proof}
 We have already proved that $\alpha $~is a tensor equivalence; it remains to
verify the \emph{condition} that $\alpha $~preserve the braiding.
Lemma~\ref{thm:28} states that the $\bigl(\Xi ,\TP(S^0) \bigr)$-bimodule~$N$
induces an isomorphism\footnote{Our conventions are that $N$~is a 1-morphism
$\TP(S^0)\to \Xi $, but in this proof we use categories of \emph{right} modules
rather than left modules, and so the convention applies oppositely.}
$\Xi \to\TP(S^0)$ in~$\FC$.  By the cobordism hypothesis this induces an
isomorphism 
  \begin{equation}\label{eq:79}
     \theta \:\Fd\xrightarrow{\;\;\sim \;\;}\TPd 
  \end{equation}
of topological field theories $\bft\to\FC$, where $\Fd(+)=\Xi $ and
$\TPd(+)=\TP(S^0)\cong \Phi \boxtimes\Phi \mo$.  ($\Fd$~is essentially the double
theory of Remark~\ref{thm:4}.)  Let $\eFd\:\bftb\to\FC$ be the extension
of~$\Fd$ with boundary theory~$\bd$ characterized by 
  \begin{equation}\label{eq:80}
     \bd(+)\:1\xrightarrow{\;\;\beta
     (S^0)\;\;}F(S^0)\xrightarrow{\;\;\sim\;\;}\Xi . 
  \end{equation}
Repeat the arguments above for~$\eFd$: introduce $\Xi ^d=\Xi \boxtimes \Xi
\mo$, the right $\Xi ^d$-module $M^d=F^d(e)$, the left $\Xi ^d$-module
$N^d=\beta ^d(S^0)$, and the tensor equivalence
  \begin{equation}\label{eq:81}
     \alpha ^d\:\Fd(\Sb)\xrightarrow{\;\;\boxtimes \id_{N^d}\;\;}
     \TPd(\Sb). 
  \end{equation}
We have $M^d=M\boxtimes M'$ and $N^d=N\boxtimes N'$, where the primed
$\Xi \mo$-modules are computed in the theory whose value on~$+$ is~$F(-)$.
Hence in the doubled theories \eqref{eq:17}~becomes 
  \begin{equation}\label{eq:82}
     \Phi \boxtimes \Phi \mo\eqcat (M\boxtimes M')\;\boxtimes _{\,\Xi \,\boxtimes
     \,\Xi \mo}\; (N\boxtimes N'), 
  \end{equation}
and $\alpha ^d$~tensors with~$\id\mstrut _{N\boxtimes N'}$.  Since tensoring
with~$N$ is the isomorphism~\eqref{eq:79} of theories, it follows that the
induced map on Drinfeld centers is~\eqref{eq:81}, i.e., $\alpha ^d=\theta
(\Sb)$.  Therefore, $\alpha ^d$~is a \emph{braided} tensor functor, and then
so too is its restriction to
  \begin{equation}\label{eq:83}
     F(\Sb) = F(\Sb)\boxtimes 1\;\subset \;F(\Sb)\boxtimes
     F(\Sb)^{\textnormal{rev}}\;\subset \;\Fd(\Sb). 
  \end{equation}
This completes the proof of Lemma~\ref{thm:29}, and so too of
Theorem~\ref{thm:5}. 
  \end{proof}

  \subsection{Proof of Theorem~\ref{thm:57}}\label{subsec:7.2}

For the remainder of this section we put in force the stronger hypotheses
that $\sC$~is fusion tensor cocomplete.  This allows the following relative
composition in terms of the notation of Definition~\ref{thm:69}.  If $H\in
\sC(y,z)$ is a right $\Phi $-module, then define 
  \begin{equation}\label{eq:93}
     H\ciF M := (H\circ M)\;\boxtimes_{\Phi \boxtimes \Phi \mo}\,\Phi . 
  \end{equation}

  \begin{proof}[Proof of Theorem~\ref{thm:57}]
 By the preceding it suffices to prove that $F$~is isomorphic to~$\TP $, and
by the cobordism hypothesis it suffices to construct an isomorphism
$F(+)\to\Phi $ in~$\sC$.  Consider 
  \begin{equation}\label{eq:84}
     \Phi \xrightarrow{\;\;\FR\;\;}1\xrightarrow{\;\;\bp\;\;}F\p, 
  \end{equation}
where $\FR$~is the regular right $\Phi $-module.  The fusion category $\Phi
=\bp^R\circ \bp$ acts on~$\FR$ on the left and on~$\bp$ on the right.  Define
$g\:\Phi \to F\p$ and $h\:F\p\to\Phi $ as 
  \begin{align}
     g &= \bp\ciF\FR \label{eq:85}\\
     h &= \FL\ciF\bp^R. \label{eq:86}
  \end{align}
We claim that $g$~and $h$~are inverse isomorphisms.  First,  
  \begin{equation}\label{eq:87}
     \begin{aligned} h\circ g &= \FL\ciF\bp^R\;\,\circ \;\,\bp\ciF\FR \\ &=
      \FL\ciF\Phi \ciF\FR \\ &=\FL_\Phi \end{aligned} 
  \end{equation}
is the $(\Phi ,\Phi )$-bimodule which represents $\id_{\Phi }$.  In the other
direction, 
  \begin{equation}\label{eq:88}
     \begin{aligned} g\circ h &= \bp\ciF\FR\;\,\circ \;\,\FL\ciF\bp^R \\
      &=\bp\ciF\bp^R\end{aligned} 
  \end{equation}
as an endomorphism of~$F\p$.   Dualize~$\bp^R$ to transpose~$g\circ h$ to a
1-morphism 
  \begin{equation}\label{eq:89}
     \ghT\:1\longrightarrow F\p\otimes F\m=F(S^0),
  \end{equation}
and since the dual to~$\bp^R$ is~$\beta(-)$ we find
  \begin{equation}\label{eq:90}
     \ghT = \beta (S^0)\boxtimes \mstrut _{\Phi \boxtimes \Phi \mo}\Phi .
  \end{equation}
Now $g\circ h=\id_{F\p}$ if and only if $\ghT$~is the coevaluation of a
duality pairing between~$F\p$ and~$F\m$.  Recall the coevaluation 1-morphism
$c$ in $\bft$ (Figure~\ref{fig:8}) which in the theory~ $\TP$ evaluates to
$\TP(c)=\Phi \:1\to \Phi \boxtimes\Phi \mo$.  Also, $\beta (S^0)=\eF(\fpm)$
is essentially the module~$N$ in~\eqref{eq:38}.  Thus rewrite~$\ghT$ as the
composition 
  \begin{equation}\label{eq:91}
     \ghT\:1\xrightarrow{\;\;\TP(c)\;\;}\TP(S^0)\xrightarrow
     {\;\;\;N\;\;\;}\Xi \xrightarrow[\;\;\phantom{N}\;\;]{\sim } F(S^0).  
  \end{equation}
Note Lemma~\ref{thm:28} implies that $N\:\TP(S^0)\to \Xi $~is an isomorphism.
Therefore, \eqref{eq:90}~is the desired coevaluation map and so $g\circ
h=\id_{F(+)}$. 
  \end{proof}

   \section{Application to physics\medskip}\label{sec:3}

A quantum mechanical system~$S$ is \emph{gapped} if its minimum energy is an
eigenvalue of finite multiplicity of the Hamiltonian, assumed bounded below,
and is an isolated point of the spectrum.  This notion generalizes to a
relativistic quantum field theory if we understand `spectrum' to mean the
spectrum of representations of the translation group of Minkowski spacetime.
A basic question:
  \begin{equation}\label{eq:8}
     \textnormal{Does a gapped system~$S$ admit a gapped boundary theory?} 
  \end{equation}
We argue heuristically that Theorems~\ref{thm:5} and~\ref{thm:57} gives an
obstruction for certain $(2+1)$-dimensional systems.  We remark that the
chiral WZW~model is a gapless boundary theory for Chern-Simons
theory~\cite{W2}, so at least for these systems a \emph{gapless} boundary
theory exists.
 
We reduce~\eqref{eq:8} to a question in topological field theory by
application of the following two heuristic physics principles:
 
\smallskip
      \begin{enumerate}

 \item the phase of a quantum system is determined by its low energy
behavior;
	
 \item the low energy physics of a \emph{gapped} quantum system is
well-approximated by a \tstar\ field theory.

      \end{enumerate}
\smallskip\noindent 
 For now we ignore the `${}^*$' in `\tstar'.  Principle~(1) seems
incontrovertible, though unproved, whereas (2)~is more problematic.  For
example, certain ``fracton'' lattice systems seem to have no continuum limit
as a standard field theory.
Nonetheless, (2)~appears to hold in many important cases; we simply assume it
here.  Applying these principles to both the bulk and boundary systems, the
general problem~\eqref{eq:8} reduces to a question in topological field
theory: Does a topological field theory~$F$ admit a nonzero boundary
theory~$\beta $?  If not, then the answer to~\eqref{eq:8} is `no'.  If the
topological field theory~$F$ does admit a nonzero boundary theory~$\beta $,
then we need a converse to~(2) to construct a gapped boundary theory.

  \begin{remark}[]\label{thm:7}
 We suspect that the answer to~\eqref{eq:8} depends only on the \emph{phase}
of~$S$, that is, its path component in a putative moduli stack of gapped
systems.
  \end{remark}

We now explain the `${}^*$' in `\tstar' by means of an example that is a main
focus of interest.  The starting point is a quantum field theory, though one
can imagine a lattice model in its place.  Namely, let $S$~be
$(2+1)$-dimensional Yang-Mills theory with a nondegenerate Chern-Simons term.
The latter gives the gauge field a mass, which means that the system is
gapped.  Its low energy physics is thought to be well-approximated by a pure
Chern-Simons theory~$\Gamma $.  Observe that $S$~in its Wick-rotated form is
a theory of manifolds equipped with an orientation and Riemannian metric.  In
other words, it is a functor on a \emph{geometric} bordism category of
oriented Riemannian manifolds.  The naive expectation is that~$\Gamma $ ~is a
functor on the same bordism category, and this is the case.  In fact, as
discussed by Witten~\cite[\S2]{W1}, the dependence on the Riemannian metric
is mild: ``locally'' $\Gamma $~is the tensor product of a \emph{topological}
field theory~$F$ and an \emph{invertible} non-topological field theory~$\ac
$, where $c\in \RR$ is the central charge.  More precisely, the pullback
of~$\Gamma $ to the bordism category of 3-framed Riemannian manifolds splits
as $\Gamma \cong F\otimes \ac $.  The theory~$F$ is topological---it does not
depend on a Riemannian metric.  It is an example of an RT theory as described
in~\S\ref{sec:1}.  The invertible dependence of~$\Gamma $ on the metric
through~$\ac $ is the `${}^*$' in `\tstar'.
 
The invertible theory~$\ao$ descends to a theory~$\bao$ with domain the
bordism category of \emph{oriented} Riemannian manifolds.  Its partition
function on a closed oriented Riemannian 3-manifold~$X$ is the exponentiated
$\eta $-invariant $\exp(2\pi i\eta _X/2)$, where $\eta _X\in \RR/2\ZZ$ is the
secondary invariant associated to the signature operator~\cite{APS}.  The
deformation class of~$\bao$ is a generator of the abelian group $[MTSO,\Sigma
^4I\ZZ]$ of invertible theories, and at least conjecturally it can
be constructed using generalized differential cohomology.  (We refer
to~\cite{FH,F2} for notation and details.)  The deformation class of the
lift~$\ao\mstrut $ of~$\bao$ to 3-framed manifolds vanishes, since 
  \begin{equation}\label{eq:9}
     [MTSO,\Sigma ^4I\ZZ]\longrightarrow [S^0,\Sigma ^4I\ZZ] 
  \end{equation}
is the zero map.  In terms of the differential cohomology construction, the
equivalence class of~$\ao$ belongs to the subgroup of topologically trivial
theories, so is defined by a universal 3-form: one-third the ``gravitational
Chern-Simons term''.  Then for any~$c\in \RR$, the family $\alpha
_{tc},\;0\le t\le1$, is an explicit deformation of the trivial theory
to~$\ac$.  Put differently, it is a ``nonflat trivialization'' $\beta
\:1\xrightarrow{\;\cong \;} \tau \mstrut _{\le2}\ac$ of the truncation~$\tau
\mstrut _{\le2}\ac$.\footnote{An example of a nonflat trivialization is a
not-necessarily-flat section of a circle bundle with connection.  The notion
of nonflat trivialization should be part of an axiomatization of
\emph{families} of field theories.}  In other words, $\ac$~is equipped with a
boundary theory; compare~\eqref{eq:7}.  Therefore, \tstar\ boundary theories
for~$\Gamma $ correspond to topological boundary theories for~$F$.
 
Theorems~\ref{thm:5} and~\ref{thm:57} give an obstruction to the existence of
a nonzero topological boundary theory for~$F$: the theory~$F$ must be of
Turaev-Viro type.  If not, then the heuristics in this section suggest that
there are no gapped boundary theories for Yang-Mills plus Chern-Simons, nor
for a lattice system meant to represent the same phase.  It would be
interesting to construct a gapped boundary theory for Yang-Mills plus
Chern-Simons in case $F$~is of Turaev-Viro type.

  \begin{remark}[]\label{thm:6}
 One implicit assumption in Principle~(2) is that a gapped quantum system
exhibits relativistic invariance in the long-range approximation.  The
Wick-rotated manifestation is the fact that the domain bordism category is
made from manifolds whose tangential structure does not break~$O_3$ further
than the subgroup~$SO_3$.  In particular, a 3-framing breaks relativistic
invariance.  Here the 3-framing is introduced to isolate the metric
dependence of~$\Gamma $ to the invertible theory~$\ac $; the physically
relevant theory has $\SO_3$-invariance.
  \end{remark}

\appendix

   \section{Bordism Multicategories}\label{sec:4}

In this appendix we give the precise definitions behind the descriptions
in~\S\ref{subsec:2.1} and the pictures throughout~\S\ref{sec:2}
and~\S\ref{sec:7}.  Complete constructions of the bordism multicategory
appear in~\cite{CS,AF} among other references.  In these approaches an object
or morphism is equipped with a \emph{global} map to a cube or stratified
ball, and this data is used to define composition laws.  Our limited goal
here is to define objects and morphisms in~$\Bord_n$ with minimal data
\emph{localized at the boundary}; they too lead to composition laws, though
we do not pursue the latter.\footnote{Nor do we specify collaring data which
would give a smooth structure on compositions.}  Underlying a bordism is a
\emph{manifold with corners}, so we begin with a quick review
in~\S\ref{subsec:4.1}.  Then in~\S\ref{subsec:4.2} we specify the additional
data required for a morphism in the bordism multicategory.  A similar
discussion is in~\cite[\S8.1]{CS}, based in part on~\cite{La}.  We
incorporate ``colored boundaries''---morphisms in~$\Bord_{n,\partial
}$---in~\S\ref{subsec:4.3}.

  \begin{remark}[Some conventions]\label{thm:77}
 In this paper we use \emph{topological} bordism multicategories, but we take
inspiration from \emph{geometric} bordism multicategories.  In a geometric
bordism category---the domain of a \emph{non-topological} field theory---a
$k$-morphism is a $k$-dimensional compact manifold~$X$ with corners which
comes equipped with embeddings
  \begin{equation}\label{eq:94}
     X^{(0)}\supset X^{(1)}\supset \dots \supset X^{(n-k)}=X, 
  \end{equation}
where $\dim X^{(i)}=n-i$ and $X^{(0)},\dots ,X^{(n-k-1)}$ are \emph{germs} of
smooth manifolds.  The successive normal line bundles are oriented: these
orientations are ``arrows of time''.  The trivial line bundles in the
stabilization~\eqref{eq:52} are what remains of this structure in the
\emph{topological} bordism category, and their standard orientations are the
remnants of the arrows of time into the germs.\footnote{We could instead
specify a completion of~$TX$ to a flag: $TX\subset E^{(n-k+1)}\subset \dots
\subset E^{(n)}$ and orient the successive quotient line bundles, but we opt
for the simpler stabilization~\eqref{eq:52}.}  Our indexing of these trivial
line bundles is $-1,-2,\dots ,-(n-k)$, reading from left to right
in~\eqref{eq:52} and use the standard orientations.  
 
A word about `time' and `space'.  In topological field theory, which is
modeled on \emph{Wick-rotated} field theory, there is no notion of time
versus space: the passage from Lorentz geometry to Euclidean geometry
discards the unique time dimension in favor of an additional space dimension.
Still, the codimension~1 boundary of an $n$-manifold in~$\Bord_n$ plays the
role of a spatial slice, hence its normal bundle can reasonably be said to
represent time, for example as depicted by the arrows in Figure~\ref{fig:1}.
In higher codimension, for example the double-headed arrows in
Figure~\ref{fig:2}, the interpretation as a ``time'' is only figurative.  

Our convention is to order line decompositions of the inflated tangent bundle
by codimension from the top dimension of the theory, in order of increasing
codimension, and we use the labels $-1, -2,\dots ,-n$ for the summands.
(See~\S\ref{subsubsec:4.2.2} and~\S\ref{subsubsec:4.2.3}.)
\emph{Heuristically}, the first direction is ``temporal'' and the remaining
$n-1$~directions are ``spatial''.

In the main text we embed $\btf\to \bft$, which facilitates the pictures;
see~\S\ref{subsubsec:4.2.6}.  As far as we know, this does not correspond to
anything in physics; it is a convenient mathematical device.  The direction
we add is ``temporal'' in our conventions, but again that choice has no
physical meaning.  By contrast, dimensional reduction---say, along a
circle---is effected via a map
$\Bord_{n-1}\xrightarrow{\;\times\cir\;}\Omega\!  \Bord_n$, and this
Cartesian product map is easily checked to be compatible with our labeling
conventions.

In~\S\ref{subsec:4.3} we bring in ``colored'' boundaries.  They model
boundaries in space, not boundaries in time, and so the transverse direction
is ``spatial''; see~\eqref{eq:97}. 
  \end{remark}

  \subsection{{Manifolds with corners}}\label{subsec:4.1}

There are several definitions and a long history of the subject of manifolds
with corners, both of which are reviewed in Joyce~\cite{J}.  He develops the
theory in detail, and we defer to his paper and the references therein for
details. 
 
Fix~$k\in \ZZ^{\ge0}$.  A neighborhood of a point in a smooth $k$-manifold is
modeled by an open set in real affine space~$\AA^k$.  Similarly, a
neighborhood of a point in a $k$-manifold with corners is modeled by an open
set in
  \begin{equation}\label{eq:43}
     \Alez = \left\{ (x^1,\dots ,x^k)\in \AA^k: x^i\le0 \right\} . 
  \end{equation}
The usual notions of chart and atlas generalize accordingly.  A point
$x=(x^1,\dots ,x^k)\in \Alez$ has \emph{depth}~$j\in \ZZ^{\ge0}$ if precisely
$j$~of its coordinates vanish.  The depth is invariant under diffeomorphism
of open sets in~$\Alez$, so is a well-defined function 
  \begin{equation}\label{eq:44}
     \textnormal{depth}\:M\longrightarrow \ZZ^{\ge0} 
  \end{equation}
on a manifold~$M$ with corners.  For~$j\in \{0,\dots ,k\}$, let $\mM_{-j}\subset
M$ denote the $(k-j)$-manifold of points in~$M$ of depth~$j$, and let
$M_{-j}\subset M$ be the closure of~$\mM_{-j}$.  If the maximum value
of~\eqref{eq:44} is~$d\in \{0,\dots ,k\}$, we say $M$~is a \emph{manifold
with corners of depth~$\le d$} and we call~$d$ the \emph{depth} of~$M$.  If
$d=1$, then $M$~is a \emph{manifold with boundary}.  There is a canonical
filtering and partition
  \begin{equation}\label{eq:45}
     \begin{aligned} M &= M_0\supset M_{-1}\supset \cdots\supset M_{-d} \\ &=
      \mM_0\;\amalg\, \mM_{-1}\;\amalg\, \cdots\amalg\, \mM_{-d}\end{aligned} 
  \end{equation}
A \emph{face} of~$M$ is the closure of a component of~$\mM_{-1}$.

The tangent space~$T_mM$ to~$M$ at~$m\in M$ is a $k$-dimensional real vector
space.  If $m$~has depth~$j$, then there are $j$ transverse hyperplanes
$H_1,\dots ,H_j\subset T_mM$ and orientations of the lines~$T_mM/H_i$: the
positively oriented direction leads out of~$M$.
 
Variant definitions of `manifold with corners' include global constraints
and/or data in addition to the local normal form.  For example, one might
require that every point of depth~$j$ lie in $j$~distinct faces.  The bigon
in Figure~\ref{fig:21} satisfies this condition, whereas the teardrop does
not.  There are more stringent possible global specifications;
see~\cite[Remark~2.11]{J} and the references therein.  The extra data we
introduce in~\S\ref{subsec:4.2} to define a morphism in a bordism
multicategory endows the underlying manifold with corners with the
data/constraints to be of these more restricted types.

  \begin{figure}[ht]
  \centering
  \includegraphics[scale=1]{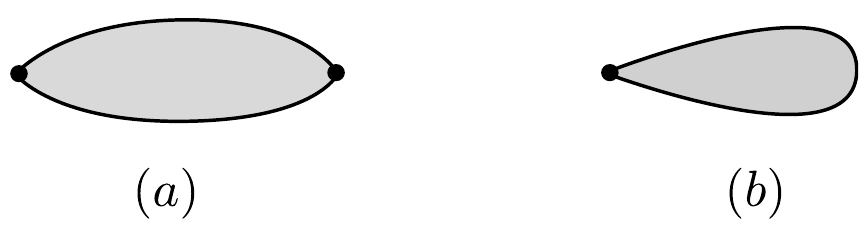}
  \vskip -.5pc
  \caption{A bigon~($a$) and a teardrop~($b$)}\label{fig:21}
  \end{figure}

There are two distinct notions of the boundary of a manifold with corners.
For our purposes we define $\partial M=M_{-1}$ as the closed subset of points
of positive depth.  This is not generally a manifold with corners, as
Figure~\ref{fig:21} illustrates.  However, there is a ``blow up'' which
surjects onto~$\partial M$ and which is a manifold with corners;
see~\cite[Definition~2.6]{J}.

  \subsection{$k$-morphisms in~$\Bord_n$}\label{subsec:4.2}

   \subsubsection{The definition}\label{subsubsec:4.2.1}
 Fix~$n\in \ZZ^{>0}$.  For $k\in \{0,\dots ,n\}$ we specify the data of a
$k$-morphism in~$\Bord_n$.  (For~$k=0$ it is an object in~$\Bord_n$.)
Tangential structures are introduced in~\S\ref{subsubsec:4.2.4}.

  \begin{definition}[]\label{thm:37}
 Fix~$n,k$ as above and suppose $d\in \{0,\dots ,k\}$.  Let $X$~be a compact
$k$-dimensional manifold with corners of depth~$\le d$.  The data of a
\emph{$k$-morphism of depth~$d$} on~$X$ are:

 \begin{enumerate}[label=\textnormal{(\roman*)}]

 \item if~$d\ge1$, closed $(k-d)$-manifolds $\Xz d,\Xo d$, not both empty;

 \item if~$d\ge2$, recursively for $j=d-1,d-2,\dots ,1$ compact
$(k-j)$-manifolds $\Xz j,\Xo j$ with corners of depth~$\le d-j$ equipped with
diffeomorphisms
  \begin{equation}\label{eq:46}
     \varphi ^\delta _{-j}\:\,\Xz{(j+1)}\;\;\cup\;\; [0,1]\times \left\{
     \Xz{(j+2)}\amalg 
     \Xo{(j+2)} \right\} \;\;\cup\;\; \Xo{(j+1)}\longrightarrow \partial (\Xd
     j),\qquad\delta \in \{0,1\}, 
  \end{equation}
where the unions are along $\{0\}\times \{ \Xz{(j+2)}\amalg \Xo{(j+2)}
\}$ and $\{1\}\times \{ \Xz{(j+2)}\amalg \Xo{(j+2)} \}$,
respectively;

 \item if~$d\ge1$, a diffeomorphism 
  \begin{equation}\label{eq:47}
     \varphi \mstrut _0\:\,\Xz{1}\;\;\cup\;\; [0,1]\times \left\{
     \Xz{2}\amalg \Xo{2}\right\} \;\;\cup\;\; \Xo{1}\longrightarrow \partial X,
  \end{equation}
where the unions are along $\{0\}\times \{ \Xz{2}\amalg \Xo{2} \}$ and
$\{1\}\times \{ \Xz{2}\amalg \Xo{2} \}$, respectively.
 \end{enumerate}
  \end{definition}

  \begin{figure}[ht]
  \centering
  \includegraphics[scale=.6]{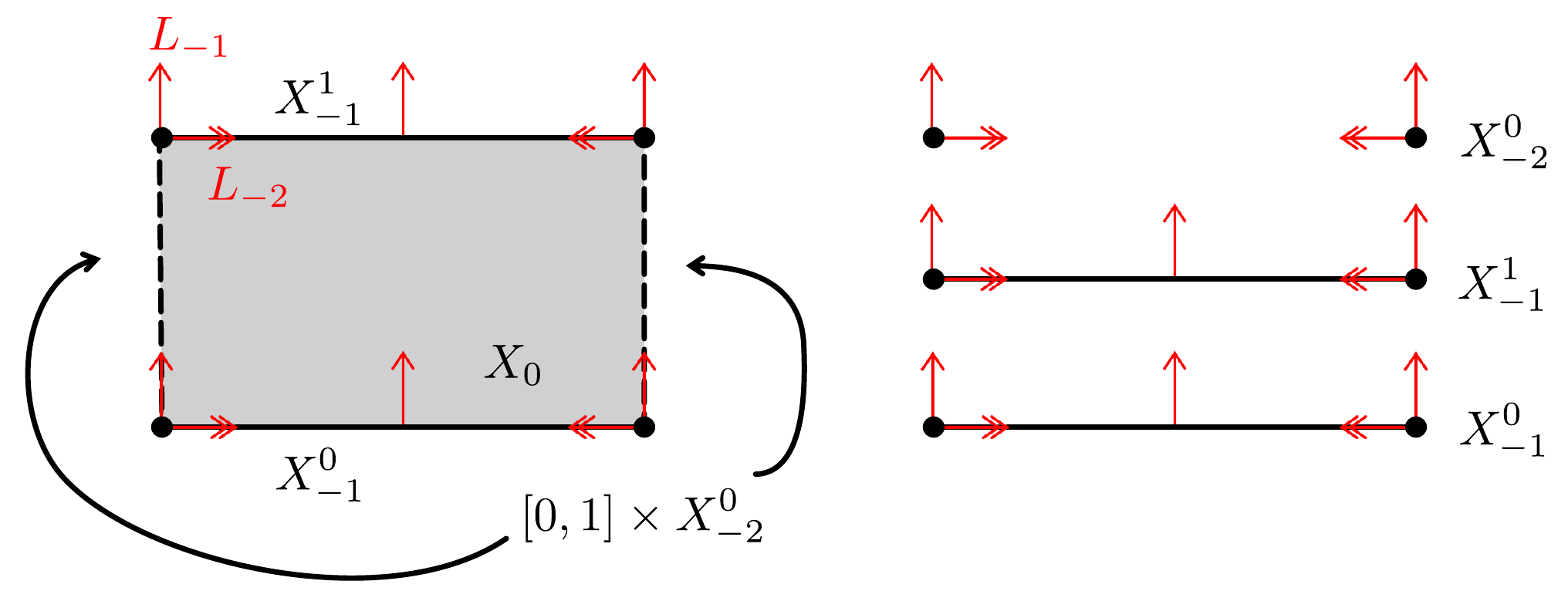}
  \vskip -.5pc
  \caption{A 2-morphism of depth 2}\label{fig:22}
  \end{figure}

  \begin{remark}[]\label{thm:38}
\ 
 \begin{enumerate}

 \item See Figure~\ref{fig:22} for an example of a 2-morphism of depth~2.  In
that example $X^0_{-2}$~consists of two points, $X^1_{-2}=\emptyset ^0$~is
the empty 0-manifold, and $X^0_{-1}\approx X^1_{-1}$~are closed intervals.

 \item To interpret~\eqref{eq:46} for~$j=d-1$, set $\Xd{(d+1)}=\emptyset $,
$\delta \in \{0,1\}$.

 \item In the categorical interpretation, $X$~is a $k$-morphism with source
and target the empty $i$-manifold~$\emptyset ^i$, $i\in \{0,\dots ,k-d-1\}$;
and source and target $i$-morphisms $\Xz{(k-i)},\Xo{(k-i)}$, respectively,
$i\in \{k-d,\dots ,k-1\}$.

 \item If $d\ge1$, the embeddings~$\varphi ^\delta _{-j},\varphi \mstrut _0$,
$j\in \{1,\dots ,d-1\}$, $\delta \in \{0,1\}$, combine to embeddings
  \begin{equation}\label{eq:48}
     \psi ^\delta _{-j}\:[0,1]^{j-1}\times \Xd j\longrightarrow \partial X,\qquad j\in \{1,\dots ,d\},\quad \delta \in \{0,1\}. 
  \end{equation}
For $j\in \{1,\dots ,d-1\}$ let $\mps$ denote the restriction of~$\psi ^\delta
_{-j}$ to $[0,1]^{j-1}\times \mX^\delta _{-j}$, and set $\mathring{\psi }^\delta
_d=\psi ^\delta _d$.  Then $\partial X$~is the disjoint union of the images
of~$\mps$, $j\in \{1,\dots ,d\}$, $\delta \in \{0,1\}$.  Heuristically, the
bordism is ``constant'' on $\mps\bigl([0,1]^{j-1}\times \{x\}\bigr)$,
$x\in \mX^\delta_{-j}$. 

 \item The pictures in~\S\ref{sec:2} and~\S\ref{sec:7} are of 1-~and
2-morphisms of various depths.  The images of the embeddings~\eqref{eq:48}
for~$j=2$ are depicted as dashed edges, as described
in~\S\ref{subsubsec:2.1.5}.

 \end{enumerate}
  \end{remark}

   \subsubsection{The tangent filtration}\label{subsubsec:4.2.2}
 The structure described in Definition~\ref{thm:37} has a tangential
implication.  Namely, let $X$~be a $k$-morphism of depth~$d$, and suppose
$x\in \bX$.  Choose the unique~$j,\delta $ and $t^1,\dots ,t^{j-1}\in [0,1]$,
$\mx\in \mX^\delta _{-j}$ such that $x=\mps(t^1,\dots ,t^{j-1};\mx)$.  Then
$T_xX$~has a decreasing filtration
  \begin{equation}\label{eq:49}
     T\mstrut _xX=T_{x,0}X\supset T_{x,-1}X\supset \cdots\supset
     T_{x,{-j}}X=T\mstrut _{\mx}\mX\mstrut _{-j} 
  \end{equation}
in which 
  \begin{equation}\label{eq:50}
     T_{x,-i}X = d\mps\bigl(0^{i-1}\oplus \RR^{j-i}\oplus
     T_{\mx}\mX_{-j}\bigr),\qquad i\in \{1,\dots ,j-1\}. 
  \end{equation}
The associated graded is a sum of real lines, which we number by codimension
\emph{in the theory},\footnote{By contrast, subscripts in~$X^\delta _{-j}$
and~$T\mstrut _{x,-i}X$ are codimensions \emph{in}~$X$, so count down
from~$k=\dim X$.}  i.e., count down from~$n$:
  \begin{equation}\label{eq:51}
     \Lx{n-k+1}\oplus \cdots\oplus \Lx{n-k+j}.
  \end{equation}
Orient~$\Lx{n-k+j}$ so that the positive direction leads into~$X$ if~$\delta
=0$ (incoming/source morphism) and leads out of~$X$ if~$\delta =1$
(outgoing/target morphism).  Orient $\Lx{n-k+i}$, $i\in \{1,\dots ,j-1\}$, so
that the positive direction points towards increasing~$t^{j-i}$.  The
orientations are constant over the image of~$\mps$.  Moreover,
Definition~\ref{thm:37}(ii) and~(iii) ensure that the orientations are
consistent as we move among the images of the various~$\mps$.

   \subsubsection{The inflated tangent bundle of a $k$-morphism}\label{subsubsec:4.2.3}
 Define the ``inflated tangent bundle'' 
  \begin{equation}\label{eq:52}
     \tT X=\underbrace{\uR\oplus \cdots\oplus \uR}_{\textnormal{$n-k$
     times}}\;\oplus \;TX\longrightarrow      X, 
  \end{equation}
where $\uR\to X$ is the constant line bundle with fiber~$\RR$.  The
orientations of the line bundles in~\eqref{eq:51} and the standard
orientations on the $n-k$ trivial line bundles in~\eqref{eq:52} are the
``arrows of time'' discussed in~\S\ref{subsubsec:2.1.1}.  We label them by
codimension: $-1,-2,\dots ,-(n-k)$.

  \begin{remark}[]\label{thm:39}
 \ 
 \begin{enumerate}

 \item For the 2-morphisms of various depths depicted in Figure~\ref{fig:22}
and in~\S\S\ref{sec:2},\ref{sec:7}, in~$\Bord_2$ the single-headed arrows
correspond to codimension~$i=1$ and the double-headed arrows correspond to
codimension~$i=2$.  In~$\Bord_3$ the codimensions should be shifted to $i=2$
and~$i=3$.

 \item Fix $i\in \{1,\dots ,n\}$.  Then the $i^{\textnormal{th}}$~duality of
the $O(1)^{\times n}$-action discussed in~\S\ref{subsubsec:2.1.6} acts
trivially if~$i\le n-k$ and exchanges~$X^0_{n-k-i}$ and $X^1_{n-k-i}$
if~$i>n-k$.  If there is a tangential structure (\S\ref{subsubsec:4.2.4}),
then for~$i\le n-k$ the tangential structure is pulled back under the
reflection in the inflated tangent bundle~\eqref{eq:52} which reverses the
sign on the summand $\uR_{-i}$.

 \item In Figure~\ref{fig:22} the arrows of time on~$X^0_{-2}$, drawn on the
upper right of the figure, carry no meaning; they merely embed the trivial
lines in~\eqref{eq:52} in the plane of the figure.  Similarly for the
single-headed arrow of time in~$X^\delta _{-1}$, $\delta \in \{0,1\}$ on the
right hand side of Figure~\ref{fig:22}.

 \item Each $\Xd j$, $j\in \{1,\dots ,d\}$, $\delta \in \{0,1\}$, has the
structure of a $(k-j)$-morphism~$Y$ of depth~$d-j$ with $Y^\epsilon
_{-i}=X^\epsilon _{-(j+i)}$, $i\in \{1,\dots ,d-j\}$, $\epsilon \in \{0,1\}$.

 \end{enumerate} 
  \end{remark}

   \subsubsection{Tangential structures}\label{subsubsec:4.2.4}
 Let $\rho _n\:\sX_n\to \BGL$ be a continuous map.  The choice of classifying
map $X\to\BGL$ for the inflated tangent bundle $\tT X\to X$ of a
$k$-dimensional manifold~$X$ with corners, $k\le n$, is a contractible choice
we assume given.  Then a \emph{tangential structure of type}~$\rho _n$
\emph{on}~$X$ is a lift of that classifying map to~$\sX_n$.  We can use rigid
models instead, such as for orientations, spin structures, or $n$-framings.
An isomorphism $X'\to X$ of manifolds with corners is a diffeomorphism $\Phi
\:X'\to X$ together with a linear isomorphism $\tT X'\to\Phi ^*\tT X$ and a
homotopy of the tangential structure on~$X'$ to the pullback of the
tangential structure on ~$X$.  If rigid models are employed, the homotopy may
be replaced by a more rigid alternative, which may be a combination of
conditions and data.
 
There is a variant of Definition~\ref{thm:37} for tangential structures of
type~$\rho _n$: each manifold~$X^\delta _{-j}$ with corners is equipped with
a tangential structure of type~$\rho _n$ and the diffeomorphisms $\varphi
^\delta _{-j},\varphi \mstrut _0$ are lifted to isomorphisms in the sense of
the previous paragraph.  The tangential structure on $[0,1]\times Y$ is taken
to be that on~$Y$, extended to be constant along the $[0,1]$-direction.  An
isomorphism $\Phi \:X'\to X$ of $k$-morphisms of depth~$d$ is an isomorphism
$X'\to X$ of manifolds with corners and tangential structures, and a
collection of isomorphisms $(X')^\delta _{-j}\to X^\delta _{-j}$ of manifolds
with corners and tangential structures, compatible with $(\varphi
')_{-j}^\delta ,\varphi _{-j} ^\delta $ and $\varphi '_0,\varphi \mstrut _0$.

   \subsubsection{Duals and adjoints}\label{subsubsec:4.2.5}
 An object in $\Bord_n$ is a finite set of points~$X$; the tangential
structure, if present, is on the trivial vector bundle $X\times \RR^n$.  The
dual object~$X\dual$ consists of the same data, but with tangential structure
pulled back via reflection $(\xi ^1,\dots ,\xi ^{n-1},\xi ^n)\mapsto (\xi
^1,\dots ,\xi ^{n-1},-\xi ^n)$ on~$\RR^n$.  Evaluation and coevaluation
morphisms are constructed from $[0,1]\times X$; see Figure~\ref{fig:8}.  The
dual of a $k$-morphism is constructed by exchanging~$X^0_{-k}$
and~$X^1_{-k}$.  More generally, a closed $k$-manifold~$X$  is an
object in~$\Omega ^k\Bord_n$.  Its dual has tangential structure pulled back
along reflection in~$\uR_{-(n-k)}$ in the inflated tangent bundle.
 
If $1\le k\le n-1$, then a $k$-morphism~$X$ in~$\Bord_n$ has both a right
adjoint~$X^R$ and a left adjoint~$X^L$.  For our purposes in this paper, we
restrict to $k$-morphisms of depth~1: manifold with boundary (no corners).
We now specify data\footnote{The triple consisting of~$X^A$, a unit, and a
counit, is unique up to unique isomorphism in an appropriate 2-category
truncation of $\Bord_n$.  Here we define~$X^A$ and only give an indication of
the construction of the unit and counit of the adjunctions.} for these
objects~$X^A$, where $A=R$ for the right adjoint and $A=L$ for the left
adjoint.  Let $X^A=X$ as a manifold with boundary.  Reverse the arrows of
time on the codimension one strata: set $(X^A)_{-1}^\delta = X_{-1}^{1-\delta
}$ for $\delta \in \{0,1\}$.  Construct a diffeomorphisms $\varphi ^A_0$ from
the corresponding diffeomorphism in the data of~$X$.  For unoriented bordisms
that is a complete specification of~$X^A$; in particular, right and left
adjoints agree.  If a tangential structure is present, define a tangential
structure on~$X^A$ according to the following procedure.  Choose collar
neighborhoods
  \begin{equation}\label{eq:60}
     \begin{aligned} c(X^A)_{-1}^0 &\approx \;\,[0,1) \;\times (X^A)_{-1}^0 \\
      c(X^A)_{-1}^1 &\approx (-1,0] \times (X^A)_{-1}^1 \\ \end{aligned} 
  \end{equation}
and let $t$~be the coordinate in the intervals $[0,1)$, $(-1,0]$.  In the
collar the tangent bundle splits off a trivial line bundle: 
  \begin{equation}\label{eq:61}
     T\bigl(c(X^A)_{-1}^\delta \bigr) \cong \uR_{-(n-k+1)}\;\oplus\;
     T(X^A)_{-1}^\delta . 
  \end{equation}
Orient the summand $\uR_{-(n-k+1)}$ according to the \emph{opposite} of the
orientation of $L_{-(n-k+1)}$ in~\eqref{eq:51}, with which it is identified
at~$t=0$.  In other words, orient it according to the arrow of time in~$X^A$.
Let $\uR_{-(n-k)}$ denote the trivial summand in the inflated tangent
bundle of~$(X^A)^\delta _{-1}$ that corresponds to codimension~$n-k$.  Let
  \begin{equation}\label{eq:62}
     V=\uR_{-(n-k)}\;\oplus \;\uR_{-(n-k+1)} 
  \end{equation}
with its direct sum orientation; $V$~ is a direct summand of the inflated
tangent bundle $\tT\bigl(c(X^A)_{-1}^\delta \bigr)$ in the collar.  Transport
the tangential structure from~$X$ to~$X^A$ as follows.  At $t=0$ transport
via the hyperplane reflection $\id\oplus -\id$ on~\eqref{eq:61}: flip the
sign on~$\uR_{-(n-k+1)}$.  Moving in the collars~\eqref{eq:60} in~$X^A$
from~$t=0$ to $t=(-1)^\delta /2$ transport via a path of rotations in~$V$
which begins at~$\id_V$ and ends at~$-\id_V$ and
turns\footnote{Counterclockwise rotation turns the positive direction in the
first summand of~\eqref{eq:62} towards the positive direction in the second
summand.}
  \begin{equation}\label{eq:63}
     \left\{\Large\substack{\textnormal{clockwise}\\[3pt]
     \textnormal{counterclockwise}}\right\} \textnormal{ according as } A =
     \left\{\Large\substack{R \\[3pt] L}\right\}. 
  \end{equation}
For $|t|\ge 1/2$ in the collar and also outside the collar,
transport the tangential structure from~$X$ to~$X^A$ via the hyperplane
reflection in the extended tangent bundle which flips the sign on~
$\uR_{-(n-k)}$. 

  \begin{figure}[ht]
  \centering
  \includegraphics[scale=.6]{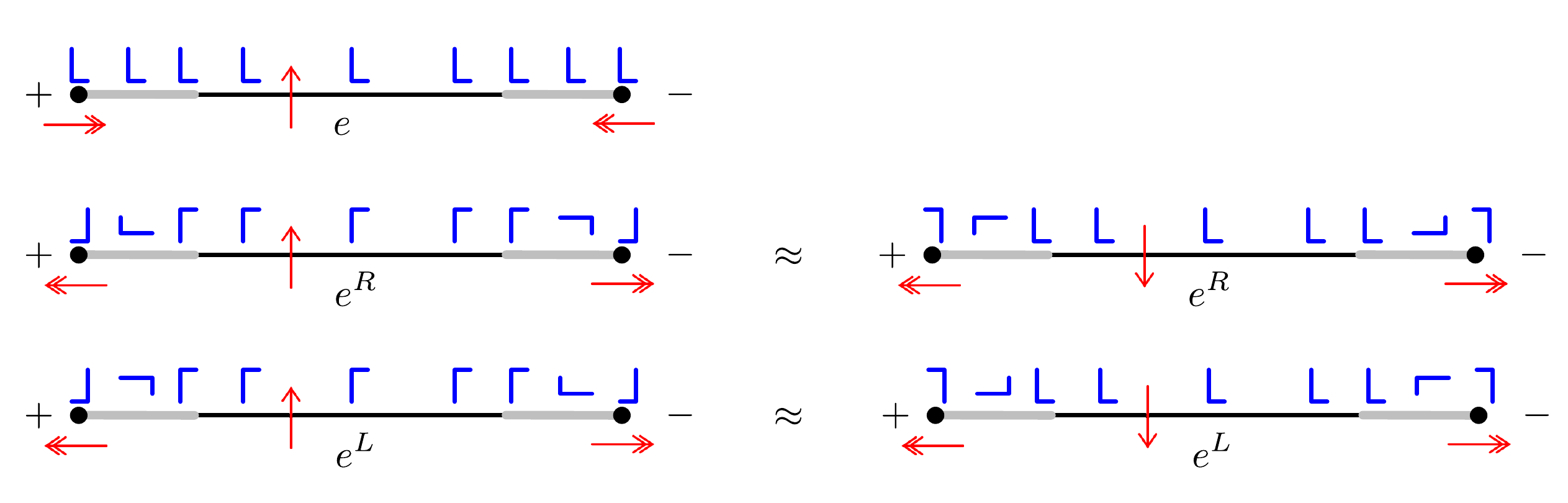}
  \vskip -.5pc
  \caption{Right and left adjoints of evaluation in~$\btf$}\label{fig:26}
  \end{figure}

Figure~\ref{fig:26}, a reworking of Figure~\ref{fig:8}, illustrates the right
and left adjoints of the evaluation map $e\:+\amalg -\to \emptyset ^0$ in
$\btf$.  In these figures the single-headed red arrows indicate the positive
direction in the summand~$\uR_{-2}$, which is necessary for the framing to
have meaning in the figure; the double-headed red arrows indicate the
orientation of~$L_{-2}$, determined by whether a boundary component is
incoming or outgoing.  The counterclockwise vs.\ clockwise
specification~\eqref{eq:63} can be checked in four adjunctions: $e^A$~as an
adjoint of~$e$ and $e$~as an adjoint of~$e^A$, each for~$A=R,L$.

  \begin{remark}[]\label{thm:44}
 A useful isomorphic representative of~$X^A$ is obtained via the identity
diffeomorphism of~$X^A$, lifted to the inflated tangent bundle~$\tT X^A$ as
the hyperplane reflection which is $-\id$ on~$\uR_{-(n-k)}$.  This is
illustrated by the diffeomorphisms in Figure~\ref{fig:26}, under which both
the framings and arrows of time have been transported.
  \end{remark}

  \begin{figure}[ht]
  \centering
  \includegraphics[scale=.7]{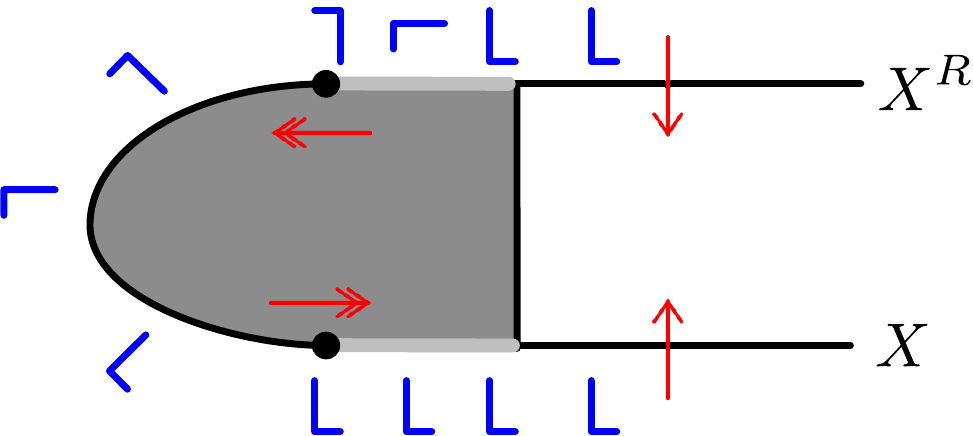}
  \vskip -.5pc
  \caption{The first move towards a counit $X\circ X^R\to\id$}\label{fig:27}
  \end{figure}

The units and counits of the adjunctions may be constructed in two stages,
which we now sketch.  The first step for the counit $X\circ X^R\to\id$ is
illustrated in Figure~\ref{fig:27}.  Glue~$(X^R)_{-1}^1$ to
$X_{-1}^0=(X^R)_{-1}^1$ by adjoining a cylinder; then the tangential
structures are such that we can push along a 2-disk in the direction of the
vector space~$V$ in~\eqref{eq:62} to construct a $(k+1)$-dimensional bordism
which eliminates the cylinder and the collared neighborhoods of those
boundary components.  For the second stage, choose a Morse function
$f\:X^R\,\setminus \,c(X^R)_{-1}^1\to[0,1)$ with $f\inv (0)=(X^R)_{-1}^0$,
and use $2-f$ as a Morse function on $X\setminus c(X_{-1}^0)$.  Do surgeries to
cancel corresponding critical points and so produce the desired
$(k+1)$-dimensional bordism to the identity $k$-morphism (cylinder) on
$(X^R)_{-1}^0$.  See Figure~\ref{fig:14} for an example of a unit and counit,
though with trivial second stage.

   \subsubsection{The inclusion $\bnon$}\label{subsubsec:4.2.6}
 If $k\in \{0,\dots ,n-1\}$ and $d\in \{0,\dots ,k\}$, then a $k$-morphism of
depth~$d$ in~$\bno$ is also a $k$-morphism of depth~$d$ in~$\bn$; see
Definition~\ref{thm:37}.  The inflated tangent bundle~\eqref{eq:52} in~$\bn$
has an extra direction, of course.  Note that the inclusion $\bnon$ increases
codimensions in the theory (from the top dimension) by~1.  We can define
general maps of tangential structures from~$\Bord_{n-1}$ to~$\Bord_n$, as
in~\S\ref{subsubsec:4.3.3} below.  For the map $\bno^{\textnormal{fr}}\to
\bn^{\textnormal{fr}}$ of framed bordism relevant to this paper, if $X$~is an
$(n-1)$-framed $k$-morphism in $\bno$, then the induced $n$-framing on~$X$
regarded as a $k$-morphism in~$\bn$ appends the standard basis vector on the
additional summand~$\uR$ in~\eqref{eq:52}.  It has label~$-1$.

  \begin{remark}[]\label{thm:75}
 The $+$ point in~$\Bord_\ell ^{\textnormal{fr}}$, $\ell \in \{n-1,n\}$, is
the manifold $X=\pt$ with the standard framing on 
  \begin{equation}\label{eq:95}
     \tT X=\underbrace{\uR\oplus \cdots\oplus \uR}_{\textnormal{$\ell $
     times}}. 
  \end{equation}
The $-$~point, defined to be dual to the $+$~point in~$\Bord_\ell
^{\textnormal{fr}}$, has the opposite framing on the \emph{last} summand
in~\eqref{eq:95}.  (The summands are ordered by increasing codimension, so
are labeled $-1,-2,\dots ,-\ell $.)  Then under the inclusion
$\bno^{\textnormal{fr}}\to \bn^{\textnormal{fr}}$ we have $+\mapsto +$ and
$-\mapsto -$: since the extra direction has label~$-1$, so is
\emph{prepended} to~\eqref{eq:95}, the last of the $n-1$ directions
in~$\Bord_{n-1}^{\textnormal{fr}}$ maps to the last of the $n$~directions in
$\Bord_n^{\textnormal{fr}}$. 
  \end{remark}

  \begin{remark}[]\label{thm:78}
 More generally, our conventions about duals and adjoints are preserved under
the map $\Bord_{n-1}\to \Bord_n$.  This justifies computing adjoints
in~$\btf$, as in Figure~\ref{fig:13}, and using the result as the adjoint
in~$\bft$.  (In these pictures we work in the bordism categories with colored
boundaries, where the same holds.) 
  \end{remark}

  \subsection{$k$-morphisms in~$\Bord_{n,\partial }$}\label{subsec:4.3}

   \subsubsection{The definition}\label{subsubsec:4.3.1}
 The bordism multicategory with boundary theory~$\Bnp$ is an extension
of~$\Bord_n$.  The boundary~$\partial X$ of a $k$-morphism of depth~$d$ has a
distinguished subset~$B_{-1}$, the ``colored'' subset
of~\S\ref{subsubsec:2.1.4}.  There are many variations of this construction,
which for example allow for multiple boundary theories and domain walls.  (We
use two boundary theories in the proof of Lemma~\ref{thm:19}.)

  \begin{definition}[]\label{thm:42}
  Fix~$n\in \ZZ^{>0}$, $k\in \{0,\dots ,n\}$ and $d\in \{0,\dots ,k\}$.  Let
$X$~be a compact $k$-dimensional manifold with corners of depth~$\le d$.  The
data of a \emph{$k$-morphism of depth~$d$ in~$\Bnp$} on~$X$ are:

 \begin{enumerate}[label=\textnormal{(\roman*)}]

 \item if~$d\ge2$, closed $(k-d)$-manifolds $\Xz d,\Xo d, \Bz d, \Bo d$, not
all empty;

 \item if~$d\ge3$, recursively for $j=d-1,d-2,\dots ,2$ compact
$(k-j)$-manifolds $\Xz j,\Xo j,\Bz j,\Bo j$ with corners of depth~$\le d-j$
equipped with diffeomorphisms
  \begin{equation}\label{eq:a46}
  \begin{gathered}
     \varphi ^\delta _{-j}\:\,\Xz{(j+1)}\;\;\cup\;\; [0,1]\times \left\{
     \Xz{(j+2)}\amalg 
     \Xo{(j+2)} \right\} \;\;\cup\;\; \Xo{(j+1)} 
     \cup\;\;\Bd{(j+1)}
     \longrightarrow \partial (\Xd
     j),\\[1ex]
     \beta ^\delta _{-j}\:\,\Bz{(j+1)}\;\;\cup\;\; [0,1]\times \left\{
     \Bz{(j+2)}\amalg \Bo{(j+2)} \right\} \;\;\cup\;\; \Bo{(j+1)}
     \longrightarrow \partial (\Bd j), 
 \end{gathered}
  \end{equation}
for~$\delta \in \{0,1\}$;

\item if~$d\ge1$, compact $(k-1)$-manifolds $\Xz 1,\Xo 1,\Bmo$ with corners
of depth~$\le d-1$ equipped with diffeomorphisms
  \begin{equation}\label{eq:aa46}
  \begin{gathered}
     \varphi ^\delta _{-1}\:\,\Xz{2}\;\;\cup\;\; [0,1]\times \left\{
     \Xz{3}\amalg 
     \Xo{3} \right\} \;\;\cup\;\; \Xo{2} 
     \cup\;\;\Bd{2}
     \longrightarrow \partial (\Xd
     1)\\[1ex]
     \beta _{-1}\:\,\Bz{2}\;\;\cup\;\; [0,1]\times \left\{
     \Bz{3}\amalg \Bo{3} \right\} \;\;\cup\;\; \Bo{2}  \longrightarrow
     \partial (\Bmo) ,
  \end{gathered}
  \end{equation}
for~$\delta \in \{0,1\}$;

 \item if~$d\ge1$, a diffeomorphism 
  \begin{equation}\label{eq:a47}
     \varphi \mstrut _0\:\,\Xz{1}\;\;\cup\;\; [0,1]\times \left\{
     \Xz{2}\amalg \Xo{2}\right\} \;\;\cup\;\; \Xo{1}
     \;\;\cup\;\; B\mstrut_{-1}\longrightarrow \partial X.
  \end{equation}

 \end{enumerate}
  \end{definition}

  \begin{figure}[ht]
  \centering
  \includegraphics[scale=.55]{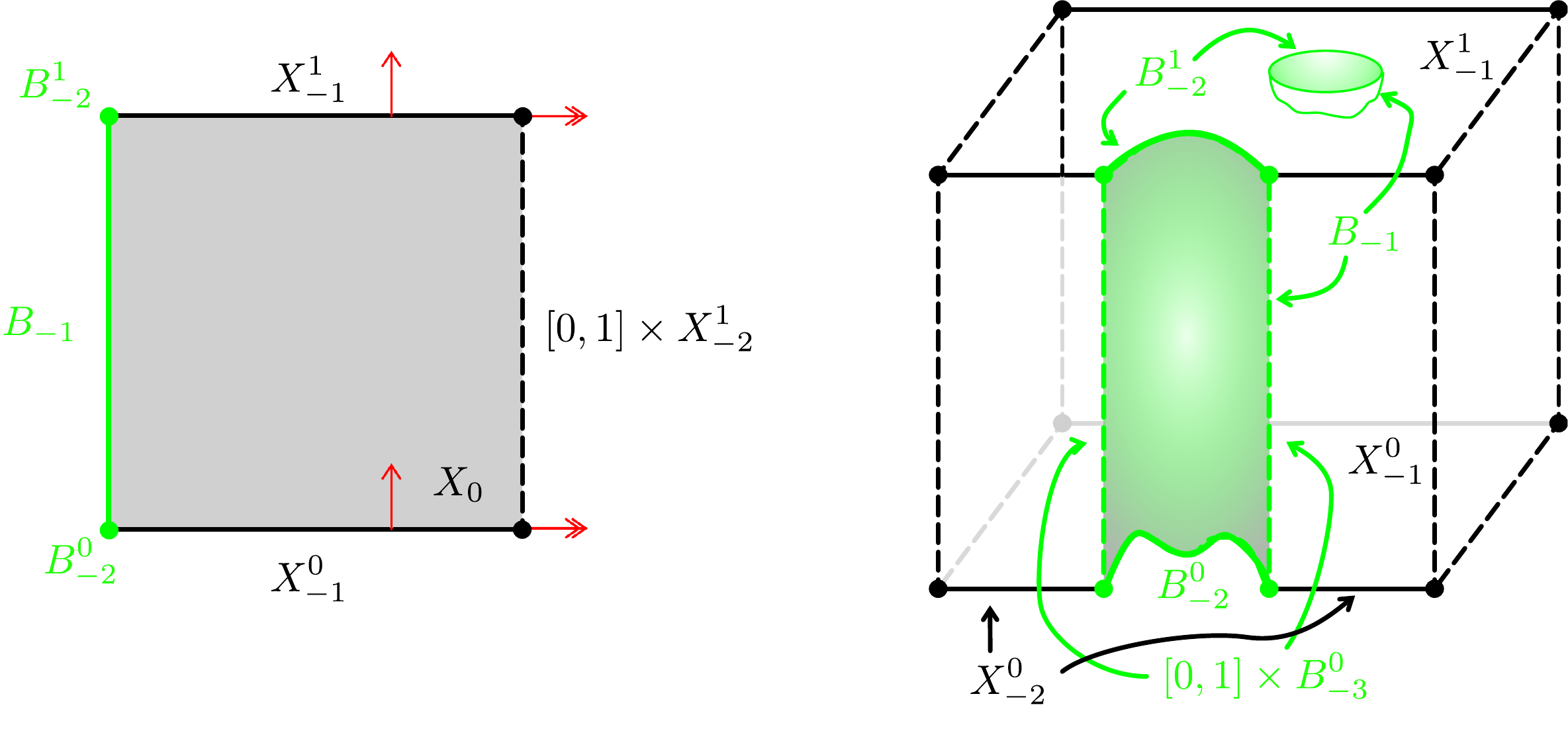}
  \vskip -.5pc
  \caption{A 2-morphism of depth 2 and a 3-morphism of depth 3}\label{fig:23}
  \end{figure}

  \begin{remark}[]\label{thm:43}
 As in Remark~\ref{thm:38}(2), set $\Xd{(d+1)}=\Bd{(d+1)}=\Bd{(d+2)}=\emptyset
$, $\delta \in \{0,1\}$.  The categorical interpretation of the bordism
described in Remark~\ref{thm:38}(3) is still valid.  The
embeddings~\eqref{eq:48} still exist, but now the embeddings~$\mps$ do not
cover~$\bX$.  Rather, the embeddings $\beta _{-j}^\delta ,\beta \mstrut
_{-1}$, $j\in \{2,\dots ,d-1\}$, $\delta \in \{0,1\}$, combine to embeddings
  \begin{equation}\label{eq:54}
     \gdj\:[0,1]^{j-2}\times \Bd j\longrightarrow \Bmo,\qquad
     j\in \{2,\dots ,d\},\quad \delta \in \{0,1\}.
  \end{equation}
Let~$\mg^\delta _{-j}$ be the restriction of~$\gdj$ to $[0,1]^{j-2}\times
\mB^\delta _{-j}$.  Then $\bX$ is the disjoint union of three sets: (i)~the
images of~$\mps$, $j\in \{1,\dots ,d\}$, $\delta \in \{0,1\}$; (ii)~the
images of~ $\mg^\delta _{-j}$, $j\in \{2,\dots ,d-1\}$, $\delta \in \{0,1\}$;
and (iii)~the $(k-1)$-manifold $\mBo$.
  \end{remark}

   \subsubsection{The tangent filtration}\label{subsubsec:4.3.2}
 For points $x\in \bX$ in the image of one of the $\mps$, the
filtration~\eqref{eq:49} and orientation of the lines~\eqref{eq:51} apply.
If $x\in \bX$ lies in the image of some $\mg^\delta _{-j}$, choose $t^1,\dots
,t^{j-2}\in [0,1]$, $\mb\in \mB^\delta _{-j}$ such that $x=\mg^\delta
_{-j}(t^1,\dots ,t^{j-2};\mb)$.  Then $T_xX$~has a decreasing filtration
  \begin{equation}\label{eq:a49}
     T\mstrut _xX=T_{x,0}X\supset T_{x,-1}X\supset \cdots\supset
     T_{x,-(j-1)}X\supset T_{x,-j}X=T\mstrut _{\mx}\mB\mstrut _{-j} 
  \end{equation}
in which 
  \begin{equation}\label{eq:a50}
  \begin{aligned}
     T_{x,-1}X&=T\mstrut _x\Bmo,\\[1.5ex]
     T_{x,-i}X &= d\mg^\delta _{-j}\bigl(0^{i-2}\oplus \RR^{j-i}\oplus
     T_{\mx}\mB_{-j}\bigr),\qquad i\in \{2,\dots ,j\}. 
  \end{aligned}
  \end{equation}
The associated graded is a sum of real lines
  \begin{equation}\label{eq:a51}
     \Lx{n-k+1}\oplus \cdots\oplus \Lx{n-k+j}.
  \end{equation}
Orient~$\Lx{n-k+j}$ so that the positive direction leads into~$X$.  (This is
for a boundary theory $1\to \tF$; for boundary theories $\tF\to 1$ choose the
opposite orientation.)  Orient~$\Lx{n-k+j-1}$ so that the positive direction
leads into~$\Bmo$ if~$\delta =0$ and leads out of~$\Bmo$ if~$\delta =1$.
Orient $\Lx{n-k+j}$, $i\in \{1,\dots ,j-2\}$, so that the positive direction
points towards increasing~$t^{j-1-i}$.  These orientations---arrows of
time---can be omitted (as in~\S\ref{sec:2} and~\S\ref{sec:7}) since they can
be deduced from the arrows of time on the rest of~$\bX$.

   \subsubsection{Tangential structures, duals, and
   adjoints}\label{subsubsec:4.3.3} 
 The distinguished boundary $B_{-1}\subset X$ has a tangential structure of
rank~$n-1$ whereas the ``bulk'' $X\setminus B_{-1}$ has a tangential
structure of rank~$n$.  The former is allowed to be ``different'' than the
latter---for example, we may have a spin boundary theory of an oriented
theory.  Following~\S\ref{subsubsec:4.2.4}, let $\rho _n\:\sX_n\to \BGL$ be
the bulk tangential structure.  A \emph{boundary tangential structure}
consists of (i)~a rank $n-1$ tangential structure $\rho _{n-1}\:\sX_{n-1}\to
\BGLo$; (ii)~an inclusion $\GL_{n-1}\RR\hookrightarrow \GL_n\RR$;
and~(iii)~a map $\phi \:\sX_{n-1}\to \sX_n$ such that the diagram 
  \begin{equation}\label{eq:96}
     \begin{gathered} \xymatrix{\sX_{n-1}\ar[r]^{\phi } \ar[d]_{\rho _{n-1}}
     & \sX_n\ar[d]^{\rho _n} \\ \BGLo\ar[r]^{} & \BGL} \end{gathered} 
  \end{equation}
commutes.  Up to homotopy, (ii)~is an element of $\O_n/\O_{n-1}\approx
S^{n-1}$, a choice of unit vector in~$\RR^n$. 

  \begin{remark}[]\label{thm:76}
 In this paper $\rho _n$~represents $n$-framings and $\rho _{n-1}$~represents
$(n-1)$-framings.  Concretely, fix a separable infinite dimensional real
Hilbert space~$\sH$, define the contractible Stiefel manifold
$\Sti_n=\{\textnormal{isometric embeddings }\RR^n\to \sH\}$, and let the
Grassmannian $\Gr_n=\Sti_n/\O_n$ be the quotient of the Stiefel manifold by
the natural right $\O_n$-action.  Then $\Gr_n\simeq \BGL$.  Our convention
in~\S\ref{subsubsec:2.1.4} uses the embedding $\O_{n-1}\hookrightarrow \O_n$
induced from the inclusion\footnote{The extra direction at a colored boundary
point is ``spatial'' in the sense of Remark~\ref{thm:77}.  The
choice~\eqref{eq:97} is made so as to be stable under the inclusion
$\bno^{\textnormal{fr}}\to \bn^{\textnormal{fr}}$.}
  \begin{equation}\label{eq:97}
     \begin{aligned} \RR^{n-1}&\longrightarrow \RR^n \\ (\xi ^1,\dots
      ,\xi ^{n-1})&\longmapsto (\xi ^1,\dots ,\xi
      ^{n-1},0)\end{aligned} 
  \end{equation}
Let $\tGr=\Sti_n/\O_{n-1}$.  A point of~$\tGr$ is an $(n-1)$-dimensional
subspace $W\subset \sH$ together with a unit vector~$\nu \in W^\perp$; the
map to the usual Grassmannian~$\Gr_{n-1}$ which forgets~$\nu $ is a homotopy
equivalence.  Let $\rho _{n-1}\:\Sti_n\to \tGr$ be the quotient map and $\phi
\:\Sti_n\to \Sti_n$ the identity map.
  \end{remark}

The constructions of duals and adjoints in \S\ref{subsubsec:4.2.5} carry over
without modification; the colored boundary components are unchanged when
forming duals and adjoints.

   \section{Semisimplicity of $2$-dualizable categories}\label{sec:8}

In this appendix we write a proof of the following folk result, stated in the
body of the paper as Theorem~\ref{thm:53} and stated here as
Theorem~\ref{thm:72}.  The theorem concerns $2$-dualizable objects in
$\mathrm{Cat} _\mathbb{C}$; see \cite{BDSV,BJSS} for related
variants.  Our proof works over an arbitrary field~${\Bbbk}$.

  \begin{theorem}[]\label{thm:72}
 If $C$ is $2$-dualizable in $\mathrm{Cat}\mstrut_{\Bbbk}$, then $C$ is finite
semisimple abelian, and $C^\vee$ may be identified with $C\op$ in such a way
that the duality pairing $\langle\ |\ \rangle: C^\vee \times C \to
\mathrm{Vect}$ is identified with $\mathrm{Hom}$: \[ \langle x\op | y\rangle
= \mathrm{Hom}_C(x,y).  \]
\end{theorem}

The proof is broken up into three lemmas, which we state after introducing
the following.

  \begin{definition}[]\label{thm:73}
 A $\Bbbk$-linear category~$C$ is \emph{\textnormal{Hom}-finite} if all of its
Hom spaces are finite dimensional.
  \end{definition}

  \begin{lemma}[]\label{thm:70}
 Let $C$ be 1-dualizable in $\mathrm{Cat}\mstrut_{\Bbbk}$, with $C$ and $C^\vee$ both $\mathrm{Hom}$-finite. There is then an equivalence of linear categories $C^\vee \equiv C\op$ with the duality pairing being
\[
x\op\times y \mapsto \mathrm{Hom}_C(y,x)^\vee.
\]
\end{lemma}

\begin{remark}
 We do not know \emph{ab initio} that $C\op$ has cokernels; this is a
consequence of the Lemma. Thus the proof is executed in the world of all
$\Bbbk$-linear categories and not within $\mathrm{Cat}\mstrut_{\Bbbk}$,
whose objects are finitely cocomplete.   In particular, $C$ and $C\op$ have
kernels and cokernels
(i.e.~ are pre-abelian). One can also show that they must be balanced (all
monic epimorphisms are isomorphisms), but we do not know if they must be
abelian.
\end{remark}

  \begin{lemma}[]\label{thm:81}
 Let $C$~be 2-dualizable in~$\ck$.  Then $C$~and $C^\vee$~are Hom-finite. 
  \end{lemma}

  \begin{lemma}[]\label{thm:71}
 Under the assumptions of Theorem~\ref{thm:72}, the functor $\mathrm{Hom}_C:
C\op \times C \to \mathrm{Vect}$ is bi-exact.\end{lemma} \noindent In
particular, all epimorphisms and monomorphisms are split.

We briefly defer the proofs of Lemma~\ref{thm:70}, Lemma~\ref{thm:81}, and
Lemma~\ref{thm:71} in favor of the following.

  \begin{proof}[Proof of Theorem~\ref{thm:72}]
 Lemmata~\ref{thm:70} and ~\ref{thm:81} imply that $C$ has kernels and
cokernels, and their splitting, from Lemma~\ref{thm:71}, implies that $C$~ is
abelian.  Semisimplicity follows from $\mathrm{Hom}$-finiteness and
Lemma~\ref{thm:71}.  Decompose objects using nonzero noninvertible
endomorphisms until their endomorphism algebras become division rings.
Finiteness of the number of simple isomorphism classes follows from the
equivalence of $C\op$ with~$C\dual$; otherwise the latter would contain
infinite products of nonisomorphic simple objects.  The remainder of
Theorem~\ref{thm:72} requires an identification of the $\mathrm{Hom}$ pairing
with the vector space dual of its opposite, which is immediate from
semisimplicity.
  \end{proof}

Now to the lemmas.

\begin{proof}[Proof of Lemma~\ref{thm:70}]
We construct a left/right adjoint pair of linear functors 
\[
L:C\op\rightleftarrows C^\vee : R
\]
which we prove to be inverse equivalences. In fact, $R$ 
is the opposite of the $C^\vee$-counterpart of~ $L$, so that $(R\op, L\op)$ is the corresponding pair of
functors if we start with the category $C^\vee$ instead 
of~ $C$.

Define $L$ as the functor $x\op\to \cx$, where 
given $x\in C$, we define $\cx\in C^{\vee}$ by
\[
\langle \cx  | y\rangle = \mathrm{Hom}_C(y,x)^\vee, 
\qquad y\in C.
\] 
Note that $\cx $ is a right exact functional on $C$, 
so it defines an object of $C^\vee=\Hom_{\ck}(C,\Vect)$. (Moreover, the assignment $x\op\mapsto \cx $ is right exact, although this does not mean much before $C\op$ is shown to have cokernels.)

As advertised, $R$ sends $\eta\in C^\vee$ 
to the object ${\cet}\op$ of $C\op$, where $\cet\in C$ is defined by
\[
\langle \xi | \cet  \rangle = \mathrm{Hom}_{C^\vee}(\xi,\eta)^\vee, 
\qquad \xi\in C^\vee.
\] 
Now for $x\in C,\; \eta\in C^\vee$ we have the desired adjunction
  \begin{equation}\label{adj}
  \Hom_{C\op}(x\op,R\eta )=
\mathrm{Hom}_C(\cet , x) = \langle \cx  | \cet  \rangle^\vee =
  \mathrm{Hom}_{C^\vee} (\cx , \eta)=\Hom_{C\dual}(Lx\op,\eta ).
  \end{equation}

The Yoneda embedding asserts that $L$ is fully faithful; a formal consequence is that the adjunction unit $\mathrm{Id}_{C\op} \to R\circ L$ is an isomorphism.   Similarly, $R\op$, and therefore $R$, is also fully faithful, so the evaluation $L\circ R \to \mathrm{Id}_{C^\vee}$ is an isomorphism as well. 
\end{proof}

  \begin{proof}[Proof of Lemma~\ref{thm:81}]
 Denote by $\Delta: \mathrm{Vect}\to C^\vee\boxtimes C$ the coevaluation of
the duality, and let $S_C$~be the Serre automorphism.  The second of the
adjunctions  
  \begin{equation}\label{eq:100}
     \Delta^L = \langle\ |\ \rangle \circ(\mathrm{Id}\boxtimes S_C), \quad
     \Delta^R = \langle\ |\ \rangle \circ(\mathrm{Id}\boxtimes S_C^{-1}), 
  \end{equation}
combined with the right exactness of $\langle\ |\ \rangle $, shows that the
functor  
  \begin{equation}\label{eq:99}
     \mathrm{Hom}_{C^\vee\boxtimes C}(\Delta(\mathbf{1}), \underline{\ \ }):
     C^\vee\boxtimes C \to \mathrm{Vect} 
  \end{equation}
is also right exact.
Since $\Delta(\mathbf{1})$ is the quotient of a product $\Xi\boxtimes X\in
C^\vee\boxtimes C$ --- as is any object in $C^\vee\boxtimes C$ --- the right exactness of~\eqref{eq:99} implies that $\Delta (\mathbf{1})$~ must
therefore 
be a direct summand of~$\Xi\boxtimes X$, i.e., the image of a projector $P$ 
in $\mathrm{End}(\Xi\boxtimes X)$.
For all $\xi\in C^\vee$ and $x \in C$,
this $P$ induces \emph{finite-rank} projectors on all spaces 
$\mathrm{Hom}(\Xi\boxtimes X, \xi\boxtimes x )$ and 
$\mathrm{Hom}(\xi\boxtimes x, \Xi\boxtimes X)$, because the respective images are the finite-dimensional spaces
 \[
\mathrm{Hom}(\Delta(\mathbf{1}), \xi\boxtimes x) = 
\langle \xi | S_C^{-1}x\rangle, \quad 
\mathrm{Hom}(\xi\boxtimes x, \Delta(\mathbf{1})) = 
\langle \xi | S_Cx\rangle^\vee.
\] 

Given now $x,y\in C$, let's compute $\mathrm{Hom}_C(x,y)$ via the Zorro diagram, 
where we denote by \emph{tr} the transposition in the two variables:
\[
\mathrm{Hom}_{C}(x,y) = \mathrm{Hom}_{C}\Bigl(\bigl[\,\langle\ |\ \rangle^{tr} \boxtimes \mathrm{Id}\bigr]\bigl[x\boxtimes \Delta(\mathbf{1})\,\bigr],y\Bigr) = \mathrm{Hom}_{C \boxtimes C^\vee \boxtimes C}\Bigl(S_C^{-1}x\boxtimes \Delta(\mathbf{1}), \Delta(\mathbf{1})^{tr}\boxtimes y\Bigr)
\]
In the last step, we have used the first adjunction in~\eqref{eq:100}.
The last $\mathrm{Hom}$ space is the common image of the two commuting projectors acting by pre- and post-composition with $P$ on the space
\[
\mathrm{Hom}_{C\boxtimes C^\vee \boxtimes C}(x\boxtimes \Xi\boxtimes X, X\boxtimes \Xi\boxtimes y) = 
\mathrm{Hom}_C(x,X)\otimes \mathrm{Hom}_{C^\vee}(\Xi,\Xi)\otimes \mathrm{Hom}_C(X,y):= U\otimes V\otimes W.
\]
Post-composition with $P$ acts on $U\otimes V$ (and as 
the identity on $W$), 
and its finite rank implies that the image is contained in $F\otimes V\otimes W$, 
for some finite-dimensional $F\subset U$. But pre-composition by~ $P$
now acts with finite rank on $V\otimes W$, which proves that
$\mathrm{Hom}_C(x,y)$ is finite dimensional.

The Hom-finiteness of~$C^\vee$ is proved by a similar argument.
  \end{proof}

\begin{proof}[Proof of Lemma~\ref{thm:71}]
 Let $\Delta:\mathrm{Vect}\to C^\vee\boxtimes C$ be the unit for duality. Its right adjoint $\Delta^R$ satisfies
\[
\Delta^R(X) = \mathrm{Hom}\mstrut _{C\dual\boxtimes C} \bigl(\Delta(\mathbf{1}), X\bigr),\qquad
X\in C\dual\boxtimes C,
\] 
which implies $\Delta ^R$~is left exact. It is also right exact, being a $1$-morphism internal to 
$\mathrm{Cat}\mstrut_{\Bbbk}$.  Recall too the formula~\eqref{eq:100} for~$\Delta
^R$.  Now
the structural functor $C^\vee\times C\to C^\vee\boxtimes C$, $\xi\times x \mapsto \xi\boxtimes x$, is bi-exact. Following it with $\Delta^R$ leads to the bi-exact functor from $C\op\times C \to \mathrm{Vect}$ 
\[
x\op\times y \to \langle \cx  | S_C^{-1}y\rangle = \mathrm{Hom}(S_C^{-1}y,x)^\vee,
\] 
which proves the bi-exactness of $\mathrm{Hom}$.
\end{proof}

   \section{Internal Duals}\label{sec:5}

We describe here an abstract notion of internal duals, generalizing from a
tensor category (Definition~\ref{duals}) to an algebra object in a
$2$-category (Theorem~\ref{final}).  In particular, we show that our TFT $F$
with nonzero boundary condition $\beta$ leads to a fusion category $\Phi=
\underline{\mathrm{End}}^R\bigl(\beta(+)\bigr)$
(Definition~\ref{thm:14}). Since our knowledge of $\Phi$ comes from TFT
calculus, we must avoid unpictorial internal structures (for example, the use
of contravariant functors such as $x\mapsto x^*$) in describing internal
duality. The main application is
  \begin{theorem}\label{thm:B1} A tensor category whose underlying category
is dual to its opposite category and which satisfies the Frobenius condition
of Definition~ \ref{frob} and the bimodule property of
Proposition~\ref{bimod}, has internal left and right duals.
\end{theorem}

\noindent
 For our~$\Phi$, these conditions are checked in Lemma~\ref{thm:18}.  

We refer to~\cite{BJS} for another discussion of rigidity and dualizability. 

\begin{remark}
In the setting of TFT, the conditions separate neatly into a Frobenius-bimodule 
condition and an adjunction condition, reflecting two different geometric properties of a 
TFT with boundary generated by an algebra object and its regular boundary conditions. 
The logic of our application  to $\Phi$ compels a different path; 
we will return to the more natural  statements in a future paper. 
\end{remark}

Let $\Phi, \Pd$ be a dual pair of objects in a symmetric monoidal $2$-category $\left(\cM,\boxtimes\right)$. 
We are mostly interested in the \emph{categorical case} when $\Phi,\Pd$ are an opposite couple of linear 
categories paired by Hom, and can even restrict to semisimple categories, but the algebra below is agnostic 
about that, unless we explicitly flag it. It is convenient to denote the duality pairing $\Pd\boxtimes 
\Phi \to \mathbf{1}$ by writing $\langle \xi \,|\,y\rangle$, as in the categorical case, when $y\in \Phi$,
$\xi\in \Pd= \Phi\op$ (the opposite category). This convention is symmetric under simultaneous swapping of the arguments 
and of $\Phi$ with $\Pd$. When checking identities, conversion to the formalism of arrows is 
straightforward.\footnote{At any rate, we can reduce to the categorical case by passing to the functors on 
$\cM$ represented by $\Phi,\Pd$.} 
Equalities stand for canonical isomorphisms of $1$-morphisms.

Assume given an $E_1$ structure on $\Phi$, with strict unit $\eta:\mathbf{1}\to \Phi$ and multiplication 
$\nabla:\Phi\boxtimes \Phi\to \Phi$. When $\Phi$ is a category and when no confusion ensues, we also 
write $x\cdot y$ for $\nabla(x,y)$ and $1$ for the tensor unit.  The dual object $\Pd$ is a 
$\Phi\text{-}\Phi$ bimodule. This bimodule is invertible, if $\Phi$ is $2$-dualizable as an algebra object, 
and represents then the \emph{Serre autofunctor} of the (category of modules
over the) $E_1$ object $\Phi$.

We shall not adopt the \emph{a priori} assumption of $2$-dualizability here; however, we will 
require that $\eta$ and $\nabla$ have right adjoints $\vep:\Phi\to\mathbf{1}$ and $\Delta:\Phi\to  
\Phi\boxtimes \Phi$. This condition is always met in the categorical case, with explicit formulas for 
$\vep$ and for the dual functor $\Dd: \Pd\boxtimes \Pd \to \Pd$: 
  \begin{equation}\label{eq:64}
     \vep(z) = \mathrm{Hom}_\Phi(1,z); \quad \Dd(x\op\boxtimes y\op) =
        \nabla(x,y)\op, \quad \text{for }x,y,z \in \Phi. 
  \end{equation}
With this structure, $\Pd$ becomes a tensor category with unit $1\op =\vep^\vee(1)$. 
More generally, the dual object $\Pd$ is an $E_1$ object with the same features as $\Phi$: 
the dual arrow $\nabla^\vee$ defines a comultiplication which is right adjoint to the multiplication 
$\Dd$, and the latter has unit $\vep^\vee$, with right adjoint $\eta^\vee$.

\begin{remark}\label{ops}
 This interpretation of the dual  right adjoint of $\nabla$ holds for any
functor $\varphi:X\to Y$ between categories which are in duality with
their opposites: namely,  $\varphi\op:X\op\to Y\op$ is $\varphi \op=(\varphi^R)^\vee
=\left(\varphi^\vee\right)^L$. In particular, adjoints exist.  Recall
also that, when $X, Y$ are $2$-dualizable, with (additive) Serre
automorphisms $S^+_X, S^+_Y$, the left and right adjoints of $\varphi$ are
related by $S^+_Y\circ\varphi^L = \varphi^R\circ S^+_X$.  Commuting duals
with adjoints will therefore bring out additive Serre functors.
 \end{remark}

Define now the pairing $B: \Phi\boxtimes \Phi \to \mathbf{1}$ as $B = \vep\circ \nabla$. 
When $\Phi$ is a category, $B(x, y) = \mathrm{Hom}_\Phi(1,x\cdot y)$, for two general 
objects $x,y$. From $B$, we define a dual pair of functors, by dualizing separately with 
respect to each variable: 
  \begin{equation}\label{eq:65}
     f, f^\vee: \Phi\to \Pd, \quad f(x) := B(x, -), \quad f^\vee(y) := B(-,
     y) 
  \end{equation}

\begin{proposition}
$f$ is a right, and $f^\vee$ a left $\Phi$-module morphism. 
\end{proposition}
\begin{proof}
$\langle f(x\cdot y) \,|\, z\rangle = B\left(x\cdot  y, z\right) = \vep( x\cdot y\cdot z ) = 
B(x, y \cdot z) = \langle f(x) \,|\, y\cdot z \rangle =\langle f(x). y \,|\, z\rangle$,  
so that $f(x\cdot y) =  f(x). y$, and similarly for $f^\vee$.
\end{proof}

\begin{definition} \label{frob}
We say that $\Phi$ satisfies the \emph{(non-symmetric) adjoint Frobenius condition} when 
$B$ ~is a perfect pairing: that is, $f$ and $f^\vee$ are isomorphisms.  If
so, we
define the Serre automorphism of~$\Phi $ as $S^\otimes =(f^\vee)\inv \circ f.$
\end{definition}

\begin{proposition}\label{stuff}
Assume that $\Phi$ satisfies the Frobenius condition. The following natural isomorphisms apply: 
\begin{enumerate}[label=\textnormal{(\roman*)}]
\item $B(x,y) = B(y, \St x)$. In particular, symmetry of $B$ is equivalent to a trivialization of $\St $.
\item  $f\circ\eta = f^\vee\circ\eta = \vep^\vee$. For a category, $f(1) = f^\vee(1) = 1^{op}$. 
\item As functors $\Phi\to\Phi^\vee$, we have  $\vep^\vee.(\underline{\ }) =\St (\underline{\ }) .\vep^\vee $. 
In the categorical case, $1^{op}. x = \St (x). 1^{op}$.
\item  $\St $ is naturally a tensor automorphism of $\Phi$, and twisting the $\Phi$-action by $\St $ induces 
the Serre autofunctor  $M\mapsto \Phi^\vee\boxtimes_\Phi M$ on the $2$-category of 
left $\Phi$-modules. 
\end{enumerate}
\end{proposition}

\begin{remark}
Promoting $\St $  to a tensor functor means equipping it with isomorphisms, compatible with the 
associativity and unit laws on $\Phi$,
\[
\St \circ \nabla \cong \nabla \circ (\St \times \St ), \quad \St \circ \eta \cong \eta;
\] 
while not evident from the expression $(f^\vee)^{-1}\circ f$, they do follow from Parts~(i)-(iii), 
as in the proof below. On the other hand, reduction of $\Phi^\vee\boxtimes_\Phi\underline{\ }$ to a 
tensor automorphism of $\Phi$ is a formal consequence of the isomorphy of $f^\vee$. 
\end{remark}

\begin{proof}
Parts (i)-(iii) are immediate from the  properties of $f,f^\vee, B$; thus, 
\[\begin{aligned}
B(x,y) &= \langle f(x) \,|\, y\rangle = \langle f^\vee\circ \St (x) \,|\, y \rangle = \langle f(y) \,|\, \St (x) \rangle
	&\text{for (i)}, \\
\langle f(1) \,|\, x\rangle &= B(1,x) = \vep(x) = B(x,1) =\langle f^\vee(1) \,|\, x\rangle &\text{for (ii),} \\
\St (x). 1^{op} &= f^\vee\left(\St (x)\right) = f(x) = 1^{op}. x &\text{for (iii)}. 
\end{aligned}
\]
Multiplicativity of $\St $ now follows:
\[
f^\vee\left( \St (xy) \right) = \St (xy).1^{op} = 1^{op}.xy = \St (x).1^{op}.y = \St (x)\cdot \St (y).1^{op} = 
	f^\vee\left(\St (x)\cdot \St (y)\right),
\]  
using categorical notation for simplicity. To complete (iv), consider the following diagram 
of right $\Phi$-modules, with left multiplication in the bottom row: 
\[
\xymatrix{\Phi\boxtimes\Phi \ar[r]^{\nabla} \ar[d]^{\St \boxtimes f} & \Phi \ar[d]^{f} \\
	\Phi \boxtimes \Phi^\vee \ar[r] & \Phi^\vee
}
\] 
We claim this commutes naturally. Assuming this, let us interpret the diagram: the right 
vertical arrow $f$ gives an isomorphism of the identity with the Serre autofunctor on 
$\Phi$-modules, while the left arrow exhibits the necessary intertwining twist by $\St $ in 
the left $\Phi$-action. 

Exploiting the right $\Phi$-module structure, it suffices to check commutativity on 
$\Phi\boxtimes\eta$, when this becomes the isomorphism $\St (x). f(1) = f(1).x = f(x)$, 
from (ii) and (iii).
\end{proof}

\begin{remark}
The Serre functor $\St $ above need not agree with the additive Serre automorphism $\Spl _\Phi$ of 
Remark~\ref{ops}, which is independently defined whenever the object $\Phi\in \cM$ is 
$2$-dualizable. However, the two will agree for a fusion category $\Phi$, because of 
its $3$-dualizability. See also Remark~\ref{twoserre} below for a general relation between 
the two. 
\end{remark}

The isomorphisms $f,f^\vee$ allow us to transport the structure tensors $\eta,\nabla,\Delta,\vep$ 
to a matching structure on $\Phi^\vee$, denoted by overbars. Choosing either $f$ or $f^\vee$ results 
in isomorphic structures on $\Phi^\vee$, because all structure tensors commute with $\St $. Dualizing 
them gives a new structure $\bar{\vep}^\vee, \bar{\Delta}^\vee, \bar{\nabla}^\vee, \bar{\eta}^\vee$ 
on $\Phi$. We get the following diagram, in which the bottom row maps are related to 
the top row maps by duality and adjunction using uniform rules, $\vep = \eta^R, \Delta=\nabla^R$,
$\bar{\vep}=\bar{\eta}^R, \bar{\Delta} = \bar{\nabla}^R$ and all ensuing relations:
\begin{equation}\label{mults}
\begin{gathered}
\xymatrix@C+12pt{
 \hspace{-20pt}\Phi\boxtimes \Phi  \ar@/^/[r]|{\nabla} & 
 	\Phi \ar@/^/[l] |\Delta \ar@/^/[r] | \vep & \mathbf{1} \ar@/^/[l] |
\eta \ar@<1ex>@{~>}[drr]+R^<<<<<<<<<{f}  &&\ar@{~>}[dll]_<<<<<<{f^\vee}&
 	\hspace{-35pt} \Phi\boxtimes \Phi  \ar@/^/[r]|{\overline{\Delta}^\vee}& 
		\Phi \ar@/^/[l] |{\overline{\nabla}^\vee}  \ar@/^/[r] | {\overline{\eta}^\vee}
			& \mathbf{1} \ar@/^/[l] | {\overline{\vep}^\vee}
	\\
 \hspace{-30pt}\Pd\boxtimes \Pd  \ar@/^/[r]|{\Dd} &
 	\Pd \ar@/^/[l] | {\nabla^\vee} \ar@/^/[r] | {\eta^\vee} & \mathbf{1} \ar@/^/[l] | {\vep^\vee}
 	& &  &
 	\hspace{-30pt}\Pd\boxtimes \Pd    
	\ar@/^/[r]|{\overline{\nabla}} & \Pd \ar@/^/[l] |{\overline{\Delta}} \ar@/^/[r] | {\overline{\vep}}
		& \mathbf{1} \ar@/^/[l] | {\overline{\eta}}
}
\end{gathered}
\end{equation}
The dual corners are related by the morphism $f^\vee$. Because $\bar{\eta} = f\circ\eta$, etc., 
we find from Proposition~\ref{stuff}.ii that
\begin{proposition} 
In the diagram above, units and traces match in each row: $\bar{\eta} = \vep^\vee$, 
$\bar{\vep} = \eta^\vee$. \qed
\end{proposition}

\begin{proposition}\label{bimod}
Under the Frobenius assumption, the following conditions are equivalent:
\begin{enumerate}[label=\textnormal{(\roman*)}]\itemsep0ex
\item The coproduct $\Delta$ is a $\Phi\text{-}\Phi$ bimodule map (for the outer $\Phi$-actions on the 
two $\Phi$-factors).
\item The multiplication $\Delta^\vee$ is a $\Phi\text{-}\Phi$ bimodule map (for the \emph{inner} 
$\Phi$-actions on the two $\Phi^\vee$-factors).
\item The two structures on $\Phi$ in the top row of \eqref{mults} are transpose-isomorphic.
\item The two structures on $\Phi^\vee$ in the bottom row of \eqref{mults} are transpose-isomorphic.
\end{enumerate}
\end{proposition}

\begin{proof}
Parts (i) and (ii) are equivalent by duality, (iii) and (iv) are so via the isomorphisms induced by $f$. 
Note further that the diagonal arrow is compatible with the $\Phi\text{-}\Phi$ bimodule structures: 
on the right, because $f$ is a right module map, an on the left, because we could equally well have 
used the left module isomorphism $f^\vee$ instead. In light of the matching units, which are free 
generators of $\Phi^\vee$ over $\Phi$,  
the bimodule condition determines the multiplication maps and forces the agreement of the 
remaining structure maps on each row.
\end{proof}
 
\begin{definition}\label{duals}
When $\Phi$ is a category, the internal right and left duals ${}^*x,x^*$ of an object 
$x\in \Phi$ are the objects characterized (up to unique isomorphism, if they exist) by 
the functorial (in $y,z$) identities 
\begin{equation}
\mathrm{Hom}(x\cdot y,z) = \mathrm{Hom}(y, x^*\cdot z),\quad 
 	\mathrm{Hom}(y\cdot x,z) = \mathrm{Hom}(y, z\cdot {}^*x).
\end{equation}
\end{definition}
It turns out that the conditions in \eqref{bimod} force the existence of internal duals and 
their expression in terms of $f^\vee$ and $f$. To see this, we first give an abstract formulation.

Dualizing the product $\nabla$ in the second argument gives the \emph{left multiplication map} 
$\lambda: \Phi\to \Phi\boxtimes\Phi^\vee$. In the categorical case, $\lambda(x)$ represents the 
left multiplication  by $x\in \Phi$.  Similarly, for the first argument we get the right 
multiplication map $\rho: \Phi\to \Phi\boxtimes\Phi^\vee$. Repeating this for $\Delta^\vee$ leads to the two maps 
$\lambda', \rho':  \Phi^\vee\to \Phi^\vee\boxtimes\Phi$. 
In the abusive but readable argument notation, with Greek arguments living in $\Phi^\vee$,
\begin{equation}\label{primes}
\langle \lambda' (\xi_1) \,|\, y\boxtimes \xi_2 \rangle =
	\langle \Delta^\vee(\xi_1, \xi_2) \,|\, y \rangle, \quad 
\langle \rho' (\xi_2) \,|\, y\boxtimes \xi_1 \rangle =
	\langle \Delta^\vee(\xi_1, \xi_2) \,|\, y \rangle. 
\end{equation}
The maps $\lambda',\rho'$ will be the abstract versions of the `tensoring with duals' 
\[
x^{op} \mapsto \left(z\mapsto x^* \cdot z\right), \quad x^{op} \mapsto \left(z\mapsto z \cdot {}^*x\right).
\]

\begin{remark}
$\lambda', \rho'$ are related to the op-conjugates $\lambda^{op}, \rho^{op}: 
\Phi^\vee\to \Phi^\vee\boxtimes\Phi$ as follows: 
\[
\lambda'\cong(\mathrm{Id}\boxtimes {\Spl _\Phi}^{-1})\circ \lambda^{op}, \quad 
	\rho'\cong (\mathrm{Id}\boxtimes {\Spl _\Phi}^{-1})\circ \rho^{op}
\] 
The source of the additive Serre correction $\Spl _\Phi$ is described in Remark~\ref{ops}.
\end{remark}

Denote by $\tau$ the symmetry $\Phi\boxtimes\Pd\to\Pd\boxtimes\Phi $.

\begin{theorem}\label{final}
The equivalent conditions of Proposition~\ref{bimod} are also equivalent to:
\begin{enumerate}[label=\textnormal{(\roman*)}]\itemsep0ex
\item $ \lambda' \cong\tau\circ\lambda\circ f^{-1}$. 
\item $\rho' \cong\tau\circ\rho\circ (f^\vee)^{-1}$.
\item In the categorical case: $\Phi$ has internal left and right duals. 
\end{enumerate}
\end{theorem}

\begin{proof}
We check the two sides by pairing against a triple of arguments $(\xi_1, y, \xi_2) \in \Phi^\vee
\times\Phi\times \Phi^\vee$, leaving to the reader the unenviable task of convert this to identities between 
morphisms, duals and adjoints. Having written out the left sides in \eqref{primes} above, 
we start with the right side of (i):
\[
\langle \tau\circ\lambda\circ f^{-1}(\xi_1) \,|\, y\boxtimes \xi_2 \rangle =  
	\langle \lambda\circ f^{-1}(\xi_1) \,|\,  \xi_2 \boxtimes y\rangle = 
	\langle  \xi_2 \,|\, f^{-1}(\xi_1)\cdot y \rangle = 
	\langle \xi_2. f^{-1}(\xi_1) \,|\, y \rangle
\]
where the middle line is the definition of $\lambda$, while dot represents the right multiplication 
action of $f^{-1}\xi_1$ upon $\xi_2\in \Phi^\vee$. Agreement with $\lambda'$ is then equivalent to
\[
 \xi_2. f^{-1}(\xi_1) = \Delta^\vee(\xi_1, \xi_2);
\]
but using the right module property of $f$, we have 
\[
\xi_2. f^{-1}(\xi_1) = f\left[f^{-1}(\xi_2) \cdot f^{-1} (\xi_1)\right],
\]
thus reaching Condition (iv) in Proposition~\ref{bimod}.

Similarly, for the right side of (ii),
\[
\langle \tau\circ\rho\circ (f^\vee)^{-1}(\xi_2) \,|\, y\boxtimes \xi_1 \rangle =  
	\langle \rho\circ (f^\vee)^{-1}(\xi_2) \,|\,  \xi_1 \boxtimes y\rangle = 
	\langle \xi_1 \,|\, y\cdot (f^\vee)^{-1}(\xi_2) \rangle = 
	\langle (f^\vee)^{-1}(\xi_2). \xi_1 \,|\, y \rangle
\]
and identity (ii) is equivalent to
\[
	 (f^\vee)^{-1}(\xi_2). \xi_1 = \Delta^\vee(\xi_1, \xi_2),
\]
which follows as before, this time from the left-module property of $f^\vee$.

Finally, for Part (iii) we must convert the identities into the recognizable form \eqref{duals}. 
For this, we let $\xi_{1,2}$ be opposites of objects $x_{1,2}\in \Phi$; then,  
$\Delta^\vee(\xi_1,\xi_2)$ is the opposite object to $x_1\cdot x_2$, and we can rewrite
\[
\begin{aligned}
\langle \xi_2 \,|\, f^{-1}(\xi_1)\cdot y \rangle &= \mathrm{Hom}_\Phi\left(x_2, f^{-1}(\xi_1)\cdot y\right), \\
\langle \xi_1 \,|\, y\cdot (f^\vee)^{-1}(\xi_2) \rangle &= 
	\mathrm{Hom}_\Phi\left(x_1, y\cdot (f^\vee)^{-1}(\xi_2) \right),\\
\langle \Delta^\vee(\xi_1, \xi_2) \,|\, y \rangle &=  \mathrm{Hom}_\Phi\left(x_1\cdot x_2, y\right)
\end{aligned}
\]  
exhibiting $f^{-1}(\xi_1)$ as $x_1^*$ and $(f^\vee)^{-1}(\xi_2)$ as ${}^*x_2$ in Definition~\ref{duals}.
\end{proof}

\begin{remark}\label{twoserre}
Under the assumptions of \eqref{bimod}, and if, in addition, $\Phi$ is
$2$-dualizable, one can prove~\cite{FT} that the additive Serre 
functor $\Spl _\Phi$ is related to $\St $:
\begin{equation}\label{AS}
\Spl (x\cdot y) = \St (x)\cdot \Spl (y) = \Spl (x) \cdot {\St} ^{-1}(y).
\end{equation} 
 In particular, we have 
\[
\Spl (x) = \St (x)\cdot \Spl (1) = \Spl (1)\cdot {\St} ^{-1}(x), 
\]
and, as the functor $\Spl $ is invertible, $\Spl (1)$ must be a unit. In the categorical case, $\St (x) = x^{**}$, 
and the relations follow by applying Serre duality to the adjunction relations in Definition~\ref{duals}. 

One instance of \eqref{AS} is when $\Spl =\St ={\St} ^{-1}$, which happens in the case of fusion categories ~
\cite{DSS}, but that is not the only option. Thus, if $\Phi$ is the derived
category of bounded complexes of coherent 
sheaves on a projective manifold with the obvious internal duals, the multiplicative Serre functor 
$\St $ is the identity, while the functor $\Spl $ is tensoring 
with the canonical line bundle of~$X$ in degree $(-\dim X)$. 
\end{remark}

   \section{Complete reducibility of fusion categories}\label{sec:6}
 
A fusion category whose unit is simple cannot be decomposed as a direct sum,
even after passing to a Morita equivalent model: otherwise, we would split
the unit.  The following converse follows easily from several statements in
\cite{EGNO}, but we give a complete proof, at the price of rehashing some
basic facts. Throughout, $\Phi$ will denote a fusion category.

\begin{theorem}[Complete Reducibility]\label{thm:cr}
$\Phi$ is Morita equivalent to a direct sum of fusion categories with simple 
unit.
\end{theorem}

  \begin{corollary}[]\label{thm:34}
 $\Phi $~is Morita equivalent to a fusion category~$\Phi _0$ with simple
unit if and only if the Drinfeld center of~$\Phi $ is invertible.
  \end{corollary}

\noindent 
 We prove Corollary~\ref{thm:34} at the end of the appendix.

  \begin{remark}[]\label{thm:63}
 A closely related statement is used in~\cite[Remark~5.2]{DMNO}: if $\Phi $
is an indecomposable fusion category, then there exists a fusion
category~$\Phi '$ with simple unit and a braided equivalence $Z(\Phi
')\xrightarrow{\;\simeq\;}Z(\Phi )$ of the Drinfeld centers.  The proof is
based on Lemma~3.24 and Corollary~3.35 in~\cite{EO}. 
  \end{remark}

We break up the proof of Theorem~\ref{thm:cr} into small steps. Let
$\mathbf{1}=\sum_i p_i$ be the decomposition of the unit of $\Phi$ into
simple objects. Call an object $x$ \emph{self-adjoint} if it is isomorphic
with $x^*$.  \begin{lemma}\label{project} Each $p_i$ is a self-adjoint
projector: $p_i^*\cong p_i$, $p_i^2= p_i$, $\mathrm{End}(p_i)=\bC$.  In addition,
$p_ip_j=0$ if $i\neq j$.
\end{lemma}
\begin{proof}
We have $p_i = p_i\cdot\mathbf{1} = 
\sum_j p_ip_j$, so $p_ip_j=0$ except for a single $j$, when it equals $p_i$. 
On the 
other hand, $p_j = \mathbf{1}\cdot p_j = \sum_k p_kp_j$, but the sum contains $p_ip_j=p_i$, 
so $p_i=p_j$, proving the multiplicative claims. Further, $\mathbf{1}^* =\mathbf{1}$, and 
$p_i^*p_i\neq 0$ because $\mathrm{Hom}(\mathbf{1},p_i^*p_i) \cong  \mathrm{End}(p_i)\neq0$, 
so we must have $p_i^*\cong p_i$.
\end{proof}

Lemma~\ref{project} gives a ``matrix decomposition'' of $\Phi$ as  
\[
\Phi \cong  \bigoplus\nolimits_{i,j} p_i\cdot\Phi\cdot p_j =: \bigoplus\nolimits_{i,j} \Phi_{ij},
\]
with fusion categories $\Phi_{ii}$ having simple units $p_i$ on the diagonal, $\Phi_{ii}\text{-}\Phi_{jj}$ 
bimodule categories $\Phi_{ij}$ (identified with $\Phi_{ji}^{op}$ under internal duality), and 
multiplication compatible with matrix calculus: 
\[
\Phi_{ij}\times \Phi_{jl} \to \Phi_{il}, \quad\Phi_{ij}\cdot\Phi_{kl} =0 \text{ if } j\neq k. 
\]
The equivalence classes of indices generated by the condition $\Phi_{ij}\neq 0$ gives a direct sum decomposition 
of $\Phi$, matching 
the block-decomposition of the matrix. Call $\Phi$ \emph{indecomposable} if a single block is present. 
We claim that an indecomposable $\Phi$ is Morita equivalent to any of its diagonal entries, 
selecting $\Phi_{11}$ for our argument. 

The  equivalence is induced by the first row and the first column of $\Phi$:  
the $\Phi_{11}\text{-}\Phi$ bimodule $R:=\bigoplus_i \Phi_{1i}$ and the $\Phi\text{-}\Phi_{11}$ 
bimodule $C: =\bigoplus_j \Phi_{j1}$. We check it in the following two lemmata.

\begin{lemma}\label{mult}
The multiplication map $R\boxtimes_\Phi C \to \Phi_{11}$ is an equivalence of $\Phi_{11}\text{-}\Phi_{11}$ 
bimodule categories.
\end{lemma}
\begin{proof}
We have $\Phi\boxtimes_\Phi \Phi =\Phi$, and splitting the left factor $\Phi$ into its rows $R_i$ 
and the right factor into its columns $C_j$ gives a direct sum decomposition of $\Phi$ as 
$R_i\boxtimes_\Phi C_j$. Examining the action of the projectors $p_k$, on $R_i$ on the left 
and on $C_j$ on the right, identifies this with the $\Phi_{ij}$ decomposition of $\Phi$.
\end{proof}

\begin{lemma}\label{mult2}
The multiplication map $\mu :C\boxtimes_{\Phi_{11}} R \to \Phi$ is an equivalence of $\Phi\text{-}\Phi$ 
bimodule categories.
\end{lemma}

\noindent 
 Lemma~\ref{mult2} concludes the proof of Theorem~\ref{thm:cr}.

The proof of this direction requires some preliminary facts.

\begin{lemma}[Linearity of adjoints]\label{moduleadjoint}
The adjoints $\varphi^L, \varphi^R$ of an $\Phi$-linear map $\varphi:M\to N$ between 
right or left finite semisimple\footnote{We only use semisimplicity here to ensure the existence 
of adjoints.} $\Phi$-module categories have a natural $\Phi$-linear structure. 
\end{lemma}

\begin{proof} Choosing left modules and the right adjoint, we write a
functorial isomorphism  
  $$  \mathrm{Hom}_M\left(m,\varphi^R(x. n)\right) =
     \mathrm{Hom}_M\left(m, x. \varphi^R(n)\right) $$
by  rewriting  the  left  side  as  \[  \mathrm{Hom}_N\left(\varphi(m),x.
n\right)     =    \mathrm{Hom}_N\left({}^*x.     \varphi(m),n\right)    =
\mathrm{Hom}_M\left(\varphi({}^*x.              m),n\right)             =
\mathrm{Hom}_M\left({}^*x. m, \varphi^R(n)\right) \] and finish by moving
$x$ back to the right. The other cases are similar.
\end{proof}

\begin{lemma}\label{sa}
If $N = \Phi$ above, then $\varphi\circ\varphi^L(\mathbf{1})$ is self-adjoint in $\Phi$.
\end{lemma}

\begin{proof}
Let $h(x,y):=\dim\mathrm{Hom}(x,y)$. In our semisimple case, $h(x,y) = h(y,x)$, as we are only 
counting multiplicities of simple objects. Moreover, $x$ is determined up to isomorphism by 
its multiplicities, so $x$ is self-adjoint iff $h(x,y) = h(x^*,y)$ for all $y$; the latter 
is also $h(x,y^*)$. We show this for 
$x=\varphi\circ\varphi^L(\mathbf{1})$:
\[
\begin{aligned}
h(x,y^*) &= h(y\cdot x,\mathbf{1}) = h(\varphi\circ\varphi^L(y),\mathbf{1}) = 
	h(\mathbf{1},\varphi\circ\varphi^L(y)) \\
	&= h(\varphi^L(\mathbf{1}),\varphi^L(y)) = h(\varphi^L(y),\varphi^L(\mathbf{1})) =
		h(y,\varphi\circ\varphi^L(\mathbf{1})) = h(x,y).
\end{aligned}
\]
  \vskip-2pc\qedhere\bigskip
  \renewcommand{\qedsymbol}{}
\end{proof}

\begin{lemma}\label{allproj}
Every self-adjoint projector $\varpi$ is isomorphic to a sum of distinct
$p_i$ from Lemma~\ref{project}, i.e., is a direct summand of the unit~1.
\end{lemma}
\begin{proof}
Let $\varpi = p +x$, where $p$ collects all the $p_i$ appearing in $\varpi$. Writing the relation 
$\varpi^2\cong \varpi$ as 
\[
p+x \cong p^2+p\cdot x +x\cdot p +x^2,
\] 
we see that each $p_i$ appears at most once, otherwise its multiplicity in $p^2$ exceeds the one 
in $p$. Moreover, an isomorphism $x\cong x^*$ gives an identification $\mathrm{Hom}(\mathbf{1},x^2) = 
\mathrm{End}(x)$, while $\mathrm{Hom}(\mathbf{1},x)=0$ by assumption; comparing left and right sides  
shows that $\mathrm{End}(x)=0$ and therefore $x=0$.
\end{proof}

\begin{proof}[Proof of Lemma~\ref{mult2}]
Writing $B$ for the $\Phi\text{-}\Phi$ bimodule category $C\boxtimes_{\Phi_{11}} R$, 
Lemma~\ref{mult} gives an equivalence $B\boxtimes_\Phi B \cong B$, which is $\mu $-compatible with the  identification 
$\Phi\boxtimes_\Phi \Phi =\Phi$. The left adjoint $\mu ^L$ is also a bimodule map, by Lemma~\ref
{moduleadjoint}, and because $\mu \boxtimes_\Phi \mu \cong \mu $ and
$\boxtimes_\Phi $ is composition of 1-morphisms in the 3-category~$\FC$, we
obtain an equivalence $\mu ^L\boxtimes_\Phi   \mu ^L \cong \mu ^L$. 
Then, $\mu \circ \mu ^L$ is an idempotent bimodule endomorphism of $\Phi$,
since $\circ$~and $\boxtimes_\Phi $ commute. It is 
the multiplication by the object $p:=\mu \circ \mu ^L(\mathbf{1})$ --- on the left, or on the right 
--- which must then be a projector in $\Phi$. Moreover, $p$ is self-adjoint by Lemma~\ref{sa}. 
Lemma~\ref{allproj} identifies it as a sum of $p_i$. If $p\not\cong \mathbf{1}$, it would split  
the image $\Phi\cdot p \cong  p\cdot\Phi$ as a block of $\Phi$, contradicting indecomposability.

It follows that $\mu \circ \mu ^L \cong  \mathrm{Id}_\Phi$,  splitting $B$ 
into $\Phi$ and a complementary bimodule. But the relation 
$B\boxtimes_\Phi B \cong B$ can only hold if this complement is zero, so $\mu $ is an 
equivalence.
\end{proof}

  \begin{proof}[Proof of Corollary~\ref{thm:34}]
 First, Morita equivalent fusion categories have braided tensor equivalent
Drinfeld centers~\cite[\S8.12]{EGNO}.  If $\Phi _0$~is a fusion category with
simple unit, then its Drinfeld center~$Z(\Phi _0)$ is
nondegenerate~\cite[\S8.20]{EGNO}.  Therefore, by~\cite{S-P,BJSS} the Drinfeld
center is invertible.  Conversely, by Theorem~\ref{thm:cr} any fusion
category is Morita equivalent to a finite direct sum of fusion categories
with simple unit.  Then, as in the proof of Lemma~\ref{thm:13}, the Drinfeld
center of a direct sum is the direct sum of the Drinfeld centers, and if the
Drinfeld center is invertible, it follows that the direct sum has a single
summand.
  \end{proof}

\providecommand{\bysame}{\leavevmode\hbox to3em{\hrulefill}\thinspace}
\providecommand{\MR}{\relax\ifhmode\unskip\space\fi MR }
\providecommand{\MRhref}[2]{%
  \href{http://www.ams.org/mathscinet-getitem?mr=#1}{#2}
}
\providecommand{\href}[2]{#2}

  \end{document}